\documentclass[11pt,a4paper]{article}

\usepackage[utf8]{inputenc}
\usepackage{subfiles} 
\usepackage{cite} 

\usepackage{multicol}
\usepackage{graphicx, wrapfig}
\usepackage{amsthm} 
\usepackage{amssymb, amsfonts} 
\usepackage[margin=1in]{geometry} 
\usepackage{tikz} 
\usepackage{tikz-cd} 
\usepackage{imakeidx} 
\usepackage{diagbox} 
\usepackage{caption} 
\usepackage{subcaption} 
\usepackage{afterpage} 
\usepackage{mathtools} 
\usepackage{titlesec} 
\usepackage{fancyhdr} 
\usepackage{siunitx} 
\usepackage{appendix}
\usepackage{enumitem}
\usepackage{tensor} 
\usepackage{url}
\usepackage{dutchcal} 
\usepackage{faktor} 
\usepackage{slashed}
\usepackage{yhmath} 
\usepackage{cases} 
\usepackage{hyperref} 

\theoremstyle{plain} 
\newtheorem{theorem}{Theorem}[section]

\newtheorem{corollary}[theorem]{Corollary}
\newtheorem{lemma}[theorem]{Lemma}
\newtheorem{proposition}[theorem]{Proposition}
\theoremstyle{definition} 
\newtheorem{definition}[theorem]{Definition}
\theoremstyle{remark}
\newtheorem{remark}[theorem]{Remark}

\numberwithin{equation}{section}
\numberwithin{table}{section}
\numberwithin{figure}{section}

\renewcommand{\Re}{\operatorname{Re}}
\renewcommand{\Im}{\operatorname{Im}}

\newcommand\norm[1]{\left\lVert#1\right\rVert}

\newcommand{\wedg}[1]{\mathbin{\ooalign{
			\rotatebox[origin=c]{0}{\scalebox{1}{$\wedge$}}\cr
			\hidewidth\kern0pt\raise-2pt\hbox{\rotatebox[origin=c]{0}{\scalebox{0.5}{$#1$}}}\hidewidth\cr
}}}

\def\C{\mathbb{C}}
\def\R{\mathbb{R}}

\def\Z{\mathbb{Z}}
\def\N{\mathbb{N}}

\def\U{\operatorname{U}}
\def\SU{\operatorname{SU}}
\def\SO{\operatorname{SO}}
\def\Sp{\operatorname{Sp}}
\def\SL{\operatorname{SL}}
\def\GL{\operatorname{GL}}

\def\Aut{\operatorname{Aut}}
\def\End{\operatorname{End}}

\def\ad{\operatorname{ad}}
\def\Lie{\operatorname{Lie}}

\def\Spin{\operatorname{Spin}}
\def\Hol{\operatorname{Hol}}
\def\Dirac{\slashed{D}}

\def\dual{\vee}
\def\iso{\cong}
\def\comp{\circ}

\def\interior{\lrcorner}
\def\invomega{\rotatebox[origin=c]{180}{$\omega$}}

\author{Nicolò Cavalleri}
\title{Complete non-compact $\Spin(7)$-manifolds from $T^2$-bundles over AC Calabi Yau manifolds}
\date{}

\makeindex

\begin{document}
\maketitle

\begin{abstract}
	We develop a new construction of complete non-compact 8-manifolds with Riemannian holonomy equal to $\Spin(7)$.
	As a consequence of the holonomy reduction, these manifolds are Ricci-flat.
	These metrics are built on the total spaces of principal $T^2$-bundles over asymptotically conical Calabi Yau manifolds, and the result is generalized to orbifolds.
	The resulting metrics have a new geometry at infinity that we call asymptotically $T^2$-fibred conical ($AT^2C$) and which generalizes to higher dimensions the ALG metrics of 4-dimensional hyperkähler geometry, analogously to how ALC metrics generalize ALF metrics.
	
	As an application of this construction, we produce infinitely many diffeomorphism types of $AT^2C$ $\Spin(7)$-manifolds and the first known examples of complete toric $\Spin(7)$-manifold.
\end{abstract}

\pagenumbering{roman}
\setcounter{tocdepth}{1}
\tableofcontents
\newpage
\pagenumbering{arabic}

\section{Introduction}\label{sec:intro}

In \cite{Foscolo2021}, Foscolo, Haskins and Nordström provide a construction of complete non-compact 7-manifolds with holonomy $G_2$ and non-maximal volume growth.
The input data for their construction is an asymptotically conical (AC) Calabi-Yau (CY) 6-manifold with a principal circle bundle satisfying a suitable topological condition.
In this paper we extend their construction to produce complete, non-compact 8-manifolds with holonomy $\Spin(7)$ and non-maximal volume growth, starting with an asymptotically conical Calabi-Yau 6-manifold and a principal $T^2$-bundle satisfying some topological condition.
The asymptotic volume growth of balls of radius $r$ in these new manifolds has leading term proportional to $r^6$ (in which case we say that the volume growth has \textit{rate} 6), whereas up to now we only had examples with volume growth proportional to $r^8$ (AC), $r^7$ (ALC) or $r$ (asymptotically cylindrical).
The idea of building $\Spin(7)$ structures on the total spaces of torus bundles over CY manifolds, suggested by Simon Donaldson, was first carried out in \cite{Fowdar2023}, where Fowdar constructed incomplete examples of $\Spin(7)$-manifolds with full holonomy.
In this paper we construct complete examples by adapting the general analytic approach of \cite{Foscolo2021} to solve the equations expressing the torsion-freeness of a $T^2$ invariant $\Spin(7)$-structure on a $T^2$-bundle over an asymptotically conical Calabi Yau manifold.
Unlike in \cite{Foscolo2021} however, we make use of the implicit function theorem as our main analytical tool and we explain how it can be adapted to the $G_2$ construction in \cite{Foscolo2021} to give a more concise proof.
Moreover, similarly to how the $G_2$ construction in \cite{Foscolo2021} was generalized to the orbifold setting in \cite{Foscolo2019}, so the result in this paper is generalized to the orbifold setting.
Finally, we use our main result to prove the existence of
\begin{enumerate}
	\item infinitely many diffeomorphism types of complete non-compact simply connected $\Spin(7)$-manifolds with volume growth rate $6$,
	\item the first known examples of complete toric $\Spin(7)$ manifolds in the sense of Madsen and Swann, first conjectured in \cite{BruunMadsen2019}.
	More specifically, we build two examples by making use of the smooth version of the construction, and we explain how to build infinitely many examples by making use of the orbifold version, should the results of \cite{Futaki2009} and \cite{Goto2009} extend to orbifolds.
\end{enumerate}
Infinitely many diffeomorphism types of complete non-compact simply connected $AT^1C$ $\Spin(7)$-manifolds were already built in \cite{Foscolo2019}.
Note moreover that the property $AT^1C$ is more commonly known as \textit{asymptotically locally conical} (ALC).
As for point 2, we believe the results of \cite{Futaki2009} and \cite{Goto2009} do extend to the orbifold case, but given that we do not have a complete proof of this, we prefer not to state the existence of infinitely many examples as a theorem.

Suppose we are given a principal $T^2$-bundle over an asymptotically conical Calabi Yau manifold $B$.
Analogously to the $G_2$ case, our construction gives, more precisely, a 1-parameter family of $T^2$-invariant $\Spin(7)$-metrics $\{g_\varepsilon\}_\varepsilon$ that collapses back to the CY metric on the manifold $B$ as $\varepsilon \to 0$.
The metric $g_\varepsilon$ has a geometry at infinity which is the higher dimensional analogue of the ALG metrics appearing in 4-dimensional hyperkähler geometry.
To put an end to the zoo of letters for ALX manifolds, we propose to call them \textit{asymptotically $T^2$-fibred conical} manifolds, abbreviated as $AT^2C$ manifolds.
More generally, we define $AT^kC$ manifolds to be manifolds that are asymptotically isometric (i.e.\ isometric up to decaying terms) to the total space $P$ of a $T^k$ principal bundle on a Riemannian cone with a metric that can be decomposed in two summands: the metric on the cone and a $T^k$-invariant flat metric on the fibers which is in some sense constant and which is lifted to the total space through a radially invariant principal connection on $P$.
By definition, for $k < n$, $AT^kC$ $n$-manifolds have volume growth rate $n - k$.

More precisely, our main result is the following.
\begin{theorem}\label{thm:main_theorem}
	Let $\left(B, g_B, \omega_0, \Omega_0\right)$ be an asymptotically conical Calabi Yau 3-fold asymptotic with rate $\nu \in \R^-$ to the Calabi-Yau cone $\left(\mathrm{C}(\Sigma), g_{\mathrm{C}}, \omega_{\mathrm{C}}, \Omega_{\mathrm{C}}\right)$ over a smooth Sasaki-Einstein 5-manifold $\Sigma$.
	Let $M \rightarrow B$ be a principal $T^2$ bundle such that $c_1(M) \in H^2\left(B; \Z^2\right)$ spans a 2-dimensional subspace in $H^2\left(B; \R\right)$ and
	\begin{equation}\label{eq:topological_condition}
		c_1(M) \smile \left[\omega_0\right] = 0 \in H^4\left(B; \R^2\right)
		.
	\end{equation}
	
	Then the 8-manifold $M$ carries an analytic curve of $T^2$-invariant torsion-free $\Spin(7)$-structures $\Phi : (0, \varepsilon_0) \to \Omega^4(M)$ such that, by denoting with $g_\varepsilon$ the Riemannian metric on $M$ induced by $\Phi_\varepsilon$,
	\begin{enumerate}
		\item $\Hol(M, g_\varepsilon) = \Spin(7)$.
		\item $(M, g_\varepsilon)$ is $AT^2C$ with rate $\max(-1, \nu)$.
		\item There exists a flat metric $g_P$ on the fibers of $M \to B$ such that
		$$
			g_\varepsilon - g_B - \varepsilon^2 g_P \to 0
		$$
		in $C^{k, \alpha}_{\text{loc}}$ for every $k \in \N$, $\alpha \in (0, 1)$.
		In particular, $\left(M, g_\varepsilon\right)$ collapses with bounded curvature to $\left(B, g_B\right)$ as $\varepsilon \rightarrow 0$.
	\end{enumerate}
\end{theorem}

The topological condition \eqref{eq:topological_condition} is both sufficient and necessary, and it is analogous to the condition that appeared in \cite{Foscolo2021}.
The condition that requires $c_1(M)$ span a 2-dimensional space is necessary to guarantee the holonomy group to be fully $\Spin(7)$ (otherwise it would be a subgroup).

$\Spin(7)$ manifolds are interesting firstly because, together with $G_2$ manifolds, they are an exceptional case in Berger's list of possible holonomies for simply connected, non-locally symmetric and irreducible Riemannian manifolds (see \cite{Berger1955}).
Moreover, they admit a parallel spinor, which is a property of great interest for supersymmetry in physics.
As a consequence of the parallel spinor, they are Ricci flat, which is a feature eagerly sought both in Riemannian geometry and in mathematical physics.
Moreover, $\Spin(7)$ geometry has deep connections with both $G_2$ and Calabi Yau geometry, and this paper highlights one of these connections.
Finally, $\Spin(7)$ manifolds have applications to F-theory in physics.
We expand more on this at the end of this section.

There is not much variety of examples of complete non-compact fully $\Spin(7)$ metrics around: the known examples can be found in \cite{Bryant1989, Cvetic2001, Cvetic2001a, Gukov2001, Bazaikin2007, Bazaikin2008, Kovalev2013, Foscolo2019, Lehmann2022}.
More specifically, in \cite{Bryant1989} Bryant and Salamon constructed the first example of complete non-compact fully $\Spin(7)$ metric as an explicit AC metric on the spinor bundle $\slashed{S}(S^4)$ of the 4-sphere.
Their metric is invariant under the natural cohomogeneity one action of $\Sp(2) \iso \Spin(5)$.
In \cite{Cvetic2001} and \cite{Cvetic2001a} Cvetič, Gibbons, Lü and Pope built a new example of a cohomogeneity one $\Sp(2)$-invariant $\Spin(7)$ metric, by solving the cohomogeneity one ODE.
In \cite{Bazaikin2008} it was rigorously proved that their example belonged to a 1-parameter family as claimed in \cite{Cvetic2001a}.
The examples of this 1-parameter family, labelled $\mathbb{B}_8$, are ALC (or $AT^1C$ according to our notation above).
The parameter of the $\mathbb{B}_8$ family is the asymptotic length $\varepsilon$ of the circle fibers.
As $\varepsilon \to 0$, the family collapses to the Bryant-Salamon $G_2$-metric on $\Lambda^-T^*S^4$.

In \cite{Lehmann2022}, Lehmann built yet two new other examples of one parameter families of complete non-compact ALC $\Spin(7)$-manifolds, that are both built through cohomogeneity-one methods with generic orbit isomorphic to the Aloff–Wallach space $N(1, -1)$, which differs from any other generic orbit of the previously built cohomogeneity one examples.
Both families of manifolds degenerate to AC manifolds at specific values of the family parameter.

In \cite{Kovalev2013}, Kovalev modified Joyce's construction of compact $\Spin(7)$ manifolds outlined in \cite{Joyce1999}, to produce the first known examples of asymptotically cylindrical $\Spin(7)$ manifolds.
These have a different type of volume growth (in fact linear, the lowest possible) than the maximal and sub-maximal growth of AC and ALC manifolds.
In this paper, we will exhibit yet another type of volume growth, distinguished from AC, ALC and asymptotically cylindrical.

Finally, in \cite{Foscolo2019}, Foscolo adapted the ideas of \cite{Foscolo2021} to construct ALC $\Spin(7)$ manifolds on total spaces of circle bundles on AC $G_2$-orbifolds.
In his construction, Foscolo uses the implicit function theorem in the same way we use it in our construction, i.e.\ to exploit the simplification of the equations in the adiabatic limit of the fibers shrinking to points (in his case circles, in our case tori).

The main idea of our proof is fairly simple.
Given an AC CY 3-fold $(B, \omega, \Omega)$ and a torus bundle $M \to B$ on it, with a principal connection specified by a pair $(\eta, \theta)$ of 1-forms on $M$, we write a candidate $\Spin(7)$ structure $\Phi$ in terms of $\omega, \Omega, \eta, \theta$ and three smooth functions $p, q, r$ (holding the data of a flat metric on the fibers), as $\Phi = F(\omega, \Omega, \eta, \theta, p, q, r)$.
The reason why $(\eta, \theta)$ specifies a principal connection is that we can always split a $T^2$-bundle $M$ as the product of two circle bundles $P_1 \times_B P_2$, and $\eta$ and $\theta$ specify principal connections on $P_1$ and $P_2$.
Now, since the holonomy of $M$ is contained in $\Spin(7)$ if and only if there is a parallel $\Spin(7)$-structure $\Phi$, and, by \cite{Fernandez1986} this is equivalent to $d\Phi =0$, by imposing that $\Phi$ be closed we are able to get a system of equations that are the object of our study.
Then, we introduce the 1-dimensional parameter $\varepsilon$, that is meant to implement the idea of the shrinking fibers.
To do so we make the substitutions $\eta \mapsto \varepsilon \eta$ and $\theta \mapsto \varepsilon \theta$ in our system.
Thus, together with the equations that impose $(\omega, \Omega)$ to be an $\SU(3)$-structure and some gauge fixing equations, we write the system succinctly as $\Psi(\varepsilon, \omega, \Omega, \eta, \theta, p, q, r) = 0$.
We then apply the implicit function theorem to the system $\Psi = 0$, where we split the domain of $\Psi$ as $\R \times Y$.
$\R$ is where $\varepsilon$ lives and $Y$ is where the other data $y = (\omega, \Omega, \eta, \theta, p, q, r)$ lives.
Thus, we treat the parameter $\varepsilon$ as the independent variable and we wish to find a curve $(\omega_\varepsilon, \Omega_\varepsilon, \eta_\varepsilon, \theta_\varepsilon, p_\varepsilon, q_\varepsilon, r_\varepsilon)$ of solutions of $\Psi = 0$.
The main steps of the proof via the implicit function theorem are to first study the equation $\Psi(0, y) = 0$ and to find a solution $y_0$ of such system.
Then one wants to study the linearization $D\Psi$ of $\Psi$ in the point $(0, y_0)$.
If one proves that such linearization restricted to $Y$ (i.e.\ the map $y \mapsto D \Psi_{\left(x_0, y_0\right)}(0, y)$) is invertible, the theorem yields the wanted curve of solutions.
We will find that the condition $\Psi(0, y) = 0$ imposes $(\omega_0, \Omega_0)$ to be Calabi Yau and $p_0, q_0, r_0$ to be constant.

As anticipated, in the final part of the paper we use our result to build new examples of complete non-compact $\Spin(7)$ manifolds.
Similarly to the $G_2$ case, this is possible thanks to the recent interest in AC Calabi-Yau manifolds, which brought many authors to build new examples and eventually to classify all AC CY manifolds (see \cite{Conlon2024}).
The topological condition needed for the $\Spin(7)$ construction, i.e.\ the existence of two linearly independent integral cohomology classes whose cup product with a fixed  Kähler class is zero, is more difficult to realize than the condition involving only one cohomology class needed in the $G_2$ case.
However, the zoo of AC Calabi-Yau 3-folds is enough populated already that we are still able to find manifolds (and orbifolds) that satisfy the needed condition.
In this paper, we show the realization of the topological condition and thus apply the new construction only in two specific cases that we considered of particular interest, but more examples can be build out of theorem \ref{thm:main_theorem}.

In order to build an infinite number of new diffeomorphism types of $\Spin(7)$ manifolds, we apply the theorem to the small resolutions of compound Du Val singularities.
This example was also used in \cite{Foscolo2021} to build an infinite number of new diffeomorphism types of complete non-compact $G_2$ manifolds.

After that, we build the first known examples of toric $\Spin(7)$ manifold.
Let us give the reader a brief explanation of what this means.
In symplectic geometry, one says that the action of a Lie group $G$ on a symplectic manifold $(M, \omega)$ is Hamiltonian if it is symplectic and it admits a moment map.
Let $\Lie(G) = \mathfrak{g}$.
A map $\mu : M \to \mathfrak{g}^*$ is called a moment map for the action of $G$ on $M$ if, for any $\xi \in \mathfrak{g}$,
$$
	d\langle\mu, \xi\rangle=\rho(\xi) \interior \omega
$$
where $\rho$ is the standard Lie algebra morphism $\mathfrak{g} \to \mathfrak{X}(M)$ induced by the action of $G$, and $\langle, \rangle$ denote the evaluation pairing.
In \cite{Madsen2011}, Madsen and Swann extend the notion of moment map to closed forms of arbitrarily high degree.
In the case of a $\Spin(7)$ 4-form $\Phi$, the action of an abelian Lie group $G$ is said to be \textit{multi-Hamiltonian} if it preserves $\Phi$ and it admits a \textit{multi-moment map}, which is defined to be a $G$-invariant map $\nu : M \to \Lambda^3 \mathfrak{g}^*$ such that, for every $\xi \in \Lambda^3\mathfrak{g}$,
$$
	d\langle\nu, \xi\rangle = \rho(\xi)\lrcorner \Phi
$$
where $\rho$ is the morphism $\Lambda^3 \mathfrak{g} \to \Gamma( \Lambda^3 TM)$ induced by $\rho : \mathfrak{g} \to \mathfrak{X}(M)$.
Their definition does not actually need the group to be abelian, but removing such hypothesis makes the definition more difficult, and we are only going to consider $G = T^k$ anyway.
Analogously to the symplectic case the maximal dimension of a torus that has an effective hamiltonian action on a $\Spin(7)$ manifold is 4.
This is because the action being multi-hamiltonian implies $\Phi|_{\Lambda^4W} = 0$, where $W$ is the subbundle given by $W_p = \Im \rho_p$.
Thus, we give the following
\begin{definition}
	A \textit{toric} Spin(7)-manifold is a torsion-free Spin(7)-manifold $(M, \Phi)$ with an effective multi-Hamiltonian action of $T^4$.
\end{definition}

Note that toric $\Spin(7)$ manifolds provide an example of non-singular Cayley fibration.

In \cite{BruunMadsen2019}, Madsen and Swann were able to produce an example of toric $\Spin(7)$-manifold with full holonomy, which is not, however, complete.
Until now no example of complete toric $\Spin(7)$-manifold with full holonomy was known, and this was one of the main motivations that lead to the research in this paper.
Indeed, the fact that our construction happens on the total space of a $T^2$ bundle suggests that we are only missing a $T^2$ action on the base manifold $B$ in order to reach the whole $T^4$-symmetry, and things turn out to be precisely so.
Note that compact $\Spin(7)$ manifolds with full holonomy do not admit continuous symmetries, so a toric $\Spin(7)$ manifold with full holonomy needs to be non-compact, and thus adding the hypothesis of completeness is the most we can ask for.
Indeed, it is a fact from Riemannian geometry that Killing fields on a compact Ricci flat manifold are parallel, but a manifold with full holonomy $\Spin(7)$ does non admit nonzero parallel vector fields, as the fundamental representation of $\Spin(7)$ is irreducible.
Note that the same is true for manifolds with full holonomy $G_2$.

\paragraph{Generalizations of theorem \ref{thm:main_theorem}}
In this paper we prove the main result for the construction on smooth asymptotically conical Calabi Yau manifolds, and then we explain a possible way to generalize the result to the orbifold setting.
This generalization is natural both mathematically and physically (as we see below), but at least on the mathematical side, the construction in the case of a smooth base already provides interesting new examples.
%

Another possible generalization that much is less straightforward and that may be the object of future work is not assuming the $T^2$ bundle $M$ to be principal and introducing singularities in the fibration.

\paragraph{Connections to physics}
As mentioned above, $\Spin(7)$ manifolds have applications in string theory, more specifically in M-theories with 3 macroscopic dimensions of spacetime ($d = 3$) and F-theories with 4 macroscopic dimensions of spacetime ($d = 4$).
There are five supersymmetric string theories (type I, IIA, IIB, HE, HO) that happen in 10 dimensions, whereas M-theory happens in 11 dimensions and formulations of F-theory have been proposed in both 11 and 12 dimensions.
Any of these theories predicts the vacua of the universe to be the product of a compact manifold and a non-compact Lorentzian spacetime.

There are many dualities between these theories, where by \textit{duality} we mean a somehow canonical bijective correspondence.
A notable duality, namely T-duality (supposed by the SYZ conjecture to be one of the main ingredients of mirror symmetry), exists between string theories of type IIA and string theories of type IIB.
Both of these theories are often studied on the product of $\R^{1, 3}$ with a Calabi Yau 3-fold.
In the case of a Calabi-Yau manifold which fibered by $S^1$, T-duality can be more concretely expressed by producing another Calabi-Yau manifold whose circle fibers have radius $\frac{1}{r}$ where $r$ is the radius of the fiber in the original manifold.

M-theory was first proposed as a theory unifying the five supersymmetric string theories, by reducing to each of them in some specific limit.
One limit of particular interest in the context of Calabi Yau and $G_2$ geometry is the weak coupling limit of M-theory to type IIA string theory.
M-theory is usually studied on the product of some Lorentzian spacetime and a compact $G_2$ manifold.
In case the $G_2$ manifold is the total space of a circle bundle over a Calabi Yau 3-fold and the $G_2$ structure is $S^1$-invariant, the weak coupling limit consists in collapsing the fibers to points (in the sense of Gromov-Hausdorff convergence), which is precisely the idea underlying the proof of the analogous statement of theorem \ref{thm:main_theorem} in the $G_2$ case proved in \cite{Foscolo2021}.
The $G_2$ manifolds built in \cite{Foscolo2021} are non-compact, which usually means that the field theory induced on the macroscopic dimensions does not contain gravity, but it does still make mathematical sense (besides providing physically interesting field theories) to consider M-theory on these manifolds and the work of Foscolo, Haskins and Nordström has provided a mathematically rigorous setting where to observe a well-defined occurrence of the weak coupling limit.

F-theory on $\Spin(7)$ manifolds was firstly proposed by Vafa in \cite{Vafa1996}.
The original hope was to have a 12 dimensional theory living on a product of $\R^{1, 3}$ and an 8-dimensional manifold with special holonomy.
However, this approach lead immediately to difficulties and another approach was taken.
This other approach relies on a specific duality, that we are about to describe, between certain $d=3$ M-theories and certain $d = 4$ F-theories.
The idea is to work on the M-theory side of the duality, which is a way more well-known setting, and then translate on the other side of the duality.

Indeed, M-theory can also happen on a $\Spin(7)$ manifold, but in this case the Minkowski space under consideration is $\R^{1,2}$, so that the total dimensions add up to $11 = 3 + 8$.
If the $\Spin(7)$ manifold is the total space of a circle bundle over a $G_2$ manifold and the $\Spin(7)$-structure is $S^1$-invariant, there is a weak coupling limit analogous to the one explained above that produces a type IIA theory happening on the product of $\R^{1,2}$ and a $G_2$ manifold.
If the $G_2$ manifold is in turn invariant under a circle action whose quotient is a Calabi Yau 3-fold (which implies that the original $\Spin(7)$ manifold is the total space of a $T^2$ bundle over a CY 3-fold, like in the case considered in this paper), we can use $T$-duality on the remaining circle to produce a type IIB string theory which, in the limit of the radius of the circle going to infinity, gives rise to an F-theory setup.
In this case, after the decompactification of the remaining circle, we are left with one extra non-compact space dimension, giving us a theory whose macroscopic dimensions are indeed 1 time dimension and 3 space dimensions, i.e.\ what we observe in our universe.
Note however, that the vacua of this theory is not in general $\R^{1,3}\times B_6$ where $B_6$ is a manifold with an $\SU(3)$ structure, because in general the $T^2$ bundle we start with will not be trivial in neither of the circles, and thus after this limit is performed, one of the macroscopic space dimensions will come from a possibly non-trivial line bundle over $B_6$.

This approach was taken in \cite{Bonetti2014a} and \cite{Bonetti2014}, where the $\Spin(7)$ manifolds under consideration are compact and they admit a non-free $\Z_2$ action on the fibers such that the quotient of the circle whose radius is meant to go at infinity is identified with a close interval, which allows carrying out calculations more in detail.

Since they admit a natural isometric $T^2$ action, the manifolds built in this paper could also provide a possible interesting setting where to study F-theory via the above described duality although, as we mentioned, their non-compactness prevents the induced field theories to contain gravity.
Yet there has been extensive research in the literature on superstring theories and M-theories on non-compact manifolds (see e.g.\ \cite{Morrison1997} and \cite{Intriligator1997}), which have led to interesting field theories, and more recently also F theory has been studied successfully in a non-compact setting in \cite{Heckman2014}, yielding a conjectural classification of superconformal field theories in 6D.
Usually, in the non-compact case, a necessary condition to yield non-trivial theories (i.e.\ where the fields have interactions), is the presence of $D$-branes.
In the case of $\Spin(7)$-manifolds  with a $T^2$ isometric action, the presence of $D7$-branes is equivalent to the $T^2$ action not being free, which is not the case considered in this paper, but as discussed above it is one of the natural generalizations of theorem \ref{thm:main_theorem}.

One of the reasons why considering F-theory on the newly found $\Spin(7)$ (or more precisely on their version with singular base), is that it could shed light on the tricky question of how supersymmetry behaves under the duality between M-theory and F-theory in 11 dimensions explained above.
Indeed, this question has been a puzzle since F-theory on $\Spin(7)$ manifolds was proposed by Vafa in his original paper \cite{Vafa1996}, where it was hypothesized that this had to do with Witten's proposed resolution of the cosmological constant problem (see \cite{Witten1995a}).
The problem is that $\Spin(7)$ manifolds have 1 parallel spinor which would imply the existence of 2 supercharges on the M-theory side of the picture and supercharges are conserved under the duality.
However, on the F-theory side, which has $d = 4$, the number of supercharges has to be a multiple of 4.
In \cite{Bonetti2014}, the proposed explanation was that this has to do with the mentioned involution, which brings down the theory on a type IIB theory where one of the macroscopic dimension is a closed interval, and the discrepancy outlined above has to do with the boundary conditions of the equations.
This problem has been extensively discussed in \cite{Heckman2019}, where in addition 12 dimensional F-theory compactified on $\Spin(7)$ manifolds was studied.
It is not clear what happens in the case of manifolds like the ones found in this paper, where such involution is not there, and the question remains open to further investigation.

\paragraph{Organization} The rest of this paper is organized as follows.
In section \ref{sec:preliminaries} we recall some preliminary material on $SU(3), G_2$ and $\Spin(7)$ structures, and on their respective torsion-free versions.
We also get into the details about expressing connections on $T^2$-bundles efficiently.
As often happens, this section is meant more as a reference rather than to be read.

In section \ref{sec:asymptotical} we recall some analytic results on AC Calabi-Yau manifolds, that we will use in the following sections.
The analysis in this section, as much as in the rest of the paper, heavily relies on the work in \cite{Foscolo2021}.

In section \ref{sec:equations} we write down rigorously the equations that a $\Spin(7)$-structure expressed in terms of the data $(\omega, \Omega, \eta, \theta, p, q, r)$ needs to satisfy in order to be torsion free.
We also rewrite the equations more abstractly in a way that could be useful for other applications.


Section \ref{sec:analytic_curve} is the heart of this paper and is dedicated to the setup of the implicit function theorem in our setting and to the solution of the equations.

In section \ref{sec:properties}, we prove the key properties of the newly-found $\Spin(7)$ manifolds that make them interesting, such as the fact that they have full holonomy and that they are $AT^2C$.

In section \ref{sec:orbifolds} we explain how to generalize the result to the orbifold setting.

In section \ref{sec:examples}, we use the existence theorem proved in section \ref{sec:analytic_curve} to build the mentioned new examples of new $\Spin(7)$-manifolds, including the first known complete toric $\Spin(7)$-manifolds.

Finally, in section \ref{sec:G2_case} we explain how to use the implicit function theorem to give a more concise proof of the analogous result for $G_2$ manifolds proved in \cite{Foscolo2021}.

\paragraph{Acknowledgments} This work was carried out during the author's PhD and was supported by the Engineering and Physical Sciences Research Council and the EPSRC Centre for Doctoral Training in Geometry and Number Theory (The London School of Geometry and Number Theory), and thus the author would like to express extreme gratitude to both.
Moreover, the author would like to thank his PhD colleagues in the special holonomy community for helpful discussions, with a mention to Jakob Stein and Viktor Majewski for carefully reviewing the paper, and a special thanks goes to Federico Bonetti for his help with the physics section in the introduction.

Finally, the author would like to give the biggest thanks his supervisor, Lorenzo Foscolo, for suggesting the problem\footnote{The original idea of building $\Spin(7)$-structures on $T^2$-bundles over CY 3-folds was firstly suggested by Simon Donaldson during a conference as an analogue of \cite{Foscolo2021}.} and for his constant support.


\section{Preliminaries}\label{sec:preliminaries}



%

The holonomy reduction of a Riemannian manifold to any of the groups in Berger's list ($\U(n)$, $\SU(n)$, $\Sp(n)$, $\Sp(n)\cdot\Sp(1)$, $G_2$, $\Spin(7)$) can be expressed by the existence of some parallel tensors that are pointwise reductible to some model tensors on $\R^n$ whose stabilizers are precisely the groups listed above.
Since $\SU(n)$ is the stabilizer of $(\omega_0, \Re \Omega_0)$, where $\omega_0$ is the standard symplectic form of $\R^n$ and $\Omega_0$ is the standard complex volume form, an $\SU(3)$-structure can be specified by two differential forms $\omega$ and $\Re \Omega$, that satisfy the following conditions:
\begin{enumerate}
	\item $\omega$ is non-degenerate and $\Re \Omega$ is stable in the sense of \cite{Hitchin2000}.
	\item $\omega \wedge \Re \Omega = 0$.
	\item The complex \textit{Monge-Ampère} equation
	\begin{equation}\label{eq:Monge_Ampere}
		\frac{1}{6} \omega^3=\frac{1}{4} \operatorname{Re} \Omega \wedge \operatorname{Im} \Omega
		.
	\end{equation}
\end{enumerate}
Note that the reason why we only need to stabilize $\Re \Omega$ and we can forget about $\Im \Omega$ is explained below in theorem \ref{thm:Hitchin}.
A $\Spin(7)$-structure on an 8-manifold can be specified by a 4-form $\Phi$ which is pointwise reductible to $\Phi_0$, the standard $\Spin(7)$-structure on $\R^8$.
One of the many possible ways to express $\Phi_0$, which highlights an important connection of $\SU(3)$ and $\Spin(7)$ geometries, is the following:
\begin{equation}\label{eq:standard_Cayley}
\Phi_0 = e^7 \wedge e^8 \wedge \omega_0 + e^7 \wedge \Re \Omega_0 - e^8 \wedge \Im \Omega_0 + \frac{1}{2} \omega_0^2
\end{equation}
where $\{e^j\}_{j=1}^8$ is the standard dual basis of $\R^8$.

In order for these structures to induce a holonomy reduction, they need to be torsion free.
This condition is equivalent to asking that the differential forms introduced above be parallel, which is in turn equivalent to asking that
\begin{enumerate}
	\item $d\omega = 0$ and $d \Re \Omega = d \Im \Omega = 0$ for $\SU(3)$-structures.
	\item $d \Phi = 0$ for $\Spin(7)$ structures (which implies also $d{\star} \Phi = 0$, since $\Phi = {\star} \Phi$).
	See \cite{Fernandez1986}.
\end{enumerate}

\subsection{$\SU(3)$-structures}

We collect here some results on $\SU(3)$-structures.
Recall that since $\SU(3) \leq \GL_3(\C)$, an $\SU(3)$-structure $(\omega, \Re \Omega)$ induces an almost complex structure that we denote by $J$.

\begin{lemma}\label{thm:SU3_identities}
	Let $V$ be a vector space, let $X \in V$, and let $(\omega, \Re \Omega)$ be an $\SU(3)$-structure on $V$.
	Then
	\begin{multicols}{2}
	\begin{enumerate}
		\item $X\lrcorner \omega=-J X^{\flat}$
		\item $X\lrcorner \operatorname{Im} \Omega=-(J X)\lrcorner \operatorname{Re} \Omega$
		\item $(X\lrcorner \operatorname{Re} \Omega) \wedge \omega=J X^{\flat} \wedge \operatorname{Re} \Omega=X^{\flat} \wedge \operatorname{Im} \Omega$
		\item $(X\lrcorner \operatorname{Re} \Omega) \wedge \operatorname{Re} \Omega=X^{\flat} \wedge \omega^{2}$
		\item $\left(X\lrcorner \operatorname{Re} \Omega\right) \wedge \operatorname{Im} \Omega=-J X^{\flat} \wedge \omega^{2}$
		.
		\item[\vspace{\fill}]
	\end{enumerate}
	\end{multicols}
	It follows that the above equalities hold by interpreting $X$ as a vector field on a manifold $M$ with an $\SU(3)$ structure $(\Re \Omega, \omega)$.
\end{lemma}

In the next proposition we make use of the decomposition of the real $\SU(3)$ representation $\Lambda^*\R^6$ into irreducible real  representations.

\begin{lemma}\label{thm:irreducible_decomposition_SU3}
	We have the following orthogonal decompositions into irreducible $\SU(3)$ representations:
	$$
	\Lambda^{2} \left(\mathbb{R}^{6}\right)^*=\Lambda_{1}^{2} \oplus \Lambda_{6}^{2} \oplus \Lambda_{8}^{2}
	$$
	where
	$$
	\Lambda_{1}^{2}=\mathbb{R} \omega_0 \quad \Lambda_{6}^{2}=\left\{X\lrcorner \operatorname{Re} \Omega_0 \mid X \in \mathbb{R}^{6}\right\} \quad
	\Lambda_{8}^{2} = \left\{ \eta \in \Lambda^2 \mid \eta \wedge \omega_0^2 = 0, \eta \wedge \Re \Omega_0 = 0 \right\}
	.
	$$
	Moreover,
	$$
	\Lambda^{3} \left(\mathbb{R}^{6}\right)^*=\Lambda_{6}^{3} \oplus \Lambda_{1 \oplus 1}^{3} \oplus \Lambda_{12}^{3}
	,
	$$
	where
	$$
	\Lambda_{6}^{3}=\left\{X \wedge \omega_0 \mid X \in \mathbb{R}^{6}\right\} \qquad
	\Lambda_{1 \oplus 1}^{3}=\mathbb{R} \operatorname{Re} \Omega_0 \oplus \mathbb{R} \operatorname{Im} \Omega_0
	$$
	$$
	\Lambda_{12}^{3}=\left\{\gamma \in \Lambda^3 \mid \gamma \wedge \omega_0 = 0, \gamma \wedge \Re \Omega_0 = 0, \gamma \wedge \Im \Omega_0 = 0\right\}
	.
	$$
	The subscript $i$ in the notation $\Lambda^k_i$ is the dimension of the irreducible component $\Lambda^k_i$.
\end{lemma}

We only went through the 2-forms and 3-forms because 0, 1, 5, and 6-forms are irreducible representations and 4-forms are determined by 2-forms by Hodge duality.
Thus, we define:
$$
\Lambda^4_{1} := \star \Lambda^2_{1} \quad \Lambda^4_{6} := \star \Lambda^2_{6}  \quad \Lambda^4_{8} := \star \Lambda^2_{8} 
$$

We can exploit lemma \ref{thm:irreducible_decomposition_SU3} to calculate explicitly the Hodge star operator on manifolds with an $\SU(3)$-structure.

\begin{lemma}\label{thm:Hodge_star_SU3}
	The Hodge star operator satisfies the following identities:
	\begin{multicols}{2}
	\begin{enumerate}
		\item ${\star} \eta =-\frac{1}{2} J \eta \wedge \omega^{2}$
		\item ${\star} \omega=\frac{1}{2} \omega^{2}$
		\item ${\star}(X\lrcorner \operatorname{Re} \Omega)=-J X^{\flat} \wedge \operatorname{Re} \Omega=X^{\flat} \wedge \operatorname{Im} \Omega$
		\item ${\star}\left(\tau_{8} \wedge \omega\right)=-\tau_{8}$ for all $\tau_{8} \in \Omega_{8}^{2}$
		\item ${\star} \sigma_{12} = - J \sigma_{12}$ for all $\sigma_{12} \in \Omega^3_{12}$
		\item ${\star} \left(\eta \wedge \omega\right) = -(J \eta) \wedge \omega = -J (\eta \wedge \omega)$
		\item ${\star} \operatorname{Re} \Omega=\operatorname{Im} \Omega$ and ${\star} \operatorname{Im} \Omega=-\operatorname{Re} \Omega$
		\item[\vspace{\fill}]
	\end{enumerate}
	\end{multicols}
\end{lemma}

\noindent The following lemma comes in useful in calculations.

\begin{lemma}\label{thm:d_star_J_d}
	Let $B$ a $2n$-manifold with a $U(n)$-structure on it $(\omega, J)$ such that $d \omega = 0$.
	Then, for any $h \in C^\infty(B)$,
	$$
	d^*Jdh = 0
	.
	$$
	Note that this expression is clearly equivalent to $d{\star} Jdh=0$.
\end{lemma}
\begin{proof}
	This is because $d{\star} Jdh = \frac{1}{2} d (dh \wedge \omega^2) = 0$.
\end{proof}

Another consequence of lemma \ref{thm:irreducible_decomposition_SU3} is the following explicit expression of the torsion of an $\SU(3)$-structure.

\begin{proposition}\label{thm:SU3_torsion}
	Let $(\omega, \Omega)$ be an $\SU(3)$-structure on $B$. Then there exist functions $w_1, \hat{w}_1$, primitive $(1,1)$-forms $w_2, \hat{w}_2$, a 3 -form $w_3 \in \Omega_{12}^3(B)$ and 1-forms $w_4, w_5$ on $B$ such that
	$$
	\begin{aligned}
		&d \omega=3 w_1 \operatorname{Re} \Omega+3 \hat{w}_1 \operatorname{Im} \Omega+w_3+w_4 \wedge \omega \\
		&d \operatorname{Re} \Omega=2 \hat{w}_1 \omega^2+w_5 \wedge \operatorname{Re} \Omega+w_2 \wedge \omega \\
		&d \operatorname{Im} \Omega=-2 w_1 \omega^2+w_5 \wedge \operatorname{Im} \Omega+\hat{w}_2 \wedge \omega
	\end{aligned}
	$$
\end{proposition}

%
%

The reason why we can forget about specifying $\Im \Omega$ when talking about an $\SU(3)$-structure, is that $\Im \Omega$ is uniquely determined by $\Re \Omega$.
This is because of the following theorem, due to Hitchin in \cite{Hitchin2000}.

\begin{theorem}\label{thm:Hitchin}
	Let $V$ be an oriented 6-dimensional real vector space and let $U \subseteq \Lambda^3 V^*$ be defined by
	$$
		U = \{\Re \Omega \mid \Omega \text{ is a complex volume form on } V \text{ for some complex structure }J \}
		.
	$$
	Then $U$ is open and there exists two unique $\GL^+(V)$-equivariant analytic maps
	$$
		\begin{aligned}
			U & \to \End(V) \\
			\psi & \mapsto J_\psi
		\end{aligned}
		\qquad
		\begin{aligned}
			U & \to U \\
			\psi & \mapsto \hat{\psi}
		\end{aligned}
	$$
	such that $J_\psi$ is a complex structure on $V$ and $\psi + i \hat{\psi}$ is a $(3, 0)$-form for $J_\psi$.
	The latter map is called the \textit{Hitchin map}.
\end{theorem}

\noindent In other words, $\Re \Omega$ uniquely determines $J$ and $\Im \Omega$.

Finally, lemma \ref{thm:irreducible_decomposition_SU3} also allows us to give an explicit expression for the derivative of the Hitchin map.

\begin{proposition}\label{thm:Hitchin_map_derivative}
	Given an $\mathrm{SU}(3)$-structure $(\omega, \Omega)$ on $B$, let $\rho \in \Omega^3(B)$ be a form with small enough $C^0$-norm so that $\operatorname{Re} \Omega+\rho$ is still a stable form. Decomposing into types we write $\rho=$ $\rho_6+\rho_{1 \oplus 1}+\rho_{12}$. Then the image $\hat{\rho}$ of $\rho$ under the linearization of Hitchin's duality map at $\operatorname{Re} \Omega$ is
	$$
	\hat{\rho}=\star\left(\rho_6+\rho_{1 \oplus 1}\right)-\star \rho_{12}
	$$
\end{proposition}
\begin{proof}
	This follows from lemma 3.3 in \cite{Moroianu2008}.
\end{proof}

\subsection{Calabi-Yau manifolds}

In the following lemma we collect some identities that hold on a manifold with a torsion-free $\SU(3)$-structure, i.e.\ a Calabi-Yau (CY) 3-fold.

\begin{definition}
	On a CY 3-fold, we define, for $\gamma \in \Omega^1(M)$,
	$$
		\operatorname{curl} \gamma=\star(d \gamma \wedge \operatorname{Re} \Omega)
	$$
\end{definition}

\begin{lemma}\label{thm:exterior_differential_SU3_decomposition}
	Let $M$ be a smooth dimensional manifold with a Calabi-Yau structure $(\Re \Omega, \omega)$.
	Then, for every $f \in C^{\infty}(M), \gamma \in \Omega^{1}(M)$ and vector field $X$ on $M$ we have
	\begin{enumerate}
		\item $d(f \omega)=d f \wedge \omega$ and $d^{*}(f \omega)=-\star d\left(\frac{1}{2} f \omega^{2}\right)=J d f$
		\item $d\left(f \omega^{2}\right)=d f \wedge \omega^{2}$ and $d^{*}\left(f \omega^{2}\right)=2 J d f \wedge \omega$
		\item $d^{*} \gamma=\star\left(d J \gamma \wedge \frac{1}{2} \omega^{2}\right)$
		\item $d \gamma=-\frac{1}{3} d^{*}(J \gamma) \omega+\frac{1}{2}(J \operatorname{curl} \gamma)^{\sharp} \lrcorner \operatorname{Re} \Omega+\pi_{8}(d \gamma)$
		\item $d \left(X\lrcorner \operatorname{Re} \Omega\right)=\frac{1}{2} \operatorname{curl} X^{\flat} \wedge \omega-\frac{1}{2}\left(d^{*} X^{\flat}\right) \operatorname{Re} \Omega+\frac{1}{2} d^{*}\left(J X^{\flat}\right) \operatorname{Im} \Omega+\pi_{12}\left(d(X\lrcorner \operatorname{Re} \Omega)\right)$
		\item $d^{*}(X\lrcorner \operatorname{Re} \Omega)=J \operatorname{curl} X^{\flat}$
	\end{enumerate}
	where $\operatorname{curl} \gamma=\star(d \gamma \wedge \operatorname{Re} \Omega)$.
\end{lemma}
\begin{proof}
	1 and 2 follow from lemma \ref{thm:Hodge_star_SU3}.2 and lemma \ref{thm:SU3_identities}.1 and that $d \omega = 0$.
	3 follows from \ref{thm:Hodge_star_SU3}.1 and that $d \omega^2 = 0$.
	4 follows from 3, the definition of $\operatorname{curl}$ and a couple calculations.
	As for 5, we write $d \left(X\lrcorner \operatorname{Re} \Omega\right)$ in components:
	$d \left(X\lrcorner \operatorname{Re} \Omega\right) = \lambda \Re \Omega + \mu \Im \Omega + \eta \wedge \omega + \pi_{12}\left(d(X\lrcorner \operatorname{Re} \Omega)\right)$.
	To find $\lambda$ we take the wedge product with $\Im \Omega$ on both sides, we use that $d \Im \Omega = 0$, lemma \ref{thm:SU3_identities}.5 and part 3 from this lemma.
	Similarly for $\mu$.
	As for $\eta$, we take the wedge product with $\omega$ on both sides, we use that $d \omega = 0$, lemma \ref{thm:SU3_identities}.3 and the definition of $\operatorname{curl}$.
	Finally, for 6,
	$$
	d^{*}(X\lrcorner \operatorname{Re} \Omega)=-* d\left(X^{b} \wedge \operatorname{Im} \Omega\right)=-*\left(d X^{b} \wedge \operatorname{Im} \Omega\right)=J \operatorname{curl} X^{b}
	.
	$$
\end{proof}

\begin{lemma}\label{thm:pi1_rho}
	For any $\rho \in \Omega^3_{12}$, $\pi_1(d \rho) = 0$.
\end{lemma}
\begin{proof}
	This is because $d\rho \wedge \omega = d (\rho \wedge \omega) = 0$.
\end{proof}

\begin{lemma}\label{thm:Omega_3_12}
	For every vector field $X$ on $B$ we have
	$$
	\star d(X\lrcorner \operatorname{Re} \Omega)-d(X\lrcorner \operatorname{Im} \Omega) \in \Omega_{12}^3, \quad \star d\left(X\lrcorner \operatorname{Re} \Omega\right)+d\left(X\lrcorner \operatorname{Im} \Omega\right) \in \Omega_{1 \oplus 1}^3 \oplus \Omega_6^3 .
	$$
\end{lemma}
\begin{proof}
	See lemma 2.16 in \cite{Foscolo2021}.
\end{proof}

\begin{corollary}\label{thm:pi6_d_rho}
	There exists an endomorphism $\sigma$ of $\Omega^3_{12}$ such that for every $\rho \in \Omega^3_{12}$, $\pi_6(d \star \rho) = \sigma (\pi_6(d\rho))$.
	In particular, if $\pi_6(d\rho)$ vanishes, $\pi_6(d \star \rho)$ also does.
\end{corollary}
\begin{proof}
	Since the decomposition of lemma \ref{thm:irreducible_decomposition_SU3} is orthogonal, by lemma \ref{thm:Omega_3_12}, for any compactly supported $X\in \mathfrak{X}(B)$ and for any $\rho \in \Omega^3_{12}(B)$, $\langle \rho,  {\star} d\left(X\lrcorner \operatorname{Re} \Omega\right)+d\left(X\lrcorner \operatorname{Im} \Omega\right) \rangle_{L^2} = 0$.
	A calculation shows that this is equivalent to $\langle d \rho, JX^\flat \wedge \Re \Omega \rangle_{L^2} = \langle d {\star} \rho, X^\flat \wedge \Re \Omega \rangle_{L^2}$.
	Hence, by defining $\sigma$ pointwise to be the morphism $X^\flat \wedge \Re \Omega \mapsto -J X^\flat \wedge \Re \Omega$ (which is an isometry), we see that $\pi_6(d \star \rho) = \sigma \left(\pi_6(d \rho)\right)$.
\end{proof}

\begin{lemma}\label{thm:pi1_6}
	Let $(B, \omega, \Omega)$ be a Calabi-Yau 3-fold. Then for every function $g$
	$$
	\pi_1 \left(d^* d\left(g \omega\right)\right)=\frac{2}{3}(\triangle g) \omega
	$$
	and the operator $(f, \gamma) \mapsto \pi_{1 \oplus 6} d^* d\left(f \omega+\gamma^{\sharp}\interior \operatorname{Re} \Omega\right)$ can be identified with
	$$
	(f, \gamma) \longmapsto\left(\frac{2}{3} \triangle f, d d^* \gamma+\frac{2}{3} d^* d \gamma\right)
	.
	$$
\end{lemma}
\begin{proof}
	See lemma 2.19 of \cite{Foscolo2021}.
	The first statement is the Hodge dual of the one in \cite{Foscolo2021}.
\end{proof}

\subsection{$\Spin(7)$ structures}

We will make use of the following characterization for manifolds with full holonomy $\Spin(7)$.

\begin{lemma}\label{thm:Bryant}
	Let $M$ be connected and simply connected.
	Let $g$ be a metric on $M$ with $\Hol(M, g) \leq \operatorname{Spin}(7)$.
	Suppose that there are no non-zero $g$-parallel 1-forms or parallel non-degenerate 2-forms on $M$ compatible with $g$ (in the sense that the almost complex structure induced through $g$ by the non-degenerate 2-form is a pointwise isometry for $g$).
	
	Then $\Hol(M, g) = \operatorname{Spin}(7)$.
\end{lemma}
\begin{proof}
	This is just lemma 2 from \cite{Bryant1987} with the additional remark that the 2-form preserved by $\SU(4)$ is Kähler, thus it is non-degenerate and it is compatible with $g$.
\end{proof}

\subsection{Spin geometry}

Since $\SU(3) \leq \SU(4) \iso \Spin(6)$ is the stabilizer of a nonzero vector in $\C^4$, an $\SU(3)$-structure is equivalently expressed as a spin structure together with a non-vanishing spinor.
In particular every manifold with an $\SU(3)$ structure is spin.
By analyzing $\SU(3)$ irreducible representations, we see that the real spinor bundle $\slashed{S}(B)$ must be isomorphic to $\underline{\R}\oplus\underline{\R} \oplus T^* B$.
The isomorphism is given by
$$
(f, g, \gamma) \longmapsto f \psi+g \mathrm{Vol} \cdot \psi+\gamma \cdot \psi
$$
where $\cdot$ is Clifford multiplication.

The following lemma describes the Dirac operator $\Dirac$ in terms of this isomorphism.

\begin{lemma}\label{thm:Dirac_operator_expression}
	Let $(M, \omega, \Omega)$ be a Calabi-Yau 3-fold. Then for every $f, g \in C^{\infty}(M)$ and $\gamma \in \Omega^1(M)$
	$$
	\Dirac(f, g, \gamma)=\left(d^* \gamma,-d^* J \gamma, \operatorname{curl} \gamma+d f-J d g\right) \text {. }
	$$
	In particular if $f=0=g$ then $\gamma$ is in the kernel of $\Dirac$ if and only if $d^* \gamma=0$ and $d \gamma \in \Omega^2_8(M)$.
\end{lemma}
\begin{proof}
	See lemma 2.22 in \cite{Foscolo2021}.
\end{proof}

\subsection{Connections on $T^2$-bundles}\label{sec:T2_connections}

We begin with the following standard result.
\begin{proposition}\label{thm:product_bundle}
	Suppose $P \to M$ is a principal $G \times H$ bundle.
	Then there are a principal $G$ bundle $P_G \rightarrow M$ and a principal $H$-bundle $P_H \rightarrow M$ with the property that $P$ is isomorphic to the product bundle $P_G \times_M P_H$.
\end{proposition}

\noindent Let us describe explicitly connections on a 2-torus bundle.

\begin{proposition}\label{thm:T2_connections}
	Let $\pi: M \to B$ be a $T^2$-bundle over a smooth manifold $B$.
	
	Every identification between the Lie algebra of $T^2$ and $\R^2$ induces a bijection between principal connections for $\pi$ and couples of $T^2$-invariant 1-forms $(\eta, \theta)$ on $M$ such that
	$$
		\eta(X) = 1, \eta(Y) = 0
	\qquad
		\theta(X) = 0, \theta(Y) = 1
	$$
	where $X, Y$ are the vector fields associated to the elements of the Lie algebra of $T^2$ that correspond to the elements of the standard basis of $\R^2$.
	
	The identification is non-canonical in the sense that it does depend on the identification of the Lie algebra of $T^2$ with $\R^2$.
\end{proposition}
\begin{proof}
	Under the given identification
	$\Omega^1(M; \operatorname{Lie}(T_2)) \iso \Omega^1(M; \R^2) \iso \Omega^1(M) \oplus \Omega^1(M)$.
\end{proof}

\begin{remark}\label{rmk:XY_commute}
	Note that the vector fields $X$ and $Y$ commute since $T^2$ is commutative.
\end{remark}

We identify $\Lie(T^2)$ and $\R^2$ in such a way that $\eta$ and $\theta$ can be interpreted as the pullback of two connections on two circle bundles whose product is the given $T^2$ bundle.
In other words, we want to choose as a basis of $\Lie(T^2)$ two elements that exponentiate (via the exponential map from $\Lie(T^2)$ to $T^2$) to give circle actions.
Not all elements of $\Lie(T^2)$ have this property, since it is well known that there are also elements that exponentiate to give faithful $\R$-action (and all elements of $\Lie(T^2)$ are of one of the two kinds).
Identifications of $\Lie(T^2)$ with $\R^2$ of the type described are in correspondence with identifications of  $T^2$ with $S^1\times S^1$ provided that we also fix an identification of $\Lie(S^1)$ and $\R$.
%
We choose a preferred (standard) such identification, which is the one that exponentiates $2\pi$ to the identity on $S^1$. 
Thus, there are $\GL_2(\Z)$ many identifications of $\Lie(T^2)$ and $\R^2$ that respect our choices, giving us $\GL_2(\Z)$ many ways to split a given $T^2$ bundle into the product of two circle bundles.
For any of this way any connection on the $T^2$ bundle will split into two connections on the two circle bundles, which will be different depending on the splitting.
%
%

For the rest of this paper, we fix an identification $\Lie(T^2) \iso \R^2$ satisfying the above described properties.
Hence, given a $T^2$-bundle, we will have two preferred fundamental vector fields $X$ and $Y$, which correspond to the two elements of the standard basis of $\R^2$.


\section{Asymptotic geometries}\label{sec:asymptotical}

In this section we state some preliminary results on asymptotically conical Calabi Yau 3-folds and we introduce the concept of asymptotically $T^k$ fibered conical manifold.
Before this though, let us quickly review the theory of Calabi-Yau cones.
A more in-depth discussion can be found in chapter 4 of \cite{Foscolo2021}.

\subsection{Calabi Yau cones}

Let $(\Sigma, g_\Sigma)$ be a Riemannian manifold.
A \textit{Calabi Yau cone} is a Riemannian cone $(C = C(\Sigma), g_C = dr^2 + r^2 g_\Sigma)$ which is also a Calabi Yau manifold.
Let us now restrict to the case $\dim_\R C = 6$ and thus $\dim_\R \Sigma = 5$.
It turns out that the condition that $C$ possesses a Calabi Yau structure is equivalent to the condition that $\Sigma$ possesses an $\Sp(1) = \SU(2)$-structure with prescribed torsion.
Such a manifold $\Sigma$ is called a Sasaki-Einstein manifold.
The $\SU(2)$-structure is equivalently given by a 4-tuple $(\eta, \omega_1, \omega_2, \omega_3) \in \Omega^1(\Sigma) \times \Omega^2(\Sigma)^3$ that satisfies the following conditions: $\eta$ is nowhere vanishing and $\omega_1, \omega_2, \omega_3$ restricted to $\ker \eta$ span a 3-dimensional subspace for $\Lambda^2(\ker \eta)^*$ at each point and they are an oriented basis for the space they span (where the orientation is the one induced by the natural orientation on the space of self-dual 2-forms of a 4-dimensional vector space).

The $\SU(2)$-structure $(\eta, \omega_1, \omega_2, \omega_3) \in \Omega^1(\Sigma) \times \Omega^2(\Sigma)^3$ will be called \textit{Sasaki-Einstein} if the $\SU(3)$-structure $(\eta, \omega_1, \omega_2, \omega_3)$ satisfies
$$
d \eta=2 \omega_1, \quad d \omega_2=-3 \eta \wedge \omega_3, \quad d \omega_3=3 \eta \wedge \omega_2
$$
which is equivalent to ask that the $\SU(3)$ structure defined on $C(\Sigma)$ by
$$
\omega_{\mathrm{C}}=r d r \wedge \eta+r^2 \omega_1, \quad \Omega_{\mathrm{C}}=r^2(d r+i r \eta) \wedge\left(\omega_2+i \omega_3\right)
$$
be Calabi Yau.
Note that since $\omega_1$ is exact, it is in particular closed.

\subsection{Asymptotically Conical Calabi Yau 3-folds}

In this section we collect some result on asymptotically conical CY 3-folds that we will use in the rest of the paper.
All these results are taken from \cite{Foscolo2021}, \cite{Conlon2012} or \cite{Conlon2024}, which we refer to for the proofs.
We begin with the definitions.

\begin{definition}
	Let $(C(\Sigma), g_C = dr^2 + r^2 g_\Sigma)$ be a Riemannian cone and let $P$ be a principal $T^k$ bundle on it, which is the pullback of a principal $T^k$-bundle $\tilde{P}$ on $\Sigma$.
	Suppose further $P$ has a metric $g_P$ which is the pullback of a $T^k$-invariant metric $g_{\tilde{P}}$ on $\tilde{P}$.
	We call $r$ also the pullback of the radial function $r$ on $P$.
	
	If $E$ is a metric vector bundle on $P$ with metric connection $\nabla$ splitting as $\nabla = \nabla^C + \nabla^T$ (where $\nabla^C$ and $\nabla^T$ are respectively the base and fiber components of $\nabla$), $\nu \in \R$ and $m \in \N$, $C^m_{\nu}(P, E)$ is defined to be the subset of $C^m(P, E)$ of those sections $u$ that converge under the norm
	$$
	\norm{u}_{C^m_\nu} = \sum_{j=0}^m r^{-\nu} \left(\norm{r^j\left(\nabla^C\right)^ju}_{C^0} + \norm{\left(\nabla^T\right)^ju}_{C^0}\right)
	.
	$$
	$C^\infty_\nu$ is defined as $\bigcap_m C^m_\nu$.
	
	Moreover, let $\alpha \in (0,1)$, and define the $C^{m, \alpha}_\nu$ norm to be
	$$
		\norm{u}_{C^{m, \alpha}_\nu} = \norm{u}_{C^m_\nu} + \left[r^{-\nu+m+ \alpha} \nabla^k u\right]_\alpha
		,
	$$
	where $[\cdot]_\alpha$ is the Hölder seminorm defined using parallel transport of $\nabla$ to identify fibers of $E$ along minimizing geodesics in a strictly convex neighborhood at each point.
\end{definition}
Note that the above definition reduces to the standard definition on the cone in the case $k = 0$, i.e.\ when the fibers are just points.

\begin{definition}\label{def:AC}
	Let $(B, g)$ be a complete Riemannian manifold and let $(C(\Sigma), g_C = dr^2 + r^2 g_\Sigma)$ be a Riemannian cone.
	
	We say that $B$ is \textit{asymptotically conical} (AC) of rate $\mu < 0$ with tangent cone $C(\Sigma)$ if, for some $R \in \R^+$ and $K \subseteq M$ compact, there exists a diffeomorphism $\Psi : C_R(\Sigma) \to B \setminus K$, such that
	$$
		\Psi^*g - g_C \in C^\infty_\mu(C(\Sigma), g_C)
		.
	$$
\end{definition}

In our definition, we are assuming that $B$ only has one topological end because we are interested in the Calabi Yau case and via the Cheeger-Gromoll splitting theorem we know that asymptotically conical manifolds with non-negative Ricci curvature only have one end.

For an asymptotically conical manifold $B$ we define $C^k_\nu(B)$ to be the set of those functions $f$ such that $\Psi_*f \in C^k_\nu(C(\Sigma))$, and we define $C^\infty_\nu$ analogously.

The following lemma allows us to reduce to simply connected manifolds in many proofs.

\begin{lemma}\label{thm:finite_fundamental_group}
	Let $B$ be an $A C$ Calabi-Yau 3-fold.
	Then $B$ has finite fundamental group.
	
	Moreover, its fundamental cover $\tilde{B}$ is also AC Calabi-Yau.
\end{lemma}
\begin{proof}
	See proposition 1.6 of \cite{Conlon2024}.
\end{proof}

\begin{definition}\label{def:Wknu}
	We denote $\mathcal{H}^k_\nu (B) = \ker (d + d^*) \subseteq C^\infty_\nu(\Lambda^k(B))$.
\end{definition}

Another important property of AC CY manifolds is that many decaying harmonic differential forms vanish, as explained in the following lemma.
For a vector bundle $E$ with a metric, denote by $C^\infty_\nu(E)$ the sections of $E$ vanishing with rate $\nu \in \R$.

\begin{lemma}\label{thm:harmonic_1_6_forms}
	Let $B$ be an irreducible AC CY 3-fold whose Sasaki-Einstein link $\Sigma$ has nonconstant sectional curvature.
	Let $\Lambda^k_l(B)$ be an irreducible $\SU(3)$ component of $\Lambda^k(B)$ of dimension $l$, if it exists.
	For $\nu < 0$, there are no harmonic forms in $C^\infty_\nu(\Lambda^k_l(B))$ for $l=1, 6$, where $\Lambda^k_l(B)$ is defined in lemma \ref{thm:irreducible_decomposition_SU3}.
	
	In particular decaying harmonic functions and 1-forms vanish.
\end{lemma}
\begin{proof}
	See lemma 5.6 in \cite{Foscolo2021}.
\end{proof}

Another useful result about harmonic forms is the following.

\begin{theorem}\label{thm:harmonic_2_forms_isomorphism}
	Let $B$ be an AC CY 3-fold whose Sasaki-Einstein link $\Sigma$ has nonconstant sectional curvature and let $\nu \in (-2, 0)$.
	Then, the natural map
	$$
	\begin{aligned}
		\mathcal{H}_\nu^2(B) & \rightarrow H^2(B) \\
		\alpha & \mapsto [\alpha]
	\end{aligned}
	$$
	is an isomorphism.
\end{theorem}
\begin{proof}
	See theorem 5.12.ii in \cite{Foscolo2021}.
\end{proof}

\begin{lemma}\label{thm:exact_star_H4}
	Let $(B, \omega, \Omega)$ be an AC CY 3-fold.
	Leth $\kappa \in \mathcal{H}^2_{-2}$ be such that $[\star_B \kappa]_{H^4} = 0$.
	Then $\kappa$ is $L^2$-orthogonal to $L^2\mathcal{H}^2(B)$.

	Note that $[\star_B \kappa]_{H^4} = 0$ is equivalent to $[\kappa] \smile [\omega] = 0$.
\end{lemma}
\begin{proof}
	See the proof of 6.3 in \cite{Foscolo2021}.
\end{proof}

\begin{lemma}\label{thm:exact_2_forms_AC_CY}
	Let $\left(B, \omega_0, \Omega_0\right)$ be an AC CY 3-fold whose Sasaki-Einstein link $\Sigma$ has nonconstant sectional curvature.
	Let $\sigma \in C_\nu^{k, \alpha}(\Lambda^2(B))$ for some $k \geq 1, \alpha \in(0,1)$ and $\nu \in(-5,-1)$, and suppose $\sigma$ exact.
	
	Then there exists a unique coclosed $\gamma \in C_{\nu+1}^{k+1, \alpha}(\Lambda^1(B))$ such that $\sigma=d \gamma$.
\end{lemma}
\begin{proof}
	See lemma 7.1 in \cite{Foscolo2021}.
\end{proof}

A crucial step in the study of the linearization of the equations that we will examine, is to identify the presence of the Dirac operator.
We will need the following theorem to apply the implicit function theorem to our equations \eqref{eq:Spin7_torsion_v1}.

\begin{theorem}\label{thm:Dirac_iso}
	Let $\left(B, \omega_0, \Omega_0\right)$ be an irreducible AC Calabi-Yau 3-fold whose Sasaki-Einstein link $\Sigma$ has nonconstant sectional curvature.
	The Dirac operator $\Dirac: C_{\nu+1}^{k+1, \alpha} \rightarrow C_\nu^{k, \alpha}$ is an isomorphism for all $\nu \in(-6,-1)$.
\end{theorem}
\begin{proof}
	See proposition 5.8 in \cite{Foscolo2021}.
\end{proof}

The following proposition and its corollary are needed for the manipulation of 4-forms on AC Calabi Yau manifolds.
They are very technical and are mainly used in the proof that the linearization of the equations is an isomorphism.
To state them, we first need to define a subspace of $C^{0, \alpha}_\nu$ that is handy to get uniqueness in the statements.
To this purpose, we remark that, because of proposition B.12 of \cite{Foscolo2021}, $\mathcal{H}^k_\nu(B) \subseteq L^2_{\nu - \delta}$ for some small $\delta \in \R^+$.
Since by the AC Sobolev embedding theorem, $C^{0, \alpha}_\nu(B) \subseteq L^2_{\nu + \delta}$, it makes sense to consider the $L^2_\nu$ inner product between elements of $\mathcal{H}^k_\nu(B)$ and $C^{0, \alpha}_\nu(B)$.

\begin{definition}
	Let $\alpha \in (0,1)$ and $\nu \in \R$ outside the set of indicial roots of $d + d^*$.
	
	We define $\mathcal{W}^k_\nu(B)$ to be the intersection of the kernels of the linear maps given by the inner product with elements in $\mathcal{H}^k_\nu(B)$.
\end{definition}

\begin{proposition}\label{thm:exact_4_forms}
	Let $\left(B, \omega_0, \Omega_0\right)$ be a non-trivial $A C$ Calabi-Yau 3-fold asymptotic to the Calabi-Yau cone $\mathrm{C}(\Sigma)$ whose Sasaki-Einstein link $\Sigma$ has nonconstant sectional curvature.
	Fix $k \geq 1, \alpha \in(0,1), \delta>0$ and $\nu \in(-3-\delta,-1)$ away from a discrete set of indicial roots.
	Then every exact $4$-form $\sigma=d \rho^{\prime}$ with $\rho^{\prime} \in C_\nu^{k, \alpha}$ can be written uniquely as
	$$
	\sigma=d * d\left(f \omega_0+\gamma^{\sharp}\lrcorner \operatorname{Re} \Omega_0\right)+d \rho_0
	$$
	where $f, \gamma \in C_{\nu+1}^{k+1, \alpha}, \rho_0 \in C_\nu^{k, \alpha} \cap \mathcal{W}_\nu^3 \cap \Omega_{12}^3$ with $d^* \rho_0=0$ and
	$$
	\|(f, \gamma)\|_{C_{\nu+1}^{k+1, \alpha}}+\left\|\rho_0\right\|_{C_\nu^{k, \alpha}} \leq C\left\|\rho^{\prime}\right\|_{C_\nu^{k, \alpha}}
	$$
	for a constant $C>0$ independent of $\rho^{\prime}$.
	Moreover, $f=0=\gamma$ if $\sigma \in \Omega_8^4$.
\end{proposition}
\begin{proof}
	See proposition 5.18 in \cite{Foscolo2021}.
\end{proof}
By nontrivial AC CY manifold, we mean one which is not $\C^3$ with the standard CY structure.

\begin{corollary}\label{thm:exact_4_forms'}
	In the notation of proposition \ref{thm:exact_4_forms}, fix $\nu \in(-3-\delta,-1)$ away from a discrete set of indicial roots.
	For every exact 4-form $\sigma=d \rho^{\prime}$ with $\rho^{\prime} \in C_\nu^{k, \alpha}$ there exist unique $\rho_0 \in C_\nu^{k, \alpha} \cap \mathcal{W}_\nu^3 \cap \Omega_{12}^3$ with $d^* \rho_0=0$ and $f, \gamma \in C_{\nu+1}^{k+1, \alpha}$ such that
	$$
	\sigma=d * d\left(f \omega_0\right)+d\left(* d\left(\gamma^{\sharp}\lrcorner \operatorname{Re} \Omega_0\right)-d\left(\gamma^{\sharp}\lrcorner \operatorname{Im} \Omega_0\right)\right)+d \rho_0 .
	$$
	Moreover,
	$$
	* d\left(\gamma^{\sharp}\lrcorner \operatorname{Re} \Omega_0\right)-d\left(\gamma^{\sharp}\lrcorner \operatorname{Im} \Omega_0\right)+\rho_0 \in \Omega_{12}^3 .
	$$
\end{corollary}
\begin{proof}
	See corollary 5.19 in \cite{Foscolo2021}.
\end{proof}

In the rest of the paper, we will always assume that the AC CY 3-folds we are considering have a Sasaki-Einstein link $\Sigma$ with nonconstant sectional curvature, but we will see in lemma \ref{thm:additional_topological_constraint} that this is not a restrictive assumption for theorem \ref{thm:main_theorem}.

\subsection{Asymptotically $T^k$-fibred conical manifolds}

Intuitively, an asymptotically $T^k$-fibred conical manifold is meant to be a manifold $M$ which, at infinity, looks like the total space of a fiber bundle over a cone $C$ whose fibers are flat tori.
This means, more formally, that outside of a compact set $M$ looks like the total space of a Riemannian submersion whose fibers are flat tori and the base is a Riemannian cone.
Moreover, we ask that in the limit at infinity, these toric fibers converge to a fixed flat torus.


\begin{definition}\label{def:ATkC}
	Let $(L, g_L)$ be a Riemannian manifold and let $(C(L), g_C)$ be the Riemannian cone on $L$.
	We use the following notation: for $R \in \R^+$,
	$$
	C_R(L) := \{x \in C(L) \mid r(x) > R\}
	$$
	where $r$ is the radial coordinate of the cone.
	Let $T^k \iso (S^1)^k$ be the $k$-torus.
	
	A Riemannian manifold $(M, g_M)$ is said to be \textit{asymptotically $T^k$-fibred conical} ($AT^kC$) with rate $\nu < 0$ if for some $R \in \R^+$ and compact $K \subseteq M$ there exists a $T^k$ fiber bundle $\pi_P : P_{T^k} \to C_R(L)$ with a radially invariant metric on the fibers $g_{T^k} \in \operatorname{Sym}^2(VP)$ such that each fiber $(\pi_P^{-1}(b), g_{T_k, b})$ is a flat torus, a radially invariant connection $A_\infty$ on $P_{T^k}$ and a diffeomorphism $\Psi : P_{T^k} \to M \setminus K$ such that
	$$
	\Psi^* g_M - g_P \in C^\infty_{\nu}\left(P_{T^k}, g_P\right) \qquad g_P = \pi_{P}^* g_C \oplus \operatorname{E}_{A_\infty}(g_{T^k})
	$$
	where $\operatorname{E}_{A_\infty}(g_{T^k}) \in S^2(TM)$ is the horizontal extension of $g_{T^k}$ with respect to $A_\infty$, i.e.\ defined to be $g_{T^k}$ on the vertical space and 0 on the horizontal space of $A_\infty$.
	We will call the Riemannian manifold $(P_{T^k}, g_P)$ the \textit{asymptotic model} the $AT^kC$ manifold.
\end{definition}
Similarly, we could define \textit{asymptotically $T^k$-fibred} for other types of ends.

\begin{remark}
	In the case where $P_{T^k}$ is a principal bundle, the vertical bundle is trivial. 
	In this case, the condition that the metric on the fibers $g_{T^k}$ be radially invariant is automatically satisfied if $g_{T^k}$ is constant with respect to the canonical trivialization $V P_{T^k} \iso C \times \Lie\left(T^k\right)$, which is a very convoluted way of saying what follows.
	A family $g_{T^k}$ of metrics on the fibers of $P_{T^k}$ is by definition a metric on the vertical bundle $V P_{T^k}$.
	The vertical bundle is canonically trivialized by the principal action as $V P_{T^k} \iso C \times \Lie\left(T^k\right)$, and thus a family of metrics on the fibers is just a function $C \to S^2 \left( \Lie\left(T^k\right)\right)$.
	Hence, a constant metric is just the choice of a metric on $\Lie\left(T^k\right) \iso \R^k$.
	The word \textit{flat} is redundant as being constant implies being flat.
	Note moreover that being constant also implies being $T^k$-invariant.
	
	In the case of the manifolds built in this paper it is this latter stronger condition that is satisfied, although the weaker condition of being radially invariant required in \ref{def:ATkC} is the direct generalization of ALG metric from 4-dimensional hyperkähler geometry.
	
	We also want to stress that a metric on the fibers alone is not enough to complete the pullback of the cone metric to a metric on $P_{T^k}$, because with this information alone we have no way to figure out what should be the angle between a vertical and a non-vertical vector.
	That is why the connection $A_\infty$ is necessary in the above definition.
	In light of the above remark, if we regard a connection $A_\infty$ as a 1-form with values in $\Lie(T^k)$, and we choose a metric $g_{\Lie(T^k)} \in S^2 \left( \Lie\left(T^k\right)\right)$, then $\operatorname{E}_{A_\infty}\left(g_{\Lie\left(T^k\right)}\right) = g_{\Lie\left(T^k\right)} \comp (A_\infty \times A_\infty)$.
\end{remark}

\section{The torsion-free condition for \texorpdfstring{$\Spin(7)$}{Spin(7)} structures on \texorpdfstring{$T^2$}{T\^2} bundles}\label{sec:equations}

In this section we are going to write down the equations for the $\Spin(7)$-holonomy reduction of the total space of a $T^2$-bundle over a Calabi Yau manifold.

We begin by studying what are the possible $T^2$-invariant $\Spin(7)$ 4-forms $\Phi$.

\begin{lemma}\label{thm:T2_invariant_Spin7}
	Let $B$ be a 6-manifold and $\pi: M \to B$ be a principal $T^2$-bundle.
	Then any $T^2$-invariant $\Spin(7)$-structure on $M$ will be of the form
	\begin{equation}\label{eq:Philambdamu}
		\Phi = \lambda \wedge \mu \wedge \omega + \lambda \wedge \Re \Omega - \mu \wedge \Im \Omega + \frac{1}{2} \omega^2
	\end{equation}
	for some $T^2$-invariant 1-forms $\lambda, \mu$ such that $\lambda(W) = \mu(Z) = 1, \lambda(Z) = \mu(W) = 0$ on two linearly independent vertical $T^2$-invariant vector fields $W, Z$ and an $\SU(3)$-structure $(\omega, \Re \Omega)$ on $B$ (that we identify again with its pullback through $\pi$).
\end{lemma}
\begin{proof}
	It is straightforward to check that $\varphi$ is a $G_2$-structure.
	To do this we need to find a basis of $T_xM$ for each $x \in M$ such that $p_x^*\Phi = \Phi_0$, where $p_x: \R^8 \to T_xM$ is the isomorphism induced by the given basis.
	Since $(\omega, \Re \Omega)$ is an $\SU(3)$-structure, there exists a basis of $T_b B$ for each $b \in B$ that reduces them to $(\omega_0, \Re \Omega_0)$.
	Choosing $W_x, Z_x$ and the horizontal lift of such basis at $T_{\pi(x)} B$ does the job, since on $\R^8$ expression \eqref{eq:standard_Cayley} holds. 
	
	Conversely, we can recover $(\lambda, \mu, \omega, \Re \Omega)$ from $\Phi$.
	To see this, note first that $\Phi$ induces a metric $g_\Phi$.
	Then we can choose two $T^2$-invariant orthonormal vertical vector fields $W$ and $Z$, for example by applying Gram-Schmidt to any two vertical $T^2$-invariant linearly independent vector fields.
	Then, we define $\lambda = W \interior g_\Phi$, $\mu = Z \interior g_\Phi$, $\omega = Z \interior W \interior \Phi$ and $\Re \Omega = W \interior \Phi - \mu \wedge \omega$.
	
	We need to verify that equation \eqref{eq:Philambdamu} holds and that $(\omega, \Re \Omega)$ is an $\SU(3)$-structure.
	Thus, the rest of the proof is just linear algebra and we omit it.
\end{proof}

\begin{remark}\label{rmk:rotation_Spin7}
	Note that the information contained in $\lambda, \mu, \omega$ and $\Re \Omega$ is redundant.
	Indeed, we can easily find a non-trivial action of the group of sections of the trivial principal bundle $B \times U(1)$ on the above data that leaves the $\Spin(7)$-structure $\Phi$ invariant.
	
	Let us consider a specific point $x$ in $M$.
	Then, if we consider the action of $U(1)$ on the above data evaluated at $x$ described by
	$$
		e^{i \theta} \cdot (\lambda, \mu, \omega, \Re \Omega) = (\cos \theta \lambda - \sin \theta \mu, \sin \theta \lambda + \cos \theta \mu, \omega, \Re e^{i\theta} \Omega)
		,
	$$
	it is immediate to check that this does indeed leave $\Phi$ invariant at $x$.
	
	It is easy to see how this is globalized: we just apply the above formula to each point (the only difficulty is that globally there might not exist a function $\theta$ which induces the section, but this is not a problem, since we only need $e^{i \theta} \in U(1)$).
\end{remark}

\begin{lemma}\label{thm:metric_induced_by_invariant_Spin7}
	The metric induced by $\Phi$ is $g_\Phi = \lambda^2 + \mu^2 + g_{\omega, \Omega}$.
\end{lemma}
\begin{proof}
	This follows directly from the fact that the stabilizer of $\Phi_0$ on $\R^8$ (which is indeed isomorphic to $\Spin(7)$), stabilizes also the standard metric on $\R^8$, which is equivalently expressed as $g_{\omega_0, \Omega_0} + (e^7)^2 + (e^8)^2$.
\end{proof}

\begin{lemma}\label{thm:Spin7_structure}
	Let $B$ be a 6-manifold and $\pi: M \to B$ be a principal $T^2$-bundle.
	Then, any $T^2$-invariant $\Spin(7)$-structure on $M$ will be of the form
	\begin{equation}\label{eq:Phi_eta_theta}
	\Phi = \eta \wedge \theta \wedge \omega + \eta \wedge p \Re \Omega - \theta \wedge (r \Re \Omega + q \Im \Omega) + \frac{1}{2}pq \omega^2
	\end{equation}
	for some principal connection $(\eta, \theta)$ for $\pi$, the pullback through $\pi$ of an $\SU(3)$-structure $(\omega, \Re \Omega)$ on $B$ and some smooth functions $p, q, r$ on $B$ of which both $p$ and $q$ are strictly positive.
	
	Moreover, by fixing a basis of $\Lie(T^2)$ and considering the vector fields $X, Y$ corresponding to such basis through the $T^2$-action on $M$, the data $(\eta, \theta, \omega, \Omega, p, q, r)$ is uniquely determined by requiring $\eta(X) = \theta(Y) = 1$ and $\eta(Y) = \theta(X) = 0$.
\end{lemma}
\begin{proof}
	$\lambda$ and $\mu$ appearing in equation \eqref{eq:Philambdamu} are almost the data of a connection on the $T^2$ bundle, in the sense that the intersection of their kernels is the horizontal bundle of some connection.
	What is missing, however, is that in general the two vector fields $W$ and $Z$ will be very far from being fundamental vector fields for the $T^2$ action, since this property is a much stronger request than just being $T^2$-invariant and vertical.
	
	Let $X$ and $Y$ be the two generators (we say \textit{the} because we have fixed an identification of $\operatorname{Lie}(T^2)$ with $\R^2$).
	Then we consider the unique connection $(\eta, \theta) \in \Lambda^1(M)^2$ whose horizontal space is the kernel of $\lambda$ and $\mu$ and that satisfies $\eta(X) = \theta(Y) = 1$ and $\eta(Y) = \theta(X) = 0$.
	In general, we can write $W$ and $Z$ as generic linear combinations of $X$ and $Y$, but here we can use some of the redundancy we have on the data $(\lambda, \mu, \omega, \Re \Omega)$ to simplify the situation.
	
	Indeed, we can use the rotation introduced in remark \ref{rmk:rotation_Spin7} to make the vector $W$ parallel to $X$ (which in terms of $\lambda$ and $\mu$ makes $\mu$ ``parallel'' to $\theta$).
	This is possible because $W$ and $X$ are both $T^2$-invariant.
	
	Hence, we write
	$$
		\begin{cases}
			\lambda = g \eta - l  \theta \\
			\mu = f \theta
		\end{cases}
		\qquad \begin{bmatrix}
			g(x) & -l(x) \\ 0 & f(x)
		\end{bmatrix}
		\in \GL_2(\R) \, \forall x \in M
	$$
	where $g, l, f$ are  a priori smooth scalar functions on $M$.
	Since $\lambda, \mu, \eta, \theta$ are $T^2$-invariant, we can actually think $g, l, f$ as elements of $C^\infty(B)$.
	Moreover, the condition that they make up a change of basis at every point is equivalent to $fg \neq 0$.
	By the above redundancy, we can assume $f$ to be positive (otherwise we can rotate an additional angle $\pi$ above).
	We can also assume $g$ to be positive, because if it were negative we can perform the following transformation
	$$
	(\lambda, \mu, \omega, \Re \Omega) = \left(\lambda, -\mu, -\omega, \Re \overline{\Omega}\right)
	$$
	which flips the orientation on $B$ but leaves $\Phi$ invariant.
	This transformation flips the sign of $Z$ and hence of both $f$ and $g$.
	
	By plugging in this expression into \eqref{eq:Philambdamu} we get
	$$
		\Phi = \eta \wedge \theta \wedge (gf \omega) + (g \eta - l \theta) \wedge \Re \Omega - f \theta \wedge \Im \Omega + \frac{1}{2} \omega^2
		.
	$$
	In general, requiring $d\Phi = 0$ will imply neither $d\omega = 0$ nor $d \Omega = 0$, but we can perform one last transformation on the $\SU(3)$-structure $(\omega, \Re \Omega)$ such that the new structure will have the property that $d\Phi = 0$ implies $d \omega = 0$.
	This transformation is a conformal rescaling by a factor of $\frac{1}{gf}$, which means that $\omega \mapsto \frac{1}{gf} \omega$ and $\Omega \mapsto \frac{1}{(\sqrt{gf})^{3}} \Omega$.
	In terms of the transformed $\SU(3)$-structure, the $\Spin(7)$-structure is expressed by
	$$
		\Phi = \eta \wedge \theta \wedge \omega + \eta \wedge p \Re \Omega - \theta \wedge (r \Re \Omega + q \Im \Omega) + \frac{1}{2}pq \omega^2
	$$
	where $p = g^{-\frac{1}{2}}f^{-\frac{3}{2}}$, $q = f^{-\frac{1}{2}}g^{-\frac{3}{2}}$ and $r=l g^{-\frac{3}{2}}f^{-\frac{3}{2}}$ (or equivalently $f = q^{\frac{1}{4}}p^{-\frac{3}{4}}$, $g = p^{\frac{1}{4}}q^{-\frac{3}{4}}$ and $l = r (pq)^{\frac{3}{4}}$).
	By the properties of $f$ and $g$, $p$ and $q$ are strictly positive.
	
	Finally, let us prove uniqueness.
	By lemma \ref{thm:metric_induced_by_invariant_Spin7} and the expressions we just calculated, we see that
	\begin{equation}\label{eq:metric_Spin7}
		g_\Phi = p^{\frac{1}{2}}q^{-\frac{3}{2}} \eta^2 + \left(r^2(pq)^\frac{3}{2}+ q^{\frac{1}{2}}p^{-\frac{3}{2}} \right) \theta^2 - 2 r p \eta \odot \theta + (pq)^{\frac{1}{2}} g_{\omega, \Omega}
		.
	\end{equation}
	In particular, $\Phi$ being in the form expressed in equation \eqref{eq:Phi_eta_theta} implies that the horizontal space for the connection $(\eta, \theta)$ is the $g_\Phi$-orthogonal complement of the vertical space.
	Indeed, from the expression of $g_\Phi$ just calculated, it is clear that if $v$ is a vertical vector and $w$ is such that $\eta(w) = \theta(w) = 0$, then $g_\Phi(v, w) = 0$, since $v \interior g_{\omega, \Omega} = 0$.
	Thus, $(\eta, \theta)$ are uniquely determined by the horizontal space and the requirement $\eta(X) = \theta(Y) = 1$ and $\eta(Y) = \theta(X) = 0$.
	
	Moreover, $\omega$ is uniquely defined as $Y \interior X \interior \Phi$.
	Suppose now that we can express $\Phi$ also as
	$$
		\Phi = \eta \wedge \theta \wedge \omega + \eta \wedge \tilde{p} \Re \tilde{\Omega} - \theta \wedge (\tilde{r} \Re \tilde{\Omega} + \tilde{q} \Im \tilde{\Omega}) + \frac{1}{2}\tilde{p}\tilde{q} \omega^2
		.
	$$
	By equating this to expression \eqref{eq:Phi_eta_theta} and taking the interior product with $X$ and $Y$ respectively, we get the following equations:
	$$
	\begin{cases}
		p \Re \Omega = \tilde{p} \Re \tilde{\Omega} \\
		r \Re \Omega + q \Im \Omega = \tilde{r} \Re \tilde{\Omega} + \tilde{q} \Im \tilde{\Omega} \\
		p q = \tilde{p} \tilde{q}
		.
	\end{cases}
	$$
	Then, by using the fact that the Hitchin map commutes with scalar multiplication and that $\Re \Omega$ and $\Im \Omega$ are linearly independent, we are able to conclude that $p = \tilde{p}$, $q = \tilde{q}$ and $r = \tilde{r}$.
	Thus, from the first equation, $\Re \Omega = \Re \tilde{\Omega}$, which concludes the proof.
\end{proof}

Note that $r = 0$ corresponds to the simpler case where $X$ and $Y$ are orthogonal, i.e.\ the case studied by Fowdar in \cite{Fowdar2023}.
In other words, if $r = 0$ is equivalent to assuming the toric fibers to be rectangular tori, which are tori obtained by identifying the edges of a rectangle.
The presence of a non-necessary null $r$, or equivalently considering $T^2$ bundles whose fibers have a more general flat metric instead of a rectangular metric, will reveal crucial in solving $d\Phi = 0$.

\begin{remark}
	As the reader might guess, the functions $p$, $q$ and $r$ have a concrete geometrical interpretation.	
	Indeed, the data of $p$, $q$ and $r$ can be equivalently expressed by the data of a smooth function from $B$ to $\operatorname{In}(\R^2)$ (where $\operatorname{In}$ denotes the space of positive definite bilinear forms), which is in turn equivalent to the data of a $T^2$-invariant metric on the fibers of the bundle $M \to B$.
	Indeed, in the proof of lemma \ref{thm:Spin7_structure}, we saw that the data of $p, q$ and $r$ is equivalent to the data of $f, g$ and $l$ which has the more immediate geometric meaning of describing the transformation law of $\lambda$ and $\mu$ into $\eta$ and $\theta$.
	$\lambda$ and $\mu$ are an orthonormal basis for vertical one forms, where the word \textit{vertical} has meaning in reference to the connection $(\eta, \theta)$ and the word \textit{orthonormal} is meant with respect to the dual metric $g_\Phi^\dual$ of the metric $g_\Phi$ restricted to the fibers.
	Thanks to lemma \ref{thm:metric_induced_by_invariant_Spin7} it is immediate to check that
	$$
		g_\Phi^\dual = \begin{bmatrix}
			g^{-2}(1+l^2f^{-2}) & l g^{-1}f^{-2} \\ l g^{-1}f^{-2} & f^{-2}
		\end{bmatrix}
		= \begin{bmatrix}
			p^{-\frac{1}{2}}q^{\frac{3}{2}}(1+p^3q) & rp^2q \\ rp^2q & q^{-\frac{1}{2}}p^{\frac{3}{2}}
		\end{bmatrix}
	$$
	where the matrix should be interpreted with respect to the basis of vertical one forms given by $\eta$ and $\theta$.
	It is also immediate to check that one can also express $f, g$ and $l$ in terms of the coefficients of $g_\Phi^\dual$.
	
	Finally note as a corollary that $p$ and $q$ can also be used to specify the information of the induced volume form on the toric fibers.
	Indeed, $\lambda \wedge \mu = (gf) \eta \wedge \theta = (pq)^{-\frac{1}{2}} \eta \wedge \theta$.
\end{remark}

In the rest of this chapter, we reinterpret the torsion-free condition $d\Phi = 0$ in terms of the data $(\omega, \Re \Omega, p, q, r)$ and we exploit the resulting equations to derive some necessary conditions for $\Phi$ to be torsion free.

\begin{lemma}\label{thm:Spin7_torsion}
	In the notation of lemma \ref{thm:Spin7_structure}, $d \Phi = 0$ is equivalent to
	\begin{subnumcases}{\label{eq:Spin7_torsion_v1}}
		d \omega = 0\label{eq:Spin7_torsion_v1_d_omega} \\
		d (p\Re \Omega) = -d \theta \wedge \omega\label{eq:Spin7_torsion_v1_dRe_Omega} \\
		d(r \Re \Omega) + d(q \Im \Omega) = -d\eta \wedge \omega\label{eq:Spin7_torsion_v1_dIm_Omega} \\
		d \eta \wedge p \Re \Omega -d \theta \wedge \left(r \Re \Omega +q \Im \Omega\right) + \frac{1}{2}d (pq) \wedge \omega^2 = 0\label{eq:Spin7_torsion_v1_third}
		.
	\end{subnumcases}
\end{lemma}
\begin{proof}
	Suppose $d\Phi = 0$.
	As usual, denote with $X$ and $Y$ the vertical vectors which are dual to $\eta$ and $\theta$.
	We saw previously that $Y \interior X \interior \Phi = \omega$.
	Then, by using the fact that $X$ and $Y$ commute (see remark \ref{rmk:XY_commute}), the fact that they preserve $\Phi$, and Cartan's magic formula, we get that
	$$
		d \omega=d(Y\lrcorner X\lrcorner \Phi)=\mathcal{L}_Y(X\lrcorner \Phi)-Y\lrcorner \mathcal{L}_X \Phi=X\lrcorner \mathcal{L}_Y \Phi=0
		.
	$$
	The rest of the proof consists in calculating explicitly $d\Phi$ from its expression in equation \eqref{eq:Phi_eta_theta} and noticing that its terms belong to three distinct subspaces of 5-forms: one is the image of horizontal 4-forms through the morphism $\Lambda^4(T^*M) \to \Lambda^5(T^*M)$ given by $\alpha \mapsto \eta \wedge \alpha$, the second is the image of horizontal 4-forms through $\alpha \mapsto \theta \wedge \alpha$ and the last is horizontal 5-forms.
	This decomposition yields the last three equations of system \eqref{eq:Spin7_torsion_v1}.
\end{proof}

\noindent System \eqref{eq:Spin7_torsion_v1} is clearly equivalent to the system
\begin{equation}\label{eq:Spin7_torsion}
	\begin{cases}
		d \omega = 0 \\
		d \Re \Omega = -\frac{1}{p} \left(dp \wedge \Re \Omega + d \theta \wedge \omega\right) \\
		d \Im \Omega = \frac{1}{pq} \left(\left(r dp - p dr\right) \wedge \Re \Omega + \left(r d \theta -p d \eta\right) \wedge \omega - pdq \wedge \Im \Omega \right) \\
		d \eta \wedge p \Re \Omega -d \theta \wedge \left(r \Re \Omega +q \Im \Omega\right) + \frac{1}{2}d (pq) \wedge \omega^2 = 0
		.
	\end{cases}
\end{equation}
which show us that equations \eqref{eq:Spin7_torsion} reinterpret the torsion of $\Phi$ in terms of the torsion of the $\SU(3)$-structure $(\omega, \Re \Omega)$ plus one equation that a priori has a less clear interpretation.
We can go further than this and reinterpret these equations in the context of lemma \ref{thm:SU3_torsion}.

\begin{lemma}\label{thm:SU3_torsion_Spin7_torsion}
	In the notation of lemma \ref{thm:Spin7_structure} and of lemma \ref{thm:SU3_torsion}, $d \Phi = 0$ is equivalent to the following constraints on the torsion of the $\SU(3)$-structure $(\omega, \Omega)$
	$$
	\begin{cases}
		w_1 = \hat{w}_1 = 0 \quad
		w_3 = 0 \quad
		w_4 = 0 \\
		w_2 = -\frac{1}{p}(d\theta)_8 \\
		\hat{w}_2 = - \frac{1}{pq}\left(r(d\theta)_8 + p (d\eta)_8\right) \\
		w_5 = \frac{1}{2pq}\left(J(pdr-rdp)- \frac{1}{2} d(pq)\right)
	\end{cases}
	$$
	together with the following constraint on the decomposition of $d\eta$ and $d\theta$ along $\Lambda^2 = \Lambda^2_1\oplus \Lambda^2_6 \oplus \Lambda^2_8$:
	\begin{equation}\label{eq:alphas_def}
		\begin{cases}
			d\eta = 0 + J\alpha_\eta^\sharp \interior \Re \Omega + (d \eta)_8 \\
			d\theta = 0 + J\alpha_\theta^\sharp \interior \Re \Omega + (d \theta)_8
		\end{cases}
	\end{equation}
	where $\alpha_\eta$ and $\alpha_\theta$ are given by
	$$
		\begin{cases}
			\alpha_\eta = \frac{1}{2pq}\left(qd(pq) + qJ\left(\frac{1}{2}qdp -\frac{3}{2}pdq\right)+rJ(rdp-pdr)+ p\left(\frac{1}{2}rdq-pdr\right)\right) \\
			\alpha_\theta = \frac{1}{2q}\left(J(rdp-pdr)+\frac{1}{2}pdq -\frac{3}{2}qdp\right)
			.
		\end{cases}
	$$
\end{lemma}
\begin{proof}
	Suppose equations \eqref{eq:Spin7_torsion} hold.
	$d \omega = 0$ is equivalent to $w_1 = \hat{w}_1 = 0, w_3 = 0$ and $w_4 = 0$.
	Because $d \omega = 0$, the equations for $\Re \Omega$ and $\Im \Omega$ imply that $d\eta$ and $d \theta$ have no component in $\Lambda^2_1$, proving part of the constraint on their decomposition.
	
	The rest of the proof consists in decomposing equations \eqref{eq:Spin7_torsion} in types and rewriting $d\eta = - J\alpha_\eta^\sharp \interior \Re \Omega + (d \eta)_8$ and $d\theta = - J\alpha_\theta^\sharp \interior \Re \Omega + (d \theta)_8$.
	We have that
	$$
		d \eta \wedge \omega = \alpha_\eta \wedge \Re \Omega + (d\eta)_8 \wedge \omega
	$$
	and similarly for $\theta$.
	The equation involving $\Re \Omega$ directly implies the equation involving $w_2$ and the one involving $\Im \Omega$ implies the one with $\hat{w}_2$.
	Neither of the two alone implies the one involving $w_5$, which depends on calculating $\alpha_\eta$ and $\alpha_\theta$ explicitly, which is what we do in the rest of the proof.
	Imposing that the $w_5$ from $d \Re \Omega$ and the $w_5$ from $d \Im \Omega$ are equal amounts to requiring
	\begin{equation}\label{eq:same_w5}
		p \alpha_\eta = (r+qJ) \alpha_\theta + J(qdp-pdq) + rdp-pdr
		.
	\end{equation}
	The last equation in $\eqref{eq:Spin7_torsion}$ implies that
	\begin{equation}\label{eq:alphas_Spin7_torsion_third}
		p \alpha_\eta = (r-qJ) \alpha_\theta -\frac{1}{2} J d(pq)
		.
	\end{equation}
	These two equations together allow us to find the explicit expressions for $\alpha_\eta$ and $\alpha_\theta$ given in the statement.
	
	This finally allows to make the equation involving $w_5$ explicit.
\end{proof}

Thus, in light of lemma \ref{thm:SU3_torsion_Spin7_torsion}, we can remark that equation \eqref{eq:Spin7_torsion_v1_third} is half of the constraint on the torsion component $w_5$ of the $\SU(3)$-structure $(\omega, \Omega)$ in order for the 4-form $\Phi$ to be closed (the other half coming from the $\Lambda_6$ projection of equations \eqref{eq:Spin7_torsion_v1_dRe_Omega} and \eqref{eq:Spin7_torsion_v1_dIm_Omega}).

In the following corollary, we put in the spotlight two equations implied by lemma \ref{thm:SU3_torsion_Spin7_torsion} that will be handy in solving the equations analytically.

\begin{corollary}\label{thm:alphas}
	Suppose system \eqref{eq:Spin7_torsion_v1} holds.
	Then the following two equations hold:
	\begin{equation}\label{eq:simil_CY_monopole}
		q d \theta \wedge \Im \Omega + \frac{1}{2}\left(\frac{1}{2}\left(pdq-3qdp\right)+J(rdp - pdr) \right)\wedge \omega^2 = 0
		.
	\end{equation}
	and
	\begin{equation}\label{eq:same_w5_v1}
		(pd\eta-rd\theta) \wedge \Re \Omega + q d \theta \wedge \Im \Omega + \left(pdq-qdp+J(rdp - pdr) \right)\wedge \omega^2 = 0
	\end{equation}
\end{corollary}
\begin{proof}
	Equation \eqref{eq:simil_CY_monopole} is just another way to write the explicit expression for $\alpha_\theta$ found in lemma \ref{thm:SU3_torsion_Spin7_torsion}.
	It follows directly from the fact that if $d\theta =  J\alpha_\theta^\sharp \interior \Re \Omega + (d \theta)_8$, then $d\theta \wedge \Re \Omega = \left(J\alpha_\theta^\sharp\right)^\flat \wedge \omega^2 = - J \alpha_\theta \wedge \omega^2$ and thus ${\star}\left(d \theta \wedge \Re \Omega \right) = -2 \alpha_\theta$.
	Moreover, we used that $J(d \theta \wedge \Re \Omega) = -d\theta \wedge \Im \Omega$, due to the fact that $\pi_1(d \theta) = 0$.
	
	Similarly, equation \eqref{eq:same_w5_v1} is just another way to write equation \eqref{eq:same_w5}.
\end{proof}

\begin{remark}
	Equation \eqref{eq:simil_CY_monopole} is equivalent to equation \eqref{eq:Spin7_torsion_v1_third} and can hence replace it.
	
	Indeed, as we saw, equation \eqref{eq:Spin7_torsion_v1_third} is just a condition on $\alpha_\eta$ and $\alpha_\theta$, but given that equations \eqref{eq:Spin7_torsion_v1_dRe_Omega} and \eqref{eq:Spin7_torsion_v1_dIm_Omega} already give one condition on them, the knowledge of any of the two $\alpha$s is sufficient to determine the other.
\end{remark}

\begin{remark}
	One difference between this case and the $G_2$ case explored in \cite{Foscolo2021} is that here none of the equations of system \eqref{eq:Spin7_torsion_v1} is redundant.
	This is because in this case we have two forms in $\Lambda^2_6$ to constrain ($\alpha_\eta$ and $\alpha_\theta$), whereas in the $G_2$ case there is only one.
	Hence, we need one more constraint which is precisely equation \eqref{eq:Spin7_torsion_v1_third}.
\end{remark}

\begin{remark}\label{rmk:free_parameter}
	Given that we want to make use of the implicit function theorem, we are interested in adding free parameters to the equations in order to make up for a possible non-surjectivity of the derivative of the function that defines the equations.
	One way to do this is to solve $d \Phi = \beta$ for some $\beta \in V$ where $V \leq \Lambda^5(M)$ is any subspace in direct sum with $d\Lambda^4(M)\leq \Lambda^5(M)$.
	In our case, we will consider the equation
	$$
	d \Phi = (pq)^{-\frac{1}{2}}\star_\Phi \eta \wedge \theta \wedge d s = p^*_M \star_\omega ds
	$$
	for some $s \in C^1(B)$, where $p_M : M \to B$ is the standard projection map.
	By applying $d$ on both sides, we get the equation $d {\star_\omega} ds$ on $B$, which tells us that $s$ is harmonic.
	So, e.g.\ in a complete non-compact case like the one we propose to study in this paper, if we require $s$ to be decaying, we get that $s = 0$ by the maximum principle.
\end{remark}

\subsection{Topological constraints}\label{sec:topological_constraints}

Equations \eqref{eq:Spin7_torsion_v1} impose some clear topological constraints on the bundle $M \to B$, which are summarized in the following proposition.

\begin{proposition}\label{thm:symplectic_topological_constraint}
	Let $B$ be a 6-manifold and let $M \to B$ be a $T^2$-bundle.
	
	Then, if $M$ admits a $T^2$-invariant torsion-free $\Spin(7)$-structure, condition \eqref{eq:topological_condition} holds.
	Equivalently, $M$ is the product of two circle bundles $P_1 \to B$ and $P_2 \to B$ such that
	$$
	c_1(P_1) \smile [\omega] = c_1(P_2) \smile [\omega] = 0 \in H^4(B)
	$$
	for some symplectic form $\omega$ on $B$.
\end{proposition}
\begin{proof}
	We prove the equivalent formulation in the statement.
	The fact that $M$ splits as the product of two circle bundles was stated in proposition \ref{thm:product_bundle}.
	
	The constraint follows from the fact that the left-hand sides of equations \eqref{eq:Spin7_torsion_v1_dRe_Omega} and \eqref{eq:Spin7_torsion_v1_dIm_Omega} are exact and that by equation \eqref{eq:Spin7_torsion_v1_d_omega}, $\omega$ is closed.
	Indeed, recall that, by Chern-Weil theory, the curvature of the connection on a principal circle bundle represents its Chern class.
\end{proof}

Equation \eqref{eq:Spin7_torsion_v1_third} yields yet another constraint.
Indeed, a priori,
$$
	d \eta \wedge p \Re \Omega-d \theta \wedge \left(r \Re \Omega +q \Im \Omega\right)
$$
is a closed form by equations \eqref{eq:Spin7_torsion_v1_dRe_Omega} and \eqref{eq:Spin7_torsion_v1_dIm_Omega}.
However, equation \eqref{eq:Spin7_torsion_v1_third} implies that such form is exact.
One way to write down this constraint is via Massey triple products, as we state in the following.
\begin{proposition}
	Let $B$ be a 6-manifold and let $M \to B$ be a $T^2$-bundle.
	
	Then if $M$ admits a $T^2$-invariant torsion-free $\Spin(7)$-structure, it is the product of two circle bundles $P_1 \to B$ and $P_2 \to B$ such that, for some $[\omega] \in H^2(B)$,
	$$
	\left\langle c_1(P_1), [\omega], c_1(P_2) \right\rangle_{MT} = 0 \in \faktor{H^5(B)}{W}
	$$
	where $\langle ,, \rangle_{MT}$ denotes the Massey triple product and $W := c_1(P_1) \wedge H^3(B) + c_1(P_2) \wedge H^3(B)$.
\end{proposition}
\begin{proof}
	The Massey triple product makes sense because of the topological constraint proved in theorem \ref{thm:symplectic_topological_constraint}.
	The thesis follows directly from equation \eqref{eq:Spin7_torsion_v1_third} and the definition of Massey triple product.
\end{proof}
Clearly, when $H^5(B) = 0$, the above constraint is trivial.
This is going to be the case in this paper, since we solve the equations on simply connected $AC$ Calabi-Yau 3-folds, and for such manifolds $H^5$ vanishes.
This is because, by Poincaré duality, $H^5(B) \iso H^1_c(B)^*$ and $H^1_c(B)$ needs to vanish if $H^1(B)$ vanishes because of the long exact sequence for manifold with boundary.
Here we are treating $B$ as a manifold with boundary $\Sigma$ because $B$ clearly has the same homotopy type of a big enough compact subset of itself, which is a manifold with boundary $\Sigma$.

\subsection{A more abstract approach to the equations}\label{sec:abstract_equations_torsion}
To write the explicit expression for $\Phi$ we made a choice of basis for the Lie algebra $\Lie(T^2)$ which we discussed thoroughly in section \ref{sec:T2_connections}.
However, we can write down $\Phi$ in a more abstract way without making a choice of basis for $\Lie(T^2)$.

To do this we first note that given a manifold $M$ and a vector bundle $E \to M$, it is somehow natural to consider the wedge product of two $E$-valued differential forms as a differential form valued in $E \wedge E$.
We denote such product with $\wedg{\wedge}$.
Indeed, if $F$ is another vector bundle, the natural way of defining the $\wedge$ product between an $E$-valued form and an $F$-valued form as an $E \otimes F$-valued form gives a well-defined theory for which the usual formulas for $\wedge$ and $d$ are valid.
We denote this product with $\wedg{\otimes}$.
In the case where $E = F$, it is natural to feed the output of $E \otimes E$-valued differential forms to the canonical projection morphism of $E \otimes E$ onto $E \wedge E$.
By construction hence, the $\wedg{\wedge}$ product of $E$-valued forms is defined uniquely by the property that pointwise, given two pure tensors $\alpha_1 \otimes e_1$ and $\alpha_2 \otimes e_2$,
$$
	(\alpha_1 \otimes e_1) \wedg{\wedge} (\alpha_2 \otimes e_2) = (\alpha_1 \wedge \alpha_2) \otimes (e_1 \wedge e_2)
	.
$$
Note that $\wedg{\wedge}$ is symmetric.
The formulas for $\wedg{\wedge}$ and $d$ that one would imagine follow in a straightforward way from those in the $\otimes$ case.

Similarly, given a form valued in $\Lambda^k E$ and a form valued in $\Lambda^mE^*$ for $k\leq m$, it is natural feed the output of the product valued in $\Lambda^k E \otimes \Lambda^mE^*$ to the \textit{interior product} morphism $\iota \Lambda^k E \otimes \Lambda^mE^* \to \Lambda^{m-k} E^*$, defined uniquely by its behavior on pure elements:
$$
	(v_1 \wedge \ldots \wedge v_k) \otimes \alpha \mapsto \iota_{v_1} \ldots \iota_{v_k} \alpha
	.
$$
We denote such product by $\wedg{e}$.
As in the previous case, the same formulas as for $\wedge$ and $d$ apply.
The order of the interior products in the above morphism is important to preserve the associativity of $\wedg{\wedge}$ and $\wedg{e}$, i.e.\ with such convention a calculation shows that
$$
	(\alpha \wedg{\wedge} \beta) \wedg{e} \psi = \alpha \wedg{\wedge} (\beta \wedg{e} \psi)
	.
$$

Now, by staring at expression \eqref{eq:Phi_eta_theta}, one might realize that if $\eta, \theta$ represent the two components of a $T^2$-connection $A \in \Omega^1(P; \Lie(T_2))$ through a chosen basis, $\eta \wedg{\wedge} \theta$ is invariant under transformations in $\SL_\R(2)$ of the given basis, and represents in fact a well-defined element of $\Omega^2(P; \Lie(T_2) \wedge \Lie(T_2))$, namely $\frac{1}{2}A \wedg{\wedge} A$.
In order to get a scalar when wedging with $\omega$, we interpret $\omega$ as a section $\invomega$ of $\Omega^2(P; \Lie(T^2)^*\wedge \Lie(T^2)^*)$.
We do this by taking the $\otimes$ product of $- \omega$ and the volume form in $\Lie(T^2)^*\wedge \Lie(T^2)^*$ given by the chosen basis of $\Lie(T^2)$.
Hence, $\frac{1}{2}A \wedg{\wedge} A \wedg{e} \invomega \in \Omega^4(P)$ and thus we want to reinterpret the other summands too as elements of $\Omega^4(P)$.
This naturally leads us to define a section $\mho \in \Omega^3(P; \Lie(T^2)^*)$ which under the chosen basis is precisely given by
$$
	\begin{bmatrix}
		p \Re \Omega \\ -r \Re \Omega - q \Im \Omega
	\end{bmatrix}
	.
$$
Thus, we can write $\Phi \in \Omega^4(P)$ more elegantly as
\begin{equation}\label{eq:abstract_Phi}
	\Phi = \frac{1}{2} A \wedg{\wedge} A \wedg{e} \invomega + A \wedg{e} \mho + \frac{\mho \wedg{\wedge} \mho}{\invomega}
	.
\end{equation}
where the last term is defined to be the unique form $\alpha$ such that $\alpha \wedge \invomega = \mho \wedg{\wedge} \mho$ and $\alpha \propto \omega^2$ where $\omega\in \Omega^2(M)$ is determined by $\invomega$ by choosing a basis of $\Lie(T^2)^* \wedge \Lie(T^2)^*$.
Note that this last condition is independent of the chosen basis.


\noindent In these abstract terms, equations \eqref{eq:Spin7_torsion_v1} become more elegantly
\begin{equation}\label{eq:abstract_torsion}
	\begin{cases}
		d \mho + dA \wedg{e} \invomega = 0 \\
		dA \wedg{e} \mho + d \left( \frac{\mho \wedg{\wedge} \mho}{\invomega} \right) = 0
		.
	\end{cases}
\end{equation}
Notice that since in this system $\invomega$ and $\mho$ are horizontal and $T^2$ invariant, we might as well interpret, $\mho$ as an element of $\Omega^3(B; \ad P^*)$ and $\invomega$ in $\Omega^2(B; \ad P^* \wedge \ad P^*)$.
Since the exterior derivative vector bundle valued forms makes sense only in terms of a connection on the vector bundle, we should specify which connections we are considering here, but since all the considered bundles are trivial (both $P \times \mathfrak{g}$ and $\ad P$ are, since $T^2$ is abelian) it is the trivial connection in all cases.

System \eqref{eq:abstract_torsion} can be derived directly from expression \eqref{eq:abstract_Phi}.
Indeed,
$$
	d \Phi = A \wedg{e} dA \wedg{e} \invomega + dA \wedg{e} \mho -A \wedg{e} d \mho + d \left( \frac{\mho \wedg{\wedge} \mho}{\invomega} \right)
	.
$$
and it is straightforward to remark that the terms beginning with $A \wedg{e}$ and those that do not belong to spaces in direct sum.


We summarize the discussion of this section in the following result.

\begin{theorem}
	Let $B$ be a 6-manifold and $\pi: M \to B$ be a principal $T^2$-bundle.
	Then, any $T^2$-invariant $\Spin(7)$-structure on $M$ will be of the form \eqref{eq:abstract_Phi}, for some unique
	$$
		A \in \mathcal{A}(M \to B) \quad \mho \in \Omega^3(B; \ad(M)^*) \quad \invomega \in \Omega^2(B; \ad(M)^* \wedge (\ad M)^*)
	$$
	satisfying $\invomega \wedg{\otimes} \mho = 0$.
\end{theorem}

Note that we do not need to require any version of the second condition defining an $\SU(3)$-structure because that comes from defining $\Omega$ in terms of $\mho$.

\section{New closed \texorpdfstring{$\Spin(7)$}{Spin(7)} forms on \texorpdfstring{$T^2$}{T\^2} bundles over AC CY 3-folds}\label{sec:analytic_curve}

In this section we prove the existence of an analytic curve of solution to equations \eqref{eq:Spin7_torsion_v1} by making use of the implicit function theorem.
We are going to use a very standard version of the implicit function theorem for Banach manifolds, which we state to fix the notation.
\begin{theorem}[Implicit function theorem]\label{thm:implicit_function_theorem}
	Let $X, Y$ be Banach manifolds and $Z$ be a Banach space.
	Let the mapping $\Psi: X \times Y \rightarrow Z$ be an analytic map.
	
	If $\left(x_0, y_0\right) \in X \times Y$, $\Psi\left(x_0, y_0\right)=0$, and
	$$
	\begin{aligned}
		D\Psi_{(x_0, y_0)} : T_{(x_0, y_0)}(X \times Y) & \to T_0Z \iso Z\\
		y &\mapsto D \Psi_{\left(x_0, y_0\right)}(0, y)
	\end{aligned}
	$$
	is a Banach space isomorphism from $T_{y_0}Y$ onto $Z$, then there exist neighborhoods $U \subseteq X$ of $x_0$ and $V \subseteq Y$ of $y_0$ and an analytic function $\Xi: U \rightarrow V$ such that $\Psi(x, \Xi(x))=0$ and $\Psi(x, y)=0$ if and only if $y=\Xi(x)$, for all $(x, y) \in U \times V$.
\end{theorem}

As anticipated in the introduction, in order to apply the theorem, the idea is to introduce a parameter $\varepsilon$ as the independent variable of the theorem, and then to treat all the other variables as dependent.
This parameter is meant to implement the shrinking of the fibers to points.
We thus ideally choose $X = \R$, so that the variable $x$ is precisely this parameter $\varepsilon$, and we choose $Y$ to be the space were the data appearing in system \eqref{eq:Spin7_torsion_v1} lives.
Once introduced the parameter $\varepsilon$, the general strategy is to manipulate the equations in such a way that it is possible to prove that their linearization is invertible.

There are multiple modifications we need to bring to the equations.
The first is to add to system \eqref{eq:Spin7_torsion_v1} equation \eqref{eq:same_w5_v1} and to add free variables to the equations, which we do in section \ref{sec:free_parameters}.
This is not yet sufficient to prove invertibility because we never took care of gauge fixing, and thus, in section \ref{sec:gauge} we add additional equations to take care of the excessive freedom.
In section \ref{sec:order_0} we study the necessary conditions for the tuple $(\varepsilon = 0, \omega, \Omega, p, q, r)$ to be a candidate base point to solve the equations by applying the implicit function theorem and we find out $(\omega, \Omega)$ needs to be Calabi Yau (hence $d \omega = 0$ and $d \Omega = 0$), and $p, q, r$ need to be constant.
In section \ref{sec:first_order} we bring the last changes to prepare the equations for the application of the theorem, that consist in reducing the equations entirely on $B$ instead of $M$ and to add some terms that make sure that they can happen in Hölder spaces of sections that have enough decay to be able to do analysis.
The first modification is needed because a priori the equations happen on $M$, since $\eta$ and $\theta$ are genuine 1-forms on $M$ and are not the pull-back of forms on $B$, unlike all the rest of the data.
To take care of this, we solve the first order in $\varepsilon$ of part of the equations by hand and in the process we take care of the vertical part of $\eta$ and $\theta$.
Since any other connection can be obtained by shifting the first order solution $(\eta_1, \theta_1)$ by horizontal forms $(\delta_\eta, \delta_\theta)$, we are brought to look for a tuple of solutions of the form $(\varepsilon, \omega, \Omega, \delta_\eta, \delta_\theta, p, q, r)$ which is all data that genuinely lives on $B$.
However, the first order solutions that we find for $\eta$ and $\theta$ have curvature $d\eta, d\theta$ in $C^\infty_{-2}$, but we need it in $C^\infty_{\nu}$ for $\nu \in (-\infty, -2)$ in order to define a function $\Psi$ on spaces on which we can do analysis.
To do so we solve the equations also for the first order in $\varepsilon$ of $\Re \Omega$, so to have a full solution of the first order equations involving $\eta$, $\theta$, and we subtract these equations from the overall equations.
This allows us to take care separately of the sections that have a decay which is too slow and hence to define the function $\Psi$ of the implicit function theorem on spaces that have enough decay.
Having brought all these changes, in section \ref{sec:linearization} we are finally able to prove that the linearization is invertible.

\subsection{Free parameters}\label{sec:free_parameters}

The free parameters come from two sources.
First, two functions $u, v$ and one vector field $X$ take care of the redundancy of the equations coming from the constraints on the torsion of the $\SU(3)$-structure $(\omega, \Omega)$ and from the addition of equation \eqref{eq:same_w5_v1}.
Secondly, one additional function $s$ comes from remark \ref{rmk:free_parameter}.
The constraints on the torsion of the $\SU(3)$ are explained in proposition \ref{thm:SU3_torsion} and the functions $u$ and $v$ take care of the fact that the decomposition of equations \eqref{eq:Spin7_torsion_v1_dRe_Omega} and \eqref{eq:Spin7_torsion_v1_dIm_Omega} in $\SU(3)$ irreducible components is redundantly imposing $w_1=0$ and $\hat{w}_1 = 0$ a second time since $d \omega = 0$ already imposes such constraints.
Moreover, equation \eqref{eq:same_w5_v1} imposes that the $w_5$ coming from equation \eqref{eq:Spin7_torsion_v1_dRe_Omega} is the same as the $w_5$ coming from equation \eqref{eq:Spin7_torsion_v1_dIm_Omega} which is also implied by proposition \ref{thm:SU3_torsion}, and thus $X$ takes care of such redundancy.

We formally prove that the system with the free variables is equivalent to system \eqref{eq:Spin7_torsion_v1} in the following lemma.
\begin{proposition}\label{thm:free_parameters}
	Let $\left(\omega_0, \Omega_0\right)$ be an AC Calabi-Yau structure on a 6 -manifold $B$ and denote by $g_0$ and $\nabla_0$ the induced metric and Levi-Civita connection.
	Fix $k \geq \N^+, \alpha \in(0,1)$ and $\nu\in(-\infty,-1)$.
	Then there exists a constant $\varepsilon_0>0$ such that the following holds.
	Let $(\omega, \Omega)$ be a second $\mathrm{SU}(3)$-structure on $B$ whose distance from $\left(\omega_0, \Omega_0\right)$ in norm $C^1_{-1}(B)$ is less than $\varepsilon_0$, i.e.\ such that $\norm{\omega-\omega_0}_{C^1_{-1}} + \norm{\Omega-\Omega_0}_{C^1_{-1}} < \varepsilon_0$.
	Suppose that there exist two integral exact 2-forms $d \theta$ and $d \eta$ on $B$ such that
	\begin{equation}\label{eq:parameters_linear_conditions}
		d \omega = 0 \qquad d \eta \wedge \omega^2=0 \qquad d \theta \wedge \omega^2=0
		.
	\end{equation}
	Moreover, assume the existence of functions $p, q, r, s, u, v$ and a vector field $X$ in $C_{\nu+1}^{k+1, \alpha}$ such that
	\begin{subnumcases}{\label{eq:parameters}}
		(pd\eta-rd\theta) \wedge \Re \Omega + q d \theta \wedge \Im \Omega + \left(pdq-qdp+J(rdp - pdr) \right)\wedge \omega^2 = 0\label{eq:parameters_CY_monopole_theta} \\
		(pd\eta-rd\theta) \wedge \Re \Omega - q d \theta \wedge \Im \Omega + \frac{1}{2}\left(d(pq) + Jds\right) \wedge \omega^2 = 0\label{eq:parameters_CY_monopole_xi} \\
		d (p\Re \Omega) + d \theta \wedge \omega = d{\star} d(u \omega)\label{eq:parameters_dRe_Omega} \\
		d(r \Re \Omega) + d(q \Im \Omega) + d\eta \wedge \omega = d {\star} d (X \interior \Re \Omega + v \omega)\label{eq:parameters_dIm_Omega}
	\end{subnumcases}
	where the Hodge $\star$ is computed with respect to the metric induced by $\left(\omega, \Omega\right)$.
	Then $u = v = 0, s = 0$ and $X = 0$, i.e.\ $(\omega, \Re \Omega, p, q, r, \eta, \theta)$ is a solution system \eqref{eq:Spin7_torsion_v1}.
\end{proposition}
\begin{proof}
	It is a simple calculation to show that given a 4-form $\chi$, $\pi_1(\chi) = {\star} (\chi \wedge \omega) \omega^2$. 
	Thus, by wedging equation \eqref{eq:parameters_dRe_Omega} with $\omega$ and using equations \eqref{eq:parameters_linear_conditions} we get that $\pi_1(d\star d(u \omega)) = 0$.
	By the fact that $\star$ commutes with $\pi_i$ for any $i$, we get that $\pi_1(d^* d(u \omega)) = 0$.
	Since $\left(\omega, \Omega\right)$ and $\left(\omega_0, \Omega_0\right)$ are $C^1_{-1}$-close, the operators $C_{\nu+1}^{k+1, \alpha} \rightarrow C_{\nu-1}^{k-1, \alpha}$ given by $u \mapsto \pi_1^{\omega, \Omega} d^* d(u \omega)$ and $u \mapsto \pi_1^{\omega_0, \Omega_0} d^* d\left(u \omega_0\right)$ differ by a bounded operator of norm controlled by $\varepsilon_0$.
	Lemma \ref{thm:pi1_6} shows that the kernel of $u \mapsto \pi_1^{\omega_0, \Omega_0} d^* d\left(u \omega_0\right)$ is given by harmonic functions.
	But since $\nu < - 1$, the kernel needs to be 0, since 0 is the only harmonic decaying function.
	Because injective operators on Banach spaces are an open set in the operator norm, by taking $\varepsilon_0$ small enough, $u \mapsto \pi_1^{\omega, \Omega} d^* d(u \omega)$ is also injective, an hence $u = 0$.
	
	We now take care of $X$, $v$ and $s$.
	Taking the wedge product of both sides of equation \eqref{eq:parameters_dIm_Omega} with $\omega$, exactly as above, shows that $\pi_1(d^*d(X \interior \Re \Omega + v \omega)) = 0$.
	Now, by using $\alpha_\eta$ and $\alpha_\theta$ defined in equation \eqref{eq:alphas_def}, equation \eqref{eq:parameters_CY_monopole_theta} is equivalent to
	$$
		p\alpha_\eta - r \alpha_\theta -q J \alpha_\theta + J(pdq - qdp) + pdr - r dp = 0
		.
	$$
	On the other hand, the projection of the left-hand side of equation \eqref{eq:parameters_dIm_Omega} on $\Omega^4_6$ is
	$$
		\frac{1}{p}\left(pdr -rdp - r \alpha_\theta + J(pdq-qdp) - qJ\alpha_\theta + p \alpha_\eta\right) \wedge \Re \Omega
	$$
	which is null by the equation just above.
	Hence, $\pi_{1\oplus 6}(d \star d (X \interior \Re \Omega + v \omega))$, and, again by using $C^1_{-1}$-closeness of $\left(\omega, \Omega\right)$ and $\left(\omega_0, \Omega_0\right)$ and lemma \ref{thm:pi1_6}, we get that $X = 0$ and $v = 0$.
	
	For $s$, we apply ${\star} d$ to both sides of equation \eqref{eq:parameters_CY_monopole_xi} and, by using that $d \omega = 0$ and equations \eqref{eq:parameters_dRe_Omega} and \eqref{eq:parameters_dIm_Omega}, we get that ${\star} d J ds \wedge \omega^2 =  2d^* ds = 0$, and thus $s$ is harmonic.
	Since it decays, this is not possible unless $s$ is constant and thus $s = 0$ again by the decaying condition.

	Finally, we prove that system \eqref{eq:parameters} with $u = v = s = 0$ and $X=0$ is equivalent to system \eqref{eq:Spin7_torsion_v1}.
	One implication is obvious since system \eqref{eq:Spin7_torsion_v1} is contained in system \eqref{eq:parameters}.
	The other follows from corollary \ref{thm:alphas}.

\end{proof}

\subsection{Gauge fixing}\label{sec:gauge}

In order to apply the implicit function theorem, we need to eliminate any excessive freedom in the equations, otherwise the linearization turns out not to be injective.
The main source of freedom comes from the fact that equations \eqref{eq:Spin7_torsion_v1} are invariant under gauge symmetry.

Indeed, suppose we have a solution $(\eta, \theta, \omega, \Re \Omega, p, q, r)$ of such equations.
Then we could get another solution in three ways: by moving on the underlying possibly non-trivial moduli space of Calabi Yau structures, by acting with a diffeomorphism on the base manifold $B$ that respects the AC structure on $(\omega, \Re \Omega, p, q, r)$ or by acting with a decaying gauge transformation of the principal bundle $M \to B$ on $(\eta, \theta)$.
By a diffeomorphism that respects the asymptotically conical structure, we mean more precisely a diffeomorphism through which the pullback of the metric $g_{\omega, \Omega}$ is still AC.
By decaying gauge transformation we mean the following: since $T^2$ is abelian, gauge transformations can be identified with smooth maps $B \to T^2$, and hence, by making use of the Lie algebra identification between $\Lie(T^2)$ and $\R^2$ we can identify the differential $d\Psi$ of a gauge transformation $\Psi$ as a pair of 1-forms $\gamma^1_\Psi$ and $\gamma^2_\Psi$.
We define the space of gauge transformations decaying with rate $\nu$ as those $\Psi$ such that $\gamma^1_\Psi, \gamma^2_\Psi \in C^\infty_{\nu - 1} (\Lambda^1(B))$.

The reason why we need to consider diffeomorphism that respect the AC structure and decaying gauge transformations is that, since we are in a non-compact situation, we aim to find some solutions by perturbing a collapsed solution with some decaying sections.
This is because the only hope to do analysis in a non-compact setting is to impose a decaying condition.
Hence, we only need to work transversally to decaying diffeomorphisms and gauge transformations, because these are the only ones that preserve the spaces of sections that we are concerned with.

We will split deformations of the Calabi-Yau structure into deformations of the Kähler form and deformations of the complex volume form, and we will deal with the former as the first thing ad with the latter as the last thing.
This is because it is convenient to get rid of the freedom on Kähler deformations before studying diffeomorphisms and it is convenient to get rid of diffeomorphism invariance before studying deformations of the complex volume form.
Let us now put this discussion in practice.

\paragraph{Kähler deformations} We use the freedom of moving on the moduli space of Kähler structures to keep fixed the cohomology class.

\begin{lemma}
	Let $k \in \N^+$, $\alpha \in (0,1)$ and $\nu \in (- \infty, 0)$.
	Let $(B, \omega_0, \Omega_0)$ be an asymptotically conical Calabi Yau manifold.
	Then any Kähler form $\omega$ on $B$ close enough to $\omega_0$ in the $C^{k, \alpha}_\nu$-topology can be deformed to have the same cohomology class as $\omega_0$, i.e.\ $[\omega] = [\omega_0]$.
\end{lemma}
\begin{proof}
	Every cohomology class near $[\omega_0]$ is a Kähler class and every Kähler class is represented by a unique Kähler class $\omega$ such that $(\omega, \Omega_0)$ is Calabi Yau (see section 7.1 in \cite{Foscolo2021} for a proof of this).
	Hence, the moduli space of Kähler forms close to $\omega_0$ is homeomorphic to a neighborhood of $[\omega_0]$ in $H^2(B)$.
	The thesis follows.
\end{proof}

\paragraph{Diffeomorphisms} The content of this paragraph comes from section 7.1 of \cite{Foscolo2021}.
The first thing to do is to understand how to use the diffeomorphism freedom on the Kähler form $\omega$ and then we consider what to do with the remaining freedom for $\Omega$.


By Moser's trick and the fact that we imposed the Kähler class to be fixed we can assume $\omega$ to be fixed, and under this hypothesis we can find a simple condition to choose a representative in the diffeomorphism class of $\Omega$.

\begin{lemma}
	Let $k \in \N^+$, $\alpha \in (0,1)$ and $\nu \in (- 5, -1)$.
	Let $(B, \omega_0, \Omega_0)$ be an asymptotically conical Calabi Yau manifold and $(\omega, \Omega)$ an asymptotically conical $\SU(3)$-structures that is symplectic (as in with closed Kähler form) with $[\omega] = [\omega_0]$, and suppose $(\omega_0, \Re \Omega_0)$ and $(\omega, \Omega)$ are $\varepsilon$-close in the $C^{k, \alpha}_\nu$-topology for some $\varepsilon \in \R^+$.
	
	Then, for $\varepsilon$ small enough, the following conditions
	\begin{subnumcases}{\label{eq:diff_fixing}}
		\omega = \omega_0 \\
		\Re \Omega \wedge \Re \Omega_0 = 0\label{eq:diffeo_fixing}
		.
	\end{subnumcases}
	fix uniquely a symplectic $\SU(3)$-structure in the diffeomorphism class of $(\omega, \Omega)$.
\end{lemma}

Thus, from now on we will take $\omega = \omega_0$ to be fixed and we add the linear condition $\Re \Omega \wedge \Re \Omega_0 = 0$.

\paragraph{Gauge transformations} As for gauge transformations, let $(\eta_1, \theta_1)$ be a connection on $M\to B$.
Then any other connection is given by $(\eta_1 + \delta_\eta, \theta_1 + \delta_\theta)$ for some $\delta_\eta, \delta_\theta \in \Omega^1(B)$.
The following condition fixes a representative in a given gauge class:
\begin{equation}\label{eq:gauge_fixing}
	d^* \delta_\eta = 0 \qquad d^* \delta_\theta = 0
	.
\end{equation}

Indeed, let $(\eta_1 + \delta_\eta + \upsilon_\eta, \theta_1 + \delta_\theta + \upsilon_\theta)$ (with $\upsilon_\eta, \upsilon_\theta \in \Omega^1(B)$) be another representative in the gauge class $[(\eta_1 + \delta_\eta, \theta_1 + \delta_\theta)]$ that satisfies \eqref{eq:gauge_fixing}.
Then $\upsilon_\eta$ and $\upsilon_\theta$ are closed because of how gauge transformations act on connections for $\U(1)$ principal bundles.
Since they are also coclosed by \eqref{eq:gauge_fixing} and they are decaying because we assumed the gauge transformations to be decaying, they must vanish.

\paragraph{Deformations of the complex volume form} Let us finally take care of the possible non-triviality of the moduli space of Calabi-Yau structures.
Let $\rho$ be an infinitesimal deformation of $\Re \Omega$.
Since we already took care of fixing the diffeomorphism class, we are only interested in the genuine Calabi Yau deformations that are not coming from diffeomorphisms.
Hence, by the discussion above, we can assume $\rho \wedge \Re \Omega_0 = 0$.
The following theorem clarifies the situation.
\begin{theorem}
	Let $\left(0, \rho\right) \in \Omega^2(B)\times \Omega^3(B)$, such that $\rho \wedge \Re \Omega_0 = 0$.
	Then, $\left(0, \rho\right)$ is an infinitesimal Calabi-Yau deformation of $\left(B, \omega_0, \Omega_0\right)$ if and only if $\rho^{\prime}$ is a closed and coclosed 3-form in $\Omega_{12}^3$.
\end{theorem}
\begin{proof}
	Since $(0, \rho)$ is an infinitesimal Calabi-Yau deformation, $d \rho = d \hat{\rho} = 0$, $\rho \wedge \omega_0= 0$ and $\operatorname{Re} \Omega_0 \wedge \hat{\rho}+\rho \wedge \operatorname{Im} \Omega_0 = 0$, and thus $\rho = f \Im \Omega_0 + \rho_{12}$ with $f \in C^\infty(B)$ and $\rho_ {12} \in \Omega_{12}^3$.
	Since $\rho \wedge \operatorname{Re} \Omega_0=0$, $\rho \in \Omega_{12}^3$ and thus $d \hat{\rho}=0$ is equivalent to $d^* \rho=0$ by proposition \ref{thm:Hitchin_map_derivative}.
\end{proof}

Hence, by integrating the last proposition, we see that we can work transversally to deformations of $\Omega$ by requiring that
\begin{equation}\label{eq:CY_deformations_fixing}
	\pi_{12}(\Re \Omega - \Re \Omega_0 ) \in \mathcal{W}_\nu^3(B)
\end{equation}
where $\mathcal{W}_\nu^3(B)$ is defined in \ref{def:Wknu}.
This last condition is an arbitrary gauge fixing choice and is asking that, if we perturb our starting $\Re \Omega_0$ by a form $\rho$, the $\Lambda^3_{12}$ component of $\rho$ has to be orthogonal to closed and coclosed 3-forms.

\subsection{The equations in the adiabatic limit}\label{sec:order_0}

The equations obtained up to now, i.e.\ system \eqref{eq:parameters} together with the linear conditions \eqref{eq:parameters_linear_conditions} and the gauge fixing equations found in section \ref{sec:gauge} are more or less ready for the application of the implicit function theorem, since the modifications that we apply in section \ref{sec:first_order} are minor and technical.
These equations are the zero locus of a function $\Psi$ that takes as input the following data: $(x, y) = (\varepsilon, \eta, \theta, \Re \Omega, p, q, r, s, u, v, X)$, where $x = \varepsilon$ as we explained above.
In this subsection we intend to study what conditions do candidate base points $(x_0, y_0)$ need to satisfy in order to apply the theorem.

The reason why it is a good idea using the implicit function theorem to solve our equations is that it should be easy to solve them in the limit where the fibers shrink to points, called the \textit{adiabatic} limit, i.e.\ choosing $x_0 = 0$.
Hence, in this subsection we study the equation $\Psi(0, y_0) = 0$, where
$$
y_0 = (0, 0, \Re \Omega_0, p_0, q_0, r_0, 0, 0, 0, 0)
,
$$
where we are imposing $\eta_0 = \theta_0 = 0$ because of the intuitive idea that for $\varepsilon = 0$ the fibers shrink and there should be no connections.
We are also imposing that $X_0 = 0$ and $s_0 = u_0 = v_0 = 0$, because these are meant to be free parameters and are not meant to carry geometrical meaning in the limiting case.
Here and in the rest, we denote by $J_0$ the almost complex structure induced by $\Re \Omega_0$.

With this notation and choices, the equation $\Psi(0, y_0) = 0$ reads
\begin{subnumcases}{\label{eq:system0}}
	\frac{1}{3}\omega_0^3 = \frac{1}{2} \Re \Omega_0 \wedge \Im \Omega_0 \\
	d (p_0\Re \Omega_0) = 0\label{eq:system0_dReOmega} \\
	d(r_0 \Re \Omega_0) + d(q_0 \Im \Omega_0) = 0\label{eq:system0_dImOmega} \\
	p_0dq_0 - q_0dp_0 + J_0(r_0dp_0 - p_0dr_0) = 0\label{eq:system0_dr} \\
	d(p_0q_0) = 0\label{eq:system0_ds}
	.
\end{subnumcases}

Equation \eqref{eq:system0_ds} implies that $p_0q_0$ is constant even without any boundedness assumption.
To conclude that $p_0, q_0$ and $r_0$ are constant, the boundedness assumption is crucial.

\begin{lemma}\label{thm:limit_varepsilon_0}
	Suppose $B$ is complete and that $p_0, q_0$ and $r_0$ are bounded.
	Then they are constant and $(\omega_0, \Omega_0)$ is Calabi-Yau.
	Equivalently, the system \eqref{eq:system0} reduces to
	\begin{equation}
		\begin{cases}			
			d p_0 = dq_0 = dr_0 = 0 \\
			d \Re \Omega_0 = d \Im \Omega_0 = 0 \\
			\frac{1}{3}\omega_0^3 = \frac{1}{2} \Re \Omega_0 \wedge \Im \Omega_0
			.
		\end{cases}
	\end{equation}
\end{lemma}
\begin{proof}
	As we noted above, $p_0 q_0$ is constant.
	
	By applying in turn $d {\star_0} p_0^{-2}$ and $d {\star_0} J_0 p_0^{-2}$ on both sides of equation \eqref{eq:system0_dr}, and by using lemma \ref{thm:d_star_J_d}, we get that $q_0p_0^{-1}$ and $r_0p_0^{-1}$ are harmonic and, since they bounded by hypothesis, they have to be constant.
	
	Since we can write $p_0, q_0$ and $r_0$ as functions of $q_0p_0^{-1}$, $q_0p_0$ and $r_0 p_0^{-1}$, $p_0, q_0$ and $r_0$ are also constant.
	Using thus that $d p_0 = dq_0 = dr_0 = 0$, we can deduce from equations \eqref{eq:system0_dReOmega} and \eqref{eq:system0_dImOmega} that $d \Re \Omega_0 = d \Im \Omega_0 = 0$ and thus that $(\omega_0, \Omega_0)$ is Calabi-Yau.
\end{proof}

The hypotheses of lemma \ref{thm:limit_varepsilon_0} are satisfied, in particular, if $B$ is asymptotically conical and $p_0, q_0$ and $r_0$ are in $C^{k, \alpha}_0$.

\subsection{First order solutions}\label{sec:first_order}

In this section, as in the statement of theorem \ref{thm:implicit_function_theorem}, we denote by $\Psi$ the function whose zero locus is given by the equations we intend to solve.
In this subsection we solve manually the first $\varepsilon$-order equations so to take care of the vertical part of $\eta$ and $\theta$ and of the first $\varepsilon$-order part of $\Re \Omega$, which does not decay fast enough to allow us to define $\Psi$ on appropriate spaces, as we are about to see.

As in the previous sections we suppose $M \to B$ is a $T^2$ principal bundle over a manifold admitting an $\SU(3)$-structure.
As we saw in section \ref{sec:T2_connections} we have a splitting
\begin{equation}\label{eq:T2_commutative_diagram}
	\begin{tikzcd}
		& M \arrow[swap]{dl}{p_{M_1}} \arrow{dr}{p_{M_2}} & \\ P_1 \arrow[swap]{dr}{p_{P_1}} & & P_2 \arrow{dl}{p_{P_2}} \\ & B &
	\end{tikzcd}
\end{equation}
We fix a Kähler form $\omega_0$ on $B$ and we require that condition \eqref{eq:topological_condition} hold.
In the splitting of the bundle given in the diagram, this condition is equivalent to
\begin{equation}\label{eq:cohomological_condition}
	c_1(P_1) \smile [\omega_0] = c_1(P_2) \smile [\omega_0] = 0
\end{equation}
which we will see shortly to be a necessary condition for our construction.

To simplify the notation, let $\Re \Omega_1 = \rho, p_1 = P, q_1 = Q, r_1 = R$ and $s_1 = S$.
Then the linearization of system \eqref{eq:parameters}, together with the linearization of the Monge-Ampère equation \eqref{eq:Monge_Ampere}, reads:
\begin{subnumcases}{\label{eq:linearization}}
	d \eta_1 \wedge \omega_0^2 = d \theta_1 \wedge \omega_0^2 = 0 \quad 
	\rho \wedge \Im \Omega_0 + \Re \Omega_0 \wedge \hat{\rho} = 0\label{eq:linearization_SU3} \\
	dP \wedge \Re \Omega_0 + p_0d\rho + d \theta_1 \wedge \omega_0 + d {\star_0} d (u_1\omega) = 0\label{eq:linearization_dReOmega} \\
	d R \wedge \Re \Omega_0 + r_0 d \rho + dQ \wedge \Im \Omega_0 + q_0 d \hat{\rho} + d \eta_1 \wedge \omega_0 + d {\star_0} d (X_1 \interior \Re \Omega + v_1 \omega_0) = 0\label{eq:linearization_dImOmega} \\
	(p_0d\eta_1 - r_0d \theta_1) \wedge \Re \Omega_0 - q_0 d \theta_1 \wedge \Im \Omega_0 + \frac{1}{2}\left(d\left(p_0Q + q_0P\right) + J_0dS\right) \wedge \omega_0^2 = 0\label{eq:linearization_CYm_xi} \\
	(p_0d\eta_1 - r_0d \theta_1) \wedge \Re \Omega_0 + q_0 d \theta_1 \wedge \Im \Omega_0 + \left(d\left(p_0Q-q_0P\right)+J_0d(r_0P - p_0R) \right)\wedge \omega_0^2 = 0\label{eq:linearization_CYm_theta}
\end{subnumcases}

We say that a connection on a $T^k$ bundle is Hermitian Yang Mills if, for one (hence for any) splitting of the bundle into $k$ circle bundles, each component of the connection is Hermitian Yang Mills.

\begin{lemma}\label{thm:first_order_necessary}
	Let $(\eta_1, \theta_1) \in \mathcal{A}(M)$ and let $(d\eta_1, d\theta_1, \rho, P, Q, R, S, u_1, v_1, X_1) \in C^{k, \alpha}_\nu$ be a solution of system \eqref{eq:linearization}, for $k \in \N$, $\alpha \in [0,1]$, $\nu \in (-\infty, 0)$, and $p_0, q_0, r_0, \omega_0$ and $\Re \Omega_0$ are as in the thesis of lemma \ref{thm:limit_varepsilon_0}.

	Then $P, Q, R, S, u_1, v_1, X_1 = 0$, $(\eta_1, \theta_1)$ is a Hermitian Yang-Mills connection and $\rho \in \Omega^3_{12}(B)$.
	Moreover, this data satisfies
	\begin{subnumcases}{\label{eq:rho_linearization}}
		d\rho = -p_0^{-1}d \theta_1 \wedge \omega_0 \\
		d \star\rho = d \left(q_0^{-1}\eta_1-r_0(p_0q_0)^{-1}\theta_1\right) \wedge \omega_0
		.
	\end{subnumcases}
\end{lemma}
\begin{proof}
	First of all, note that it needs to be $\rho \in \Omega^3_{12}(B)$.
	Indeed, $\pi_{1 \oplus 1}(\rho)$ needs to vanish because of equation \eqref{eq:linearization_SU3} and of the linearization of the gauge fixing equation \eqref{eq:diffeo_fixing}.
	Moreover, $\pi_6(\rho)$ needs to vanish because another linear condition we have is $\omega_0 \wedge \Re \Omega = 0$.
	Hence, $\rho \in \Omega^3_{12}(B)$ and thus, by proposition \ref{thm:Hitchin_map_derivative}, $\hat{\rho} = - {\star} \rho$.
	
	Now let us prove that $P, Q, R$ and $S$ vanish.
	By applying $d$ and $d J_0$ to both sides of equations \eqref{eq:linearization_CYm_xi} and \eqref{eq:linearization_CYm_theta} and using lemma \ref{thm:d_star_J_d}, we get that the smooth functions $S$, $q_0 P + p_0Q$, $p_0 R - r_0P$ and $p_0Q - q_0P$ are harmonic and hence, since they decay, they have to vanish.
	By simple linear algebra and the fact that $p_0$ and $q_0$ are strictly positive, $P, Q, R, S$ also have to vanish.
	
	Thus, equations \eqref{eq:linearization_CYm_xi} and \eqref{eq:linearization_CYm_theta} reduce to
	$$
		(p_0d\eta_1 - r_0d \theta_1) \wedge \Re \Omega_0 - q_0 d \theta \wedge \Im \Omega_0 = 0 \qquad
		(p_0d\eta_1 - r_0d \theta_1) \wedge \Re \Omega_0 + q_0 d \theta \wedge \Im \Omega_0 = 0
	$$
	which imply that $(\eta_1, \theta_1)$ is Hermitian Yang-Mills.
	Indeed, by defining $\alpha_\eta$ and $\alpha_\theta$ as in lemma \ref{thm:alphas}, we see that a linear combination of the above equations gives $\alpha_\eta = \alpha_\theta = 0$, thus proving that the $\Lambda^2_6$ part of $d{\eta_1}$ and $d{\theta_1}$ is zero.
	On the other hand, $\pi_1{d\eta_1} = \pi_1{d\theta_1} = 0$ because we imposed it manually in equation \eqref{eq:linearization_SU3}.
	
	Next, let us prove that $u_1, v_1, X_1 = 0$.
	Since $d{\star_0} d (u_1 \omega_0) = d {\star_0} (d u_1 \wedge \omega_0) = -d J_0 d u_1 \wedge \omega_0$, by wedging equation \eqref{eq:linearization_dReOmega} with $\omega_0$ and using that $\omega_0$ is closed, $P = 0$, $\theta_1$ is HYM and $\rho \in \Omega^3_{12}$, we get that $d J_0 d u_1 \wedge \omega_0^2 = 0$ and thus $d{\star_0d} u_1 = 0$ by lemma \ref{thm:Hodge_star_SU3}.1.
	Hence, $u_1$ is harmonic and since it decays it has to vanish.
	With $u_1 = 0$, equation \eqref{eq:linearization_dReOmega} projected onto $\Lambda^4_6$ implies that $\pi_6(d\rho) = 0$, since $\theta_1$ is HYM.
	Thus, by corollary \ref{thm:pi6_d_rho}, $\pi_6(d\star_0 \rho) =  0$.
	Because of this, together with lemma \ref{thm:pi1_6}, and the fact that $\eta_1$ is HYM, we see that $\pi_{1 \oplus 6}(d \star_0 d (v_1 \omega_0 + X_1 \interior \Re \Omega_0)) = 0$, and thus by lemma \ref{thm:pi1_6}, analogously to the proof of theorem \ref{thm:free_parameters}, $v_1$ and $X_1$ vanish.
	
	Finally, since $p_0$ and $q_0$ are both positive, we can manipulate with simple liner algebra equations \eqref{eq:linearization_dReOmega} and \eqref{eq:linearization_dImOmega} so to get system \eqref{eq:rho_linearization}.
\end{proof}

System \eqref{eq:rho_linearization} is elliptic and thus we expect it to have a unique solution on appropriately chosen spaces.

There are possibly many Hermitian Yang Mills connections on $M \to B$, but we can get uniqueness by fixing an asymptotic connection on $B$ and adding a gauge fixing condition analogous to the one discussed in section \ref{sec:gauge}, as shown in the following lemma:
\begin{lemma}\label{thm:HYM_connection}
	Let $M \to B$ be a principal $T^2$ bundle on an AC Calabi Yau manifold $B$, splitting as in diagram \eqref{eq:T2_commutative_diagram}.

	Then, for any principal connection $\left(\bar{\eta}_\infty, \bar{\theta}_\infty \right)$ on $M$ such that $d \bar{\eta}_\infty, d\bar{\theta}_\infty \in C^\infty_{-2}(\Lambda^2(B))$, there exists a unique HYM connection $\left(\eta_1, \theta_1\right)$ with $(\eta_1, \theta_1) - \left(\bar{\eta}_\infty, \bar{\theta}_\infty \right) \in C^\infty_{-1}\left(\Lambda^1(B)^2\right)$ such that $d^* \left(\eta_1-\bar{\eta}_\infty \right) = d^* \left(\theta_1-\bar{\theta}_\infty \right) = 0$.

	Moreover, suppose that $(B, \omega_0)$ has a connected symmetry group $G$.
	Then, the curvature forms $(d\eta_1, d\theta_1)$ are also $G$-invariant.
\end{lemma}
\begin{proof}
	We do this separately for each circle bundle.
	By theorem \ref{thm:harmonic_2_forms_isomorphism}, for $\nu \in (-2, 0)$, the natural map $\mathcal{H}_\nu^2(B) \rightarrow H^2(B)$ is an isomorphism.
	Thus, we can represent $c_1(P_1) = [\kappa]$ by a unique closed and coclosed 2-form $\kappa$ with decay rate $\nu \in (-2, 0)$.
	By uniqueness $\kappa \in \bigcap_{\nu > -2}C^\infty_\nu(\Lambda^2(B))$, which is not quite $C^\infty_{-2}(\Lambda^2(B))$ because it contains sections that decay like $r^{-2} + a\log^b(r)$, but we actually have $\kappa \in C^\infty_{-2}(\Lambda^2(B))$ because of proposition B.12 of \cite{Foscolo2021}.
	Hence, $\kappa - d\bar{\eta}_\infty \in C^\infty_{-2}\left(\Lambda^2(B)\right)$.
	By lemma \ref{thm:exact_2_forms_AC_CY}, there is a unique $\delta \in C^\infty_{-1}\left(\Lambda^1(B)\right)$ such that $d \delta = \kappa - d\bar{\eta}_\infty$  and $d^*\delta = 0$.
	Let $\eta_1 = \bar{\eta}_\infty + \delta$.
	It is clear that $d \eta_1 = \kappa$.
	Since $d \eta_1$ is harmonic and decays, by lemma \ref{thm:harmonic_1_6_forms}, it cannot have any component in $\Omega^2_{1\oplus6}$ and thus it is HYM.
	
	The proof for $\theta_1$ and $P_2$ is analogous.
	
	As for the symmetry statement, a connected group acts trivially on the cohomology classes, by homotopy invariance of cohomology.
	Moreover, the map $\mathcal{H}_\nu^2(B) \rightarrow H^2(B)$ is equivariant by definition of the action of a group on cohomology.
	Thus, also the $G$ action on $\mathcal{H}_\nu^2(B)$ is trivial, proving our statement.
\end{proof}

\begin{remark}
	Since $\eta_1$ and $\theta_1$ are HYM, they satisfy
	\begin{equation}\label{eq:HYM_connection_linear}
		d\eta_1 \wedge \omega_0^2 =d\theta_1 \wedge \omega_0^2 = 0
		.
	\end{equation}
\end{remark}

\noindent Now that we found $(\eta_1, \theta_1)$, we can solve system \eqref{eq:rho_linearization} for $\rho$.

\begin{theorem}\label{thm:rho_first_order}
	Let $M \to B$ be a principal $T^2$ bundle on an AC Calabi Yau manifold $B$, splitting as in diagram \eqref{eq:T2_commutative_diagram}, and suppose the cohomological condition \eqref{eq:cohomological_condition} holds.
	Fix $\nu \in (-1, 0)$.
	
	Then there exists a unique solution $\rho \in C_{-1}^{\infty}\left(\Lambda_{12}^3(B)\right) \cap \mathcal{W}_\nu^3(B)$ of \eqref{eq:rho_linearization}.
\end{theorem}
\begin{proof}
	System \eqref{eq:rho_linearization} is equivalent to the equation
	\begin{equation}\label{eq:Hodge_deRham}
		(d + d^*) \rho = p_0^{-1}\star_0 d \theta_1 +\left(-q_0^{-1}d\eta_1+r_0(p_0q_0)^{-1}d\theta_1\right) := \tau
		.
	\end{equation}
	Since $d + d^* : \Omega^{\text{odd}} \to \Omega^{\text{even}}$ is an admissible operator of order 1, by corollary B.10 in \cite{Foscolo2021}, equation \eqref{eq:Hodge_deRham} has a $C^\infty_{-1+\mu}$ solution $\rho$ if and only if
	$$
	\langle \tau, \sigma \rangle_{L^2} = 0
	$$
	for any $\sigma \in \mathcal{H}^{\text{even}}_{-5-\nu}$.
	Since elements of $\mathcal{H}^{\text{even}}_{-5-\nu}$ are square integrable, their pure degree components must be individually closed and coclosed.
	Indeed, one such element $\sigma$ satisfies $\langle d\sigma, d \sigma\rangle = 0$ which implies that the $L^2$-norm of the pure components of $d\sigma$ are null and similarly for $d^* \sigma$.
	Thus, for the above equality to hold, it is sufficient to prove that
	$$
	\forall \sigma \in \mathcal{H}^2_{-5-\nu} \quad \langle d \eta_1, \sigma \rangle_{L^2} = \langle d \theta_1, \sigma \rangle_{L^2} = 0
	$$
	where we used that ${\star_0}$ induces an isomorphism between harmonic forms of pure degree 2 and 4, and that $p_0, q_0$ and $r_0$ are all constant.
	This last equation is true because of lemma \ref{thm:exact_star_H4}.
	
	Clearly, any other solution to equation \eqref{eq:Hodge_deRham} differs from a fixed solution $\rho$ by an element in the kernel of $d + d^*$, i.e.\ a closed and coclosed 3-form in $\mathcal{H}^3_\nu(B)$.
	Hence, we can require that $\rho$ be unique by imposing $\rho \in \mathcal{W}^3_{\nu}(B)$.
	Moreover, $\rho$ is harmonic because
	$$
	\Delta \rho = (d + d^*)^2 \rho =(d+d^*) \tau = p^{-1}_0 d{\star_0} d \theta +\left(q_0^{-1}d^*d\eta_1+r_0(p_0q_0)^{-1}d^*d\theta_1\right) = 0
	$$
	where the last equality holds because both $d \eta_1$ and $d\theta_1$ are of type $\Omega^2_8$, and thus $d{\star_0} d \eta_1 = -d (d\eta_1 \wedge \omega_0) = 0$
	by closedness of $\omega_0$ and similarly for $\theta$.
	
	$\rho$ is in $\Omega^3_{12}$ because its components are individually harmonic and thus lemma \ref{thm:harmonic_1_6_forms} applies.
	
	The reason why $\rho$ is actually $C^\infty_{-1}$ is entirely analogous to what is explained in the proof of lemma 6.3 of \cite{Foscolo2021}.
\end{proof}

\subsection{The setting of the implicit function theorem}\label{sec:implicit_function_theorem}

In this subsection we set up the implicit function theorem to build a solution of equations \eqref{eq:parameters_linear_conditions} and \eqref{eq:parameters} together with the gauge fixing conditions \eqref{eq:diff_fixing}, \eqref{eq:gauge_fixing} and \eqref{eq:CY_deformations_fixing}.
Let us keep in mind the notation of theorem \ref{thm:implicit_function_theorem}.

In our case $X = \R$.
As for $Y$, fix $l \in \N$, $\alpha \in (0,1)$ and a connection $(\bar{\eta}_\infty, \bar{\theta}_\infty)$ and consider first the affine subspace $\tilde{Y}$ of $C^{l, \alpha}_0\left(\Lambda^1(M)^2 \times \Lambda^3(B) \times \Lambda^0(B) \times TB \right)$ cut out by the following affine constraints: for
$$
	(\eta, \theta, \Re \Omega, p, q, r, s, u, v, X) \in C^{l, \alpha}_0\left(\Lambda(B)^2 \times \Lambda^3(B) \times \Lambda^0(B)^6 \times TB\right)
$$
we require that
\begin{subnumcases}{\label{eq:linear_constraints}}
	d \eta \wedge \omega_0^2 = 0 \quad d \theta \wedge \omega_0^2 = 0\label{eq:alphas_1} \\
	d^*(\eta - \bar{\eta}_\infty) = 0 \quad d^* (\theta - \bar{\theta}_\infty) = 0\label{eq:gauge_fixing_alphas} \\
	\Re \Omega \wedge \omega_0 = 0\label{eq:gauge_fixing_Omega_omega} \\
	\Re \Omega \wedge \Re \Omega_0 = 0\label{eq:gauge_fixing_Omega_Omega} \\
	\pi_{12}(\Re \Omega - \Re \Omega_0) \in \mathcal{W}^3_\nu(B)\label{eq:CY_deformations_fixing_rho}
\end{subnumcases}
where equations \eqref{eq:gauge_fixing_alphas}, \eqref{eq:gauge_fixing_Omega_Omega} and \eqref{eq:CY_deformations_fixing_rho} are the gauge fixing equations \eqref{eq:gauge_fixing}, \eqref{eq:diffeo_fixing} and \eqref{eq:CY_deformations_fixing}, \eqref{eq:gauge_fixing_Omega_omega} is a part of the condition that $(\Re \Omega, \omega_0)$ be an $\SU(3)$-structure and \eqref{eq:alphas_1} together with \eqref{eq:HYM_connection_linear} takes care that the curvatures have no $\Lambda^2_1$ component, which is required in order to apply proposition \ref{thm:free_parameters}.

Moreover, consider the open set of $\tilde{Y}$
$$
V_0 = \tilde{Y} \cap C^{l, \alpha}_0\left(\Lambda^1(M)^2 \times \Lambda_{\neq 0}^3(B) \times \Lambda_+^0(B)^2 \times \Lambda^0(B)^4 \times TB\right)
$$
where $\Lambda_+^0(B) = B \times \R^+$.
Take $\nu \in (-2, -1)$, $y_0 \in V_0$ and take
$$
	Y = \tilde{Y} \cap \left(y_0 + C^{l, \alpha}_\nu\left(\Lambda^1(B)^2 \times \Lambda^3(B) \times \Lambda^0(B)^6\right) \times C^{l + 1, \alpha}_{\nu + 1} \left(TB\right) \right) \qquad V = V_0 \cap Y
	.
$$
which has a natural structure of a Banach manifold, being an open set of an affine space over a Banach space.
The reason why we build $Y$ in this way instead of taking $C^{l, \alpha}_\nu$ in the first place, is that, in order for our argument to work, the base point $(x_0, y_0)$ does not need to be given by decaying sections and it is clear that in practice $y_0$ will not be in $C^{l, \alpha}_\nu$.
However, we can only perturb it with decaying sections, since the decaying hypothesis is crucial to be able to do analysis.
Note that we already used this hypothesis in proposition \ref{thm:free_parameters}.

In our setting, $Z$ is the subspace of $C^{l, \alpha}_{\nu}\left(\Lambda^6(B)\right) \times C^{l-1, \alpha}_{\nu-1} \left( \Lambda^4(B)^2 \times \Lambda^5(M)^2\right)$ cut out by the linear constraints that $z_1$ and $z_2$ be exact.
Finally, $\Psi$ is defined to be the function that maps the vector
$$
	(\varepsilon, \delta_\eta, \delta_\theta, \Re \tilde{\Omega}, p, q, r, s, u, v, X)
$$
to
$$
	\begin{bmatrix}
		\frac{1}{4}\Re \Omega \wedge \Im \Omega - \frac{1}{6} \omega^3_0 \\[5pt]
		d\left(p \Re \tilde{\Omega} + \varepsilon(p-  p_0) \Re \Omega_1\right) + d\delta_\theta \wedge \omega_0 -  d {\star} d (u \omega_0) \\[5pt]
		d\left(r \Re \tilde{\Omega} + \varepsilon(r-  r_0) \Re \Omega_1\right) + d\left(q \Im \tilde{\Omega} + \varepsilon(q-  q_0) \Re \Omega_1\right) + d\delta_\eta \wedge \omega_0 -  d {\star} d ( X \interior \Re \Omega + v \omega_0) \\[5pt]
		d\left(\varepsilon \eta_1 + \delta_\eta\right) \wedge (p \Re \Omega) + d\left(\varepsilon \theta_1 + \delta_\theta\right)\wedge \left( q\Im \Omega- r \Re \Omega\right) + \left(pdq - qdp + J(rdp - pdr)\right) \wedge \omega_0^2 \\[5pt]
		d\left(\varepsilon \eta_1 + \delta_\eta\right) \wedge (p \Re \Omega) - d\left(\varepsilon \theta_1 + \delta_\theta\right)\wedge \left( q \Im \Omega + r \Re \Omega\right) + \frac{1}{2}\left(d(pq) + Jds\right)\wedge \omega_0^2
	\end{bmatrix}
$$

\noindent where $J$ is the complex structure induced by $\Re \Omega$.
We denote its components by $\Psi_0, \ldots, \Psi_4$.
Note that in the definition of the above function, $\Re \Omega = \Re \tilde{\Omega} + \varepsilon \Re \Omega_1$, i.e.\ we are isolating the first $\varepsilon$-order of $\Re \Omega$, since we took care of it separately.
$\Psi_0, \Psi_3$ and $\Psi_4$ need no explanation whereas $\Psi_1$ and $\Psi_2$ come from equations \eqref{eq:parameters_dRe_Omega} and \eqref{eq:parameters_dIm_Omega} by subtracting system \eqref{eq:rho_linearization} appropriately inverted.
Note that $\Psi_1$ and $\Psi_2$ are manifestly exact.
Since $d\eta_1$, $d\theta_1$ and $d\Omega_1$ have decay rate $-2$ but the last 4 components of $\Psi$ are required to decay with rate $< -2$ we need to check that they do decay fast enough.
This is the case because, thanks to the fact that we subtracted the first $\varepsilon$-order solutions from $\Psi_1$ and $\Psi_2$, $\Re \Omega_1$ and $\Im \Omega_1$ are multiplied by terms of the kind $p - p_0$, $q-q_0$ or $r-r_0$ which decay with rate $\nu$ by definition of $\Psi$, so their product decays with rate $\nu -1$ which is indeed fast enough.
Similarly, for $\Psi_3$ and $\Psi_4$, $\eta_1$, $\theta_1$, $\Re \Omega_1$ and $\Im \Omega_1$ only appeared in $\wedge$ products with other decaying sections, so $\Psi_3$ and $\Psi_4$ also belong to the claimed spaces of decaying sections.
Indeed, the only terms that do not seem to be decaying fast enough are $d \eta_1 \wedge \Re \Omega_0$, $d \theta_1 \wedge \Re \Omega_0$ and $d \theta_1 \wedge \Im \Omega_0$, but these vanish because $\eta_1$ and $\theta_1$ are Hermitian Yang Mills, as proved in lemma \ref{thm:HYM_connection}.



\subsection{The linearization}\label{sec:linearization}
In this subsection we deal with the main ingredient needed to make use of the implicit function theorem: the invertibility of the linearization.
To do so, we calculate the derivative of $\Psi$ in 
$$
	(0, 0, 0, \Re \Omega_0, p_0, q_0, r_0, 0, 0, 0, 0)
$$
where $(\omega_0, \Re \Omega_0)$ is Calabi-Yau and $p_0, q_0$ and $r_0$ are constant.
A simple calculation shows that $D \Psi_{(x_0, y_0)}$ maps
$$
 	(\varepsilon, \gamma_\eta, \gamma_\theta, \rho, P, Q, R, S, U, V, X_1)
$$
to
\begin{equation*}
	\begin{bmatrix}
		\operatorname{Re} \Omega_0 \wedge \hat{\rho}+\rho \wedge \operatorname{Im} \Omega_0 \\[4pt]
		d(P\Re \Omega_0 + p_0\rho) + d (\varepsilon \theta_1 + \gamma_\theta) \wedge \omega_0 -  d {\star} d (U \omega_0) \\[4pt]
		d R \wedge \Re \Omega_0 + r_0 d \rho + dQ \wedge \Im \Omega_0 + q_0 d \hat{\rho} + d (\varepsilon \eta_1 + \gamma_\eta) \wedge \omega_0 -  d {\star} d (X_1 \interior \Re \Omega_0 + V \omega_0) \\[4pt]
		(p_0d\eta - r_0d \theta) \wedge \Re \Omega_0 - q_0 d \theta \wedge \Im \Omega_0 + \frac{1}{2}\left(d\left(p_0Q + q_0P\right) + J_0dS\right) \wedge \omega_0^2 \\[4pt]
		(p_0d\eta - r_0d \theta) \wedge \Re \Omega_0 + q_0 d \theta \wedge \Im \Omega_0 + \left(d\left(p_0Q-q_0P\right)+J_0d(r_0P - p_0R) \right)\wedge \omega_0^2	
	\end{bmatrix}
\end{equation*}
where $\eta = \varepsilon \eta_1 + \gamma_\eta$ and $\theta = \varepsilon \theta_1 + \gamma_\theta$.

\begin{theorem}\label{thm:linearization}
	Let $\left(B, \omega_0, \Omega_0\right)$ be an AC Calabi-Yau 3-fold.
	Fix $k \in \N^+, \alpha \in(0,1), \delta \in (0, + \infty)$ and $\nu \in(-3-\delta,-1)$ away from a discrete set of indicial roots of $d+d^*$.
	
	Then the map $y \mapsto D\Psi_{x_0, y_0}(0, y)$ is an isomorphism of Banach spaces.
\end{theorem}
\begin{proof}
	Let $z_0$ be a 6-form in $C_\nu^{k, \alpha}$, $z_1=d \beta_1$ and $z_2=d \beta_2$ exact 4-forms with $\beta_1, \beta_2 \in C_\nu^{k, \alpha}$, and $z_3, z_4$ 5-forms in $C_{\nu-1}^{k-1, \alpha}$.
	Then proving bijectivity amounts to proving that there exist unique functions $P, Q, R$, 1-forms $\eta, \theta$ all in $C_\nu^{k, \alpha}$, a $3$-form $\rho$ in $C_\nu^{k, \alpha}$, and functions $U, V$ and a vector field $X_1$ in $C_{\nu+1}^{k+1, \alpha}$ such that
	\begin{subnumcases}{\label{eq:final_linearization}}
		d^* \eta = d^* \theta = 0 \qquad d \eta \wedge \omega^2_0 = d \theta \wedge \omega_0^2 = 0\label{eq:step_preamble_} \\
		\rho \wedge \omega_0 = 0 \qquad \rho \wedge \Re \Omega_0 = 0 \qquad \pi_{12}(\rho) \in \mathcal{W}^3_\nu\label{eq:step_preamble_rho} \\
		\rho \wedge \Im \Omega_0 + \Re \Omega_0 \wedge \hat{\rho} = z_0\label{eq:step_SU3_} \\
		dP \wedge \Re \Omega_0 + p_0d\rho + d \theta \wedge \omega_0 + d {\star_0} d (u\omega) = z_1\label{eq:step_dReOmega_} \\
		d R \wedge \Re \Omega_0 + r_0 d \rho + dQ \wedge \Im \Omega_0 + q_0 d \hat{\rho} + d \eta \wedge \omega_0 + d {\star_0} d (X_1 \interior \Re \Omega + v \omega_0) = z_2\label{eq:step_dImOmega_} \\
		(p_0d\eta - r_0d \theta) \wedge \Re \Omega_0 - q_0 d \theta \wedge \Im \Omega_0 + \frac{1}{2}\left(d\left(p_0Q + q_0P\right) + J_0dS\right) \wedge \omega_0^2 = z_3\label{eq:step_CYm_xi_} \\
		(p_0d\eta - r_0d \theta) \wedge \Re \Omega_0 + q_0 d \theta \wedge \Im \Omega_0 + \left(d\left(p_0Q-q_0P\right)+J_0d(r_0P - p_0R) \right)\wedge \omega_0^2 = z_4\label{eq:step_CYm_theta_}
	\end{subnumcases}
	Equations \eqref{eq:step_CYm_xi_}  and \eqref{eq:step_CYm_theta_} can be rewritten equivalently as
	\begin{subnumcases}{\label{eq:linearization_subsystem}}
		\star_0 \left((p_0d\eta - r_0 d \theta) \wedge \Re \Omega_0\right) + d\left(r_0P - p_0 R + S \right) + \frac{1}{2}J_0d\left(q_0P - 3p_0Q\right) = {\star_0} \frac{1}{2} (z_3+z_4) \\
		\star_0 \left(q_0 d \theta \wedge \Re \Omega_0\right) + \frac{1}{2}d(p_0Q - 3q_0 P) + J_0 d(r_0 P - p_0 R - S) = {\star_0} \frac{1}{2} J_0 (z_4 - z_3)\label{eq:step_G2_monopole}
	\end{subnumcases}
	which we write more concisely as
	\begin{equation}\label{eq:pre_Dirac}
		\begin{cases}
			\star_0 \left(d\xi_2 \wedge \Re \Omega_0\right) + dg + J_0 dh = w_3 \\
			\star_0 \left(d\xi_1 \wedge \Re \Omega_0\right) + df + J_0 dt = w_4
		\end{cases}
	\end{equation}
	respectively, where $f, g, h, t, \xi_i$ and $w_i$ are implicitly defined.
	
	It is immediate to see that equations \eqref{eq:step_preamble_} are equivalent to $d\xi_1 \wedge \omega_0^2 = d \xi_2 \wedge \omega_0^2 = 0$ and $d^*\xi_1 = d^* \xi_2 = 0$.	
	With such conditions, the system \eqref{eq:pre_Dirac} is equivalent to
	\begin{equation}\label{eq:Dirac_equivalent}
		\begin{cases}
			\Dirac(g, -h, \xi_2) = (0, 0, w_3) \\
			\Dirac(f, -t, \xi_1) = (0, 0, w_4)
			.
		\end{cases}
	\end{equation}
	By lemma \ref{thm:Dirac_iso} both equations have unique solutions.
	
	It is clear that $\xi_1, \xi_2$ and $f, g, h, t$ determine $\eta, \theta$ and $P, Q, R, S$.
	Indeed, it is a simple calculation to show that
	\begin{equation}\label{eq:explicit_expressions}
		\begin{cases}
			P = -\frac{1}{4q_0}(h+3f) \qquad
			Q = -\frac{1}{4p_0}(3h+f) \qquad
			S = \frac{1}{2}(g-t) \qquad \\
			R = -\frac{1}{p_0} \left(\frac{r_0}{4q_0}(h+3f)+\frac{1}{2}(g+t)\right) \qquad
			\eta = \frac{1}{p_0} \left( \xi_2 +\frac{r_0}{q_0} \xi_1 \right) \qquad
			\theta = \frac{1}{q_0} \xi_1
			.
		\end{cases}
	\end{equation}
	and thus, we are given $\eta, \theta, P, Q, R, S$ that solve equations \eqref{eq:step_preamble_}, \eqref{eq:step_CYm_theta_} and \eqref{eq:step_CYm_xi_}.
	Hence, we just need to find $\rho, u, v, X$ that solve equations \eqref{eq:step_dReOmega_} and \eqref{eq:step_dImOmega_}.
	
	Equations \eqref{eq:step_preamble_rho} force $\rho = f \Re \Omega_0 + \rho_{12}$ for some function $f$ and $\rho_{12}$ of type $\Omega^3_{12}$ with $\pi_{12}(\rho) \in \mathcal{W}^3_\nu$.
	Equation \eqref{eq:step_SU3_} forces $f = \frac{1}{8}\star_0 z_0$.
	
	Now, by corollary \ref{thm:exact_4_forms'} we can write
	\begin{multline*}
		z_1-\frac{1}{8} p_0 d\left((\star_0z_0) \operatorname{Re} \Omega_0\right)-dP \wedge \operatorname{Re} \Omega_0-d \theta \wedge \omega_0 = \\ d * d\left(u_1 \omega_0\right)+d\left(* d\left(Y_1\interior \operatorname{Re} \Omega_0\right)-d\left(Y_1\interior \operatorname{Im} \Omega_0\right)\right)+d \rho_0
	\end{multline*}
	for unique $\left(u_1, Y_1, \rho_0\right)$ with $\rho_0 \in \Omega_{12}^3 \cap \mathcal{W}_\nu^3$ and $d^* \rho_0=0$.
	Moreover, the $C_\nu^{k, \alpha}$-norm of $\rho_0$ and the $C_{\nu+1}^{k+1, \alpha}$-norm of $\left(u_1, Y_1\right)$ are uniformly controlled by $\left\|\left(z_0, \beta_1\right)\right\|_{C_\nu^{k, \alpha}}$ and $\|(P, Q, R, \eta, \theta)\|_{C_\nu^{k, \alpha}}$.
	
	We set $U = u_1$.
	Then, equation \eqref{eq:step_dReOmega_} forces 
	$$
	\rho_{12} = * d\left(Y_1\interior \operatorname{Re} \Omega_0\right)-d\left(Y_1\interior \operatorname{Im} \Omega_0\right)+\rho_0+* \rho_0^{\prime}
	$$
	for some $\rho_0^{\prime} \in \Omega_{12}^3 \cap \mathcal{W}_\nu^3$ with $d^* \rho_0^{\prime}=0$.
	Indeed, proposition \ref{thm:Omega_3_12} implies that the sum of the first and second term lies in $\Omega_{12}^3$.
	Using lemma \ref{thm:Hitchin_map_derivative} we calculate
	$$
	\hat{\rho}=\frac{1}{8}(\star_0z_0) \operatorname{Im} \Omega_0+d\left(Y_1\interior \operatorname{Re} \Omega_0\right)+* d\left(Y_1\interior \operatorname{Im} \Omega_0\right)-* \rho_0+\rho_0^{\prime}
	.
	$$
	By proposition \ref{thm:exact_4_forms} we furthermore write in a unique way
	\begin{multline*}
		z_2- dR \wedge \Re \Omega_0 - d Q \wedge \operatorname{Im} \Omega_0-r_0d\left(\frac{1}{8}(\star_0z_0) \operatorname{Im} \Omega_0+* d\left(Y_1\interior \operatorname{Re} \Omega_0\right)+\rho_0 \right) + \\ - q_0d\left(\frac{1}{8}(\star_0z_0) \operatorname{Im} \Omega_0+* d\left(Y_1\interior \operatorname{Im} \Omega_0\right)\right)= d * d\left(Y_2\interior \operatorname{Re} \Omega_0+u_2 \omega_0\right)+d \rho_0^{\prime \prime}
	\end{multline*}
	with $d^* \rho_0^{\prime \prime}=0$ and $\rho_0^{\prime \prime} \in \Omega_{12}^3 \cap \mathcal{W}_\nu^3$.
	We then set $V=u_2, X_1=Y_2$ and $\rho_0^{\prime}=\rho_0^{\prime \prime}$.
	
	It remains to prove continuity (continuity of the inverse is guaranteed, as usual, by the open mapping theorem).
	We are going to check continuity separately for each component.
	
	Continuity of $\Psi_0$ is implied by continuity of the linearization of the Hitchin map and continuity of the wedge product.
	Continuity of $\Psi_1$ and $\Psi_2$ follows directly from proposition \ref{thm:Dirac_iso} and the fact that the explicit expressions in the system \eqref{eq:explicit_expressions} are obviously continuous.
	Finally, continuity of $\Psi_3$ follows from corollary \ref{thm:exact_4_forms'} and that of $\Psi_4$ from proposition \ref{thm:exact_4_forms}.
\end{proof}

	\begin{remark}
		Equation \eqref{eq:step_G2_monopole} has a deeper geometrical interpretation: it comes from the linearization of the equation
		\begin{equation}\label{eq:CY_monopole_2HF}
			\left(d\eta - \frac{r}{p} d\theta \right) \wedge \Re \Omega = - {\star} \left(\frac{3}{2}p^{\frac{1}{3}}Jd\left(qp^{-\frac{1}{3}}\right) + pd\left(\frac{r}{p}\right)\right)
		\end{equation}
		which is an equation that comes from the dimensional reduction of the $\U(1)$ $G_2$-monopole equation
		\begin{equation}\label{eq:G2_monopole_type}
			d\left(\eta - \frac{r}{p} \theta \right)\wedge \star_\varphi \varphi = -{\star_\varphi} d\left(\frac{3}{2}qp^{-\frac{1}{3}}\right)
			.
		\end{equation}
	
		Indeed, we could derive system \eqref{eq:Spin7_torsion_v1} in two steps: we could consider the $T^2$-bundle $P$ as the product of two circle bundles and we could find a $G_2$-structure $\varphi$ on one of the two bundles and express the $\Spin(7)$-structure $\Phi$ in terms of $\varphi$.
		This step would produce some equations prescribing the torsion of $\phi$ (in the same way we got equations prescribing the torsion of the $\SU(3)$-structure $(\omega, \Re \Omega)$ in section \ref{sec:equations}) which would imply equation \eqref{eq:G2_monopole_type} which is that of a $\U(1)$-invariant $\U(1)$ $G_2$-monopole.
		Then, as a second step, we could study the dimensional reduction of $\U(1)$-invariant $\U(1)$ $G_2$-monopoles on the total spaces of circle bundles over Calabi Yau 3-folds, and we would find that equation \eqref{eq:G2_monopole_type} is equivalent to a system of two equations expressing a Calabi Yau monopole type equation with two Higgs fields.
		This system would be a system in terms of the $\U(1)$-invariant connection $\xi = \eta - \frac{r}{p} \theta$, and it would be made by equation \eqref{eq:CY_monopole_2HF} and the equation $\xi \wedge \omega^2 = 0$.
	\end{remark}

\subsection{The final result}

We summarize the content of this section in the following result.
The existence part of theorem \ref{thm:main_theorem} follows.

\begin{theorem}\label{thm:final_result}
	Let $\left(B, \omega_0, \Omega_0\right)$ be a simply connected AC Calabi-Yau 3-fold, let $p_0, q_0, r_0 \in \R$ with $p_0, q_0 > 0$ and let $M \rightarrow B$ be a non-trivial principal $T^2$-bundle that satisfies condition \eqref{eq:topological_condition}.
	Then system \eqref{eq:Spin7_torsion_v1} has an analytic curve of solutions
	$$
		(0, \varepsilon_0) \to \mathcal{A}(P_1) \times \mathcal{A}(P_2) \times C^\infty_0(\Lambda^2(B) \times \Lambda^3(B) \times \Lambda^0_+(B)^2 \times \Lambda^0(B))
	$$
	of the kind:
	\begin{equation}\label{eq:epsilon_split}
	\varepsilon \mapsto
		\begin{bmatrix}
			\eta = \varepsilon{\eta}_1 + \varepsilon^2 \eta_\varepsilon \\
			\theta = \varepsilon{\theta}_1 + \varepsilon^2 \theta_\varepsilon \\
			\omega = \omega_0 \\
			\Re \Omega = \Re \Omega_0 + \varepsilon \Re \Omega_\varepsilon \\
			p = p_0 + \varepsilon p_\varepsilon \\
			q = q_0 + \varepsilon q_\varepsilon \\
			r = r_0 + \varepsilon r_\varepsilon
		\end{bmatrix}
	\end{equation}
	where $(\omega_0, \Re \Omega_0)$ is CY, $ {\eta}_1$ and $ {\theta}_1$ are Hermitian Yang Mills connections on the circle bundles, and $p_0, q_0$ and $r_0$ are constant.
	
	Here, $\Re \Omega_\varepsilon, \eta_\varepsilon, \theta_\varepsilon, p_\varepsilon, q_\varepsilon, r_\varepsilon$ are also analytic functions of $\varepsilon$ and are all in $C^\infty_{-1}$.
	
	For each choice of fixed connection $(\bar{\eta}_\infty, \bar{\theta}_\infty)$ with curvature decaying with rate $-2$, this curve can be chosen to be the unique curve of solution in a neighborhood of $y_0 = (0, 0, \Re \Omega_0, p_0, q_0, r_0)$ with the property that it tends to $y_0$ as $\varepsilon \to 0$, and that, for some $\nu \in (-1, 0)$,
	\begin{equation}
		d^* \left(\eta - \bar{\eta}_\infty\right) = d^* \left(\theta - \bar{\theta}_\infty\right) = 0 \qquad \Re \Omega_\varepsilon \wedge \Re \Omega_0 = 0 \qquad \pi_{12}\left(\Re \Omega_\varepsilon\right) \in \mathcal{W}^3_{\nu}
		.
	\end{equation}
	Equivalently, any other curve of solutions to system \eqref{eq:Spin7_torsion_v1} can be obtained by twisting the curve given by \eqref{eq:epsilon_split} by either a diffeomorphism on $B$ respecting the AC structure, a gauge transformation to $M$ that decays at infinity, or by deforming the Calabi Yau structure on $B$.
\end{theorem}
\begin{proof}
	Proposition \ref{thm:free_parameters} adds free parameter to system \eqref{eq:Spin7_torsion_v1} to get rid of redundancy in the equations, yielding system \eqref{eq:parameters}.
	Theorem \ref{thm:HYM_connection} allows us to fix a choice of HYM connection $(\eta_1, \theta_1)$ that makes the function $\Psi$ defined at the end of section \ref{sec:implicit_function_theorem} well-defined.
	Theorem \ref{thm:linearization} allows us to apply the implicit function theorem \ref{thm:implicit_function_theorem} to such function and get solutions in $C^{l, \alpha}_{-1}$.
	However, we have that $l$ is arbitrary and the solutions for each $l$ are unique on a small enough neighborhood.
	Moreover, $C^{l, \alpha}_{-1} \leq C^{m, \alpha}_{-1}$ for $l \geq m$, so, by restricting $\varepsilon_0$ if necessary, we get solutions in $C^{\infty}_{-1}$.

	Because the solution is analytic in $\varepsilon$, we can Taylor expand the solution $\Re \Omega$ as
	$$
		\Re \Omega_0 + \varepsilon \sum_{k = 1}^{+\infty} \varepsilon^{k-1} \Re \Omega_k = \Re \Omega_0 + \varepsilon \Re \Omega_\varepsilon
	$$
	where $\Re \Omega_\varepsilon$ is implicitly defined.
	We do the same for $p,q, r, \eta, \theta$.
	Note that $\eta = \eta_1 + \varepsilon \eta_\varepsilon$ is still a connection because $ \eta_\varepsilon$ is horizontal.
	The same holds for $\theta$.
	
	All the mentioned properties in the statement follow from sections \ref{sec:implicit_function_theorem} and \ref{sec:order_0}.
	
\end{proof}

In the above result, we do not consider negative values for $\varepsilon$ because for such values $\eta$ and $\theta$ are not connections, as they do not reproduce the generators of the $T^2$-action.

\begin{remark}
We note that in case the starting Calabi Yau manifold $(B, \omega_0, \Omega_0)$ has a connected symmetry group $G$, by running the implicit function theorem on $G$-invariant Banach spaces, we get that the tuple $(\omega, \Omega, p, q, r, d \eta, d \theta)$ produced by theorem \ref{thm:final_result} is also $G$-invariant.
\end{remark}

%

\section{Properties of the new \texorpdfstring{$\Spin(7)$}{Spin(7)} manifolds}\label{sec:properties}

In this section we prove the main properties that make interesting the manifolds built in section \ref{sec:analytic_curve}.
This section completes the proof of theorem \ref{thm:main_theorem}.

\subsection{Asymptotic behavior}

\begin{theorem}\label{thm:AT2C}
	Let $\left(B, \omega_0, \Omega_0\right)$ be a simply connected non-trivial AC Calabi-Yau 3-fold asymptotic with rate $\nu < 0$ to the cone $(C, g_C)$, let $p_0, q_0, r_0 \in \R$ with $p_0, q_0 > 0$ and let $M \rightarrow B$ be a principal $T^2$-bundle that satisfies \eqref{eq:topological_condition}.
	
	Then, the manifolds $(M, g_{\Phi_\varepsilon})$ constructed in theorem \ref{thm:final_result} out of $\left(B, \omega_0, \Omega_0, M, p_0, q_0, r_0\right)$ are asymptotically $T^2$-fibred conical with rate $-1$, as defined in \ref{def:ATkC}.
\end{theorem}
\begin{proof}
	By hypothesis $\left(B, g_{\omega_0, \Omega_0}\right)$ is AC with model cone $(C = C_R(\Sigma), g_C)$ and isomorphism $\Psi$, where $\Sigma$ is Sasaki-Einstein (see definition \ref{def:AC}).
	It is clear that $\left(B, g_{\omega_\varepsilon, \Omega_\varepsilon}\right)$ is also AC, because adding a something $C^\infty_{-1}$ to an AC metric leaves it AC and in theorem \ref{thm:final_result} we saw that $\omega_\varepsilon, \Omega_\varepsilon$ are in $C^\infty_{-1}$.
	
	Call $\pi_C : C \to \Sigma$ the canonical radial projection.
	We define the asymptotic model for $(M, g_{\Phi_\varepsilon})$ to be $P_{T^2} = \pi_C^* \dot{P}$ where $\dot{P} = \iota_r^* \Psi^* M$ is a bundle on $\Sigma$ and $\iota_r :  \Sigma \to C$ is the inclusion of $\Sigma$ in $C$ at height $r$ with $r \in \R^+$ is such that $r > R$.
	Note that we also have a canonical radial projection $P_{T^2} \to \dot{P}$.
	Clearly $P_{T^2}$ and $M|_{B \setminus K}$, i.e.\ the restriction of $M$ to the complement of a big enough compact set $K$, are isomorphic, but not in a canonical way.
	Thus, we fix an isomorphism of principal fiber bundles between $P_{T^2}$ for appropriately big radius and $M|_{B \setminus K}$.
	By Hodge theory on compact manifolds, we choose harmonic $(\kappa_{\eta_\infty}, \kappa_{\theta_\infty}) \in \Omega^2(\Sigma)^2$ representing $c_1\left(\dot{P}\right)$ and we choose a connection $(\eta_\infty, \theta_\infty) \in \Omega^1(\dot{P})$ on the $T^2$ bundle $\dot{P} \to \Sigma$ with curvature $(\kappa_{\eta_\infty}, \kappa_{\theta_\infty})$, and we pull it back on $P$.
	Then we complete it to a connection $\left(\bar{\eta}_\infty, \bar{\theta}_\infty\right)$ on the whole $M$ in the following way: choose any connection $(\tilde{\eta}, \tilde{\theta})$ on $M \to B$ and choose a $[0,1]$-valued function $\chi$ that is identically 0 on $K$ and that is identically 1 outside a compact set bigger than $K$.
	Then, we call $\bar{\eta}_\infty = \chi \theta_\infty + (1- \chi) \tilde{\eta}$ and similarly for $\bar{\theta}_\infty$.
	We can now run theorem \ref{thm:final_result} on $M$ with $\left(\bar{\eta}_\infty, \bar{\theta}_\infty\right)$ as the chosen fixed connection on $M \to B$, yielding a solution $(\eta, \theta, \omega, \Re \Omega, p, q, r)$.
	
	In light of equation \eqref{eq:metric_Spin7}, we choose the following metric on the asymptotic model:
	\begin{equation}\label{eq:asymptocic_metric}
		g_{P_{T^2}} = \varepsilon^2\left(p_0^{\frac{1}{2}}q_0^{-\frac{3}{2}} \eta_\infty^2 + \left(r_0^2(p_0q_0)^\frac{3}{2}+ q_0^{\frac{1}{2}}p_0^{-\frac{3}{2}} \right) \theta_\infty^2 - 2 r_0 p_0 \eta_\infty \odot \theta_\infty\right) + (p_0q_0)^{\frac{1}{2}} g_C
		.
	\end{equation}
	This metric trivially satisfies all properties of definition \ref{def:ATkC} except for the decay at infinity.
	This follows from the fact that $p - p_0, q - q_0$ and $r - r_0$ are all in $C^\infty_{-1}$, that $g_B - g_C$ decays with rate $-1$ by hypothesis and that we can write $\varepsilon^2 \eta_\infty^2 - \eta^2$ as $(\varepsilon \eta_\infty + \eta) \odot (\varepsilon \eta_\infty - \eta)$, where $\varepsilon \eta_\infty + \eta$ is bounded and $\varepsilon \eta_\infty - \eta \in C^\infty_{-1}$.
	
	To see this last statement note that $\eta$ differs from $\varepsilon \eta_1$ only up to terms in $C^\infty_{-1}$ and thus it is sufficient to prove $\eta_\infty - \eta_1 \in C^\infty_{-1}$.
	Clearly this is equivalent to $\bar{\eta}_\infty - \eta_1 \in C^\infty_{-1}$.
	But this was a defining property of $\eta_1$ in lemma \ref{thm:HYM_connection}.
	
	The reasoning for $\theta^2$ and $\eta \odot \theta$ is analogous.
%
%
\end{proof}
%

\begin{remark}
	The connections $\eta_\infty$ and $\theta_\infty$ are Hermitian Yang Mills on $P$.
	The fact that they are Yang Mills, i.e.\ that their curvature is harmonic, follows from the fact that we chose $(\kappa_{\eta_\infty}, \kappa_{\theta_\infty})$ to be harmonic.
	That they are Hermitian Yang Mills, i.e.\ that their curvature lies in $\Lambda^2_8(B)$, if follows from remark 4.10 in \cite{Foscolo2021}.
\end{remark}

\subsection{Full holonomy}

Before the following theorem, we recall two standard lemmas.

\begin{lemma}\label{thm:fundamental_group_circle_bundles}
	Let $B$ be a manifold and $p : P \to B$ be a principal $T^k$ bundle.
	Then $\pi_1(P)$ is finite if and only if $\pi_1(B)$ is finite and $c_1\left(P\right)$ generates a space of dimension $k$ inside $H^2(B)$.
%
\end{lemma}

\begin{lemma}\label{thm:universal_cover_circle_bundle}
	Let $B$ be a simply connected manifold and $p : P \to B$ be a principal circle bundle.
	Suppose $\pi_1(P)$ is finite.
	Let $\pi_{\tilde{P}} : \tilde{P} \to P$ be the finite universal cover.	
	
	Then, $p \comp \pi_{\tilde{P}} : \tilde{P} \to B$ is a non-trivial principal circle bundle and $c_1\left(\tilde{P}\right)$ is a primitive element in $H^2(B; \Z)$ such that $c_1(P)$ is a $\Z$-multiple of $c_1\left(\tilde{P}\right)$.
\end{lemma}
%

In order to prove that the newly built manifolds have full holonomy, it will be convenient to introduce a new connection with torsion on the asymptotic model $(P_{T_k}, g_P)$, which is analogous to the one introduced in \cite{Reinhart1983} and has been used in a setting similar to ours by \cite{Foscolo2019}.
The main motivation to introduce such a connection $\nabla^P$ is that the generators of the principal $T^k$-action are not parallel for the Levi Civita connection of $g_P$.
Indeed, if $X, Y$ are generators of the $T^k$-action and $V, W$ are $T^k$-invariant horizontal vector fields, a simple calculation using the Killing condition for $V$ and the fact that $[X, V] = 0$ shows that $\nabla_{X}^{{P}} Y=0$ and
\begin{equation}\label{eq:Levi-Civita_P}
	 \nabla_V^{{P}} X=\frac{1}{2}\left(V\lrcorner d X^\flat\right)^{\sharp} \quad \nabla_{X}^{{P}} V=\frac{1}{2}\left(V\lrcorner d X^\flat\right)^{\sharp} \quad \nabla_V^{{P}} W = -\frac{1}{2} \rho \comp dA_\infty(V, W) +\operatorname{HL}\left(\nabla_V^{C} W\right)
\end{equation}
where $HL$ denotes the horizontal lift, $\nabla^C$ denotes the Levi Civita connection on the cone $C$, $\rho : \Lie\left(T^k\right) \to \mathfrak{X}(P_{T_k})$ is the canonical morphism of Lie algebras and we are denoting by $W$ also the projection of $W$ on $C$, which exists because $W$ is $T^k$-invariant.
$\nabla_V^{{P}} X$ is never going to be zero unless the connection $A_\infty$ is flat, which is not the case for our construction.

\begin{definition}
	Let $(P_{T_k}, g_P)$, with $P_{T^k} \to C$ be an asymptotic model for an $AT^kC$ manifold.
	Then we define the \textit{adapted connection} on $(P_{T_k}, g_P)$ to be the unique metric connection whose covariant derivative $\nabla^A$ is uniquely defined by requiring that, if $X, Y$ are two generators of the principal $T^k$-action and $V, W$ are two horizontal $T^k$-invariant vector fields,
	\begin{equation}\label{eq:adapted_connection}
		\nabla^A_XY = 0 \quad \nabla^A_V X = 0 \quad \nabla^A_X(V) = 0 \quad \nabla^A_V W = \operatorname{HL} \left(\nabla^C_V W \right)
	\end{equation}
	where $HL$ denotes the horizontal lift, $\nabla^C$ denotes the Levi Civita connection on the cone $C$ and we are denoting by $W$ also the projection of $W$ on $C$, which exists because $W$ is $T^k$-invariant.
\end{definition}

Since the generators of the $T^k$ action and horizontal $T^k$-invariant fields generate the tangent space to $P_{T^k}$ at each point, the above condition uniquely defines a covariant derivative through the Leibniz condition.
$\nabla^A$ has torsion $\rho \comp dA_\infty$ where $\rho : \Lie\left(T^k\right) \to \mathfrak{X}(P_{T_k})$ is the canonical morphism of Lie algebras.

We also need the following technical result.

\begin{lemma}\label{thm:closed_coclosed_constant_forms}
	Let $M$ be asymptotically $T^2$-fibered conical to $P \to C(\Sigma)$ with rate $\mu < 0$.
	
	Then for any parallel $k$-form $\gamma$ on $M$ there exists a $k$-form $\gamma_\infty$ on $P$ such that $\gamma - \gamma_\infty \in C^\infty_{\mu}(P)$ and $\nabla^A \gamma_\infty = 0$.
\end{lemma}
\begin{proof}
	The result is obvious for $k = 0$, since a parallel function is constant and thus we can choose the function itself at infinity.
	
	Thus, let $\gamma$ be a parallel $k$-form for $k\geq 1$, and call $\gamma$ also its pullback on the cone outside a compact set.
	Call $\dot{P}$ the restriction of $P$ to $\Sigma$ which a priori has no natural metric (thus we consider it just as a differentiable manifold), and call $\iota_R : \dot{P} \to P$ the canonical embedding of $\dot{P}$ into $P$ at $r = R \in \R^+$.
	Moreover, call $\dot{P}_r$ the manifold $\dot{P}$ with the pullback metric through $\iota_r$.
	Then $\gamma = \sum_{j = 0}^{k-1} dr \wedge r^j \alpha_j + \sum_{j = 0}^{k}r^j \beta_j$ for some  $(k-1)$-forms $\alpha_j \in \Lambda^j HP^* \wedge \Lambda^{k - 1 - j} VP^*$ and $k$-form $\beta_j \in \Lambda^j HP^* \wedge \Lambda^{k - j} VP^*$ on $P$ such that $\partial_r \interior \alpha_j = \partial_r \interior \beta_j = 0$.
	Here, $HP$ denotes the horizontal space with respect to the connection $A_\infty = (\eta, \theta)$.
	Let $\nabla^M$ and $\nabla^P$ be the Levi-Civita connections on $M$ and $P$ respectively.
	Since $\nabla^M \gamma = 0$ and $\nabla^P_{\partial_r} dr = 0$,
	$$
		\nabla^P_{\partial_r} \gamma = \sum_{j = 0}^{k-1} dr \wedge \nabla^P_{\partial_r} (r^j \alpha_j) + \sum_{j = 0}^{k} \nabla^P_{\partial_r} (r^j \beta_j) \in C^\infty_{\mu - 1}(\Lambda^k(P))
		,
	$$
	since $\nabla^M - \nabla^P$ is a tensor with decay rate $\mu - 1$, given that $g_M - g_P$ decays with rate $\mu$ by hypothesis.
	Since all the addends in the above expansion of $\nabla^P_{\partial_r} \gamma$ are linearly independent, they individually have to decay with rate $\mu - 1$, and since $dr$ does not decay, it has to be $\nabla^P_{\partial_r} (r^j \alpha_j), \nabla^P_{\partial_r} (r^j \beta_j) \in C^\infty_{\mu - 1}$.
	Now call $\alpha_j^r = \iota_r^* \alpha_j$ and similarly for $\beta_j$.
	These are curves in $C^\infty_0$ and their covariant derivatives along radially invariant vector fields are also curves in $C^\infty_0$.
	Let us prove that these curves are Cauchy in norm $C^0$, and thus they have a limit.
	One can verify that, given the above metric on $P$, parallel transporting through $\nabla^P$ $\alpha^t$ to $r = s$ and then pulling it back through $\iota_s$ is the same as pulling it back through $\iota_t$ and multiplying it by $\left(\frac{t}{r}\right)^j$.
	Then, for some $C_{\alpha_j} \in \R$,
	\begin{multline}\label{eq:estimate_alphas}
		\norm{\alpha^{t+s}_j - \alpha^t_j}_{C^0\left(\dot{P}_1\right)} = \norm{\iota^*_1\left((t+s)^j\operatorname{Pt}_{t+s} \left(\alpha_j\right) - t^j \operatorname{Pt}_{t} \alpha_j \right)}_{C^0\left(\dot{P}_1\right)} \leq \\ \int_t^{t+s}\norm{\nabla^P_{\partial_r} (r^j\alpha_j)}_{C^0\left(\dot{P}_r\right)}dr \leq C \int_t^{t+s}r^{\mu - 1}dr \leq C t^{ \mu}
	\end{multline}
	and so for the iterated covariant derivatives of $\alpha_j$.
	Hence, the curve $\alpha_j^r$ and all the curves of its derivatives are Cauchy in $C^0 \left(\Lambda^{k-1}\left(\dot{P}_1\right)\right)$ and thus $\alpha_j^r$ has a limit $\alpha^\infty_j$ in $C^\infty\left(\Lambda^{k-1}\left(\dot{P}\right)\right)$.
	The same exact proof holds for each $\beta_j$, yielding $\beta^\infty_j$.
	By taking $s \to \infty$ in the inequality \eqref{eq:estimate_alphas} for $\alpha_j$ and its derivatives, and by remarking that $\norm{\cdot}_{C^0\left(\dot{P}_1\right)} = r^j\norm{\cdot}_{C^0\left(\dot{P}_r\right)}$ we get that $r^j(\alpha_j^\infty - \alpha_j) \in C^\infty_\mu(P)$.
	This is because $\partial_r \lrcorner \alpha_j = 0$ and thus the $C^0_\mu(P)$ norm is determined by the $C^0\left(\dot{P}_r\right)$ norms at various $r$.
	The same holds for $\beta_j$.
	
	In the proof of existence of the limit we used that $\nabla^M_{\partial_r} = 0$, but we still did not use that $\nabla^M_{X} = 0$ for any other vector field $X$ orthogonal to $\partial_r$.
	Without loss of generality suppose $X$ is the pullback of a vector field on $\dot{P}$.
	This reveals to be crucial to prove that $\nabla^A \gamma_\infty = 0$.
	Indeed, this implies that, for any vector field $X$ orthogonal to $\partial_r$,
	$$
	\nabla^P_{X} \gamma = \sum_{j = 0}^{k-1} \nabla^P_{X} dr \wedge r^j \alpha_j + \sum_{j = 0}^{k-1} dr \wedge \nabla^P_{X} (r^j \alpha_j) + \sum_{j = 0}^{k} \nabla^P_{X} (r^j \beta_j) \in C^\infty_{\mu}(\Lambda^k(P))
	,
	$$
	by the same reasoning in the case of $\nabla^P_{\partial_r}$ and the fact that $X$ grows with a factor $r$.
	Now let us distinguish the case where $X$ is horizontal and when it is vertical.
	If $X$ is horizontal then $\nabla^P_X dr = r X^\flat$.
	In this case, by collecting linearly independent term as above we find that $\nabla^P_X \beta_0 \in C^\infty_{\mu}(\Lambda^k(P))$ and moreover
	$$
		r^j \nabla^P_X(\beta_j) - X^\flat \wedge r^{j-1}\alpha_{j} \in C^\infty_{\mu}(\Lambda^k(P)) \text{ for } j \geq 1 \qquad r^j\nabla^P_X \alpha_j \in C^\infty_{\mu}(\Lambda^k(P))
		.
	$$
	Now let us look e.g.\ to last equation which is equivalent to $\nabla^P_X \alpha_j \in C^\infty_{\mu - j}(\Lambda^k(P))$.
	By the equations defining $\nabla^A$, $\nabla^A_X \alpha_j^\infty$ decays with rate $-j$.
	Moreover, $\nabla^P_X \alpha_j = \nabla^A_X \alpha_j^\infty + \left(\nabla^P_X-\nabla_X^A\right) \alpha_j+\nabla_X^A\left(\alpha_j-\alpha_j^{\infty}\right)$ with the last two terms decaying with rate $\mu - j$.
	Hence, the only possibility is having $\nabla^A_X \alpha_j^\infty = 0$.
	By reasoning analogously for the other two equations, we get that
	\begin{equation}\label{eq:nabla_A_alpha_beta}
		\nabla^A_X \beta_0^\infty = 0 \qquad \nabla^A_X(\beta_j^\infty) = X^\flat \wedge \alpha_{j - 1}^\infty \text{ for } j \geq 1 \qquad \nabla^A_X \alpha_j^\infty = 0
		.
	\end{equation}
	On the other hand, if $X$ is vertical, reasoning as above we get that $\nabla^A_X \alpha_j^\infty = 0$ and $\nabla^A_X \beta_j^\infty = 0$.
	
	Now, define $\gamma_\infty = \sum_{j = 0}^{k-1} dr \wedge r^j \alpha^\infty_j + \sum_{j = 0}^{k}r^j \beta^\infty_j$, and call $\gamma_\infty$ also its pullback on $P$.
	It is clear by the above estimates for $r^j\alpha^\infty_j - r^j \alpha_j$ and $r^j\beta^\infty_j - r^j \beta_j$ that $\gamma - \gamma_\infty \in C^\infty_{\mu}(P)$.
	We further claim that $\nabla^A\gamma_\infty = 0$.
	By equations \eqref{eq:nabla_A_alpha_beta}, we see that for any vector field $X$ that is the pullback of a vector field on $\dot{P}$, $\nabla^A_X \gamma_\infty = 0$.
	Moreover, since the connection $A_\infty$ is by hypothesis radially invariant, $\partial_r \lrcorner d A_\infty = 0$ and thus $\nabla^P_{\partial_r} = \nabla^A_{\partial_r}$ by equations \eqref{eq:Levi-Civita_P} and \eqref{eq:adapted_connection}.
	Since $\nabla^P_{\partial_r}(r^j \alpha_j^\infty) = \nabla^P_{\partial_r}(r^j \beta_j^\infty) = 0$, it follows that $\nabla^A_{\partial_r}\gamma_\infty = 0$, which concludes the proof.
\end{proof}

Before proving that the newly found manifolds have full holonomy, let us prove a useful lemma about the holonomy of Calabi Yau cones.

\begin{lemma}\label{thm:CY_cone_holonomy}
	Any nontrivial Calabi Yau cone $C$ of real dimension 6 has full holonomy $\SU(3)$.
	Here by trivial, we mean isomorphic to $\C^3$.
\end{lemma}
\begin{proof}
	To show that the holonomy of $C$ needs to be the whole $\SU(3)$, we remark that if it were not, it would have to be $\Sp(1) \times 1$ or a subgroup.
	In this case the cone should be $\C \times C_4$ where $C_4$ is a 4-cone.
	However, unless $C_4 = \C^2$ the singular locus of $C$ would have nonzero dimension, and this is impossible.
\end{proof}

\begin{theorem}\label{thm:full_holonomy}
	Let $M$ and $B$ be as in the statement of theorem \ref{thm:final_result}.
	
	Then, if $c_1(M)$ generates a 2-dimensional subspace in $H^2(B)$, then $\Hol(M) = \Spin(7)$.
\end{theorem}
\begin{proof}
	Note that it is sufficient to prove that $\Hol^0(M) = \Spin(7)$.
	Indeed since $\Hol^0(M) \leq \Hol(M)$ and the existence of a torsion free $\Spin(7)$-structure implies that $\Hol(M) \leq \Spin(7)$ we get the thesis.
	
	We can assume $B$ to be simply connected. Indeed, by lemma \ref{thm:finite_fundamental_group}, $B$ has finite fundamental group and its universal cover $\tilde{B} \to B$ is AC Calabi-Yau. Let $\tilde{M} \to \tilde{B}$ be the pullback of $M \to B$ through the standard projection $\tilde{B} \to B$.
	$\tilde{B}$ has a CY-structure pulled back from $B$ and $\tilde{M}$ has $\Spin(7)$ structure $\tilde{\Phi}_\varepsilon$ built from $\tilde{B}$ through theorem \ref{thm:final_result}.
	Then, the canonical map $\pi_{\tilde{M}} : \tilde{M} \to M$ is a covering space and $ \pi_{\tilde{M}}^* \Phi_\varepsilon = \tilde{\Phi}_\varepsilon$, by the uniqueness statement in the theorem.
	Thus $\tilde{M}$ is a Riemannian covering space of $M$ for the metrics induced by the $\Spin(7)$ structures, and hence we have that $\Hol^0(M) = \Hol^0\left(\tilde{M}\right)$.
	Thus, we replace $B$ with $\tilde{B}$ and $M$ with $\tilde{M}$ and we drop the tildes.
	
	Now, by lemma \ref{thm:fundamental_group_circle_bundles}, the total space of $M$ has finite fundamental group.
	By lemma \ref{thm:universal_cover_circle_bundle} applied twice, its universal cover $\tilde{M}$ is also a $T^2$ bundle over $B$.
	Hence, again by invariance of restricted holonomy under covering spaces we can assume $M$ to be simply connected.
	
	
	The hypothesis of $c_1(M)$ generating a 2-dimensional space is equivalent to $c_1(P_1)$ and $c_1(P_2)$ being linearly independent.
	Denote by $\tilde{P}_1, \tilde{P}_2$ the universal covers of the original $P_1$ and $P_2$.
	Note that, since we passed to universal covers, by lemma \ref{thm:universal_cover_circle_bundle}, the $c_1\left(\tilde{P}_1\right)$ and $c_1\left(\tilde{P}_2\right)$ considered here are primitive elements of $H^2(B; \Z)$ whose $\Z$-multiples contain the original $c_1(P_1)$ and $c_1(P_2)$.
	However, it is clear that $\R$-linear independence of a set of elements of a lattice is equivalent to $\R$-linear independence of the relative primitive elements.
	
	Now, thanks to simple connectedness of $M$, by lemma \ref{thm:Bryant}, proving that the restricted holonomy of $M$ is the whole group $\Spin(7)$ amounts to proving that there are no parallel 1-forms and no parallel non-degenerate 2-forms on $M$, where parallel is meant with respect to the Levi-Civita connection.
	Now we prove that parallel 1-forms and 2-forms are $T^2$-invariant.
	By studying the possible parallel forms for $\nabla^A$ and by lemma \ref{thm:closed_coclosed_constant_forms}, a parallel 1-form $\alpha$ and a parallel 2-form $\beta$ on $M$ can be thus written, outside of a compact $K$, as $\alpha = a_\infty \eta_\infty + b_\infty \theta_\infty + \alpha_H + \delta_\alpha$ and $\beta = c_\infty \eta_\infty \wedge \theta_\infty + \eta_\infty \wedge \beta^1_H + \theta_\infty \wedge \beta_H^2 + \beta_H^3 + \gamma_\beta$ where $\delta \in \Omega^1_{-1}(M)$, $\gamma \in \Omega^2_{-1}(M)$ and $\alpha_H, \beta_H^k$ are forms on $C$ of appropriate degree that are parallel with respect to the Levi-Civita connection of $C$.
	Now, since the holonomy of $C$ has to be the whole $\SU(3)$ then $\alpha_H = \beta^1_H = \beta^2_H = 0$ and $\beta^3_H = e_ \infty \omega_\infty$, where $e_\infty \in \R$ and $\omega_\infty$ is the Kähler form on the cone.
	Indeed, the existence of other parallel forms would reduce the holonomy of $C$ to a subgroup of $\SU(3)$, which would mean that $C \iso \C^3$ by lemma \ref{thm:CY_cone_holonomy}.
	On the other hand, by Bishop-Gromov inequality, the only CY manifold asymptotically conical to $\C^3$ is $\C^3$ itself, but this is impossible since we assume $B$ has nontrivial $H^2$.
	
	Hence,
	\begin{equation}\label{eq:beta}
		\alpha = a_\infty \eta_\infty + b_\infty \theta_\infty + \delta_\alpha \qquad
		\beta = c_\infty \eta_\infty \wedge \theta_\infty + e_\infty \omega_\infty + \gamma_\beta
		.
	\end{equation}	
	The two maps
	$$
		\begin{aligned}
			\mathcal{L}_1 : \mathcal{H}^1_{\nabla_{g_M}} & \to \R^2 \\
			\alpha & \mapsto (a_\infty, b_\infty)
		\end{aligned}
		\qquad
		\begin{aligned}
			\mathcal{L}_2 : \mathcal{H}^2_{\nabla_{g_M}} & \to \R^2 \\
			\beta & \mapsto (c_\infty, e_\infty)
		\end{aligned}
	$$
	are clearly linear.
	They are both also injective because in both cases the preimage of $(0, 0)$ is a decaying parallel form which is therefore null.
	By definition of principal connection, the $T^2$ action preserves $\eta_\infty$ and $\theta_\infty$ and it also preserves $\omega_\infty$, it being a horizontal form.
	Thus, for any $\alpha \in \mathcal{H}^1_{\nabla_{g_M}}$ and for any $g \in T^2$, $\mathcal{L}_1 (g^* \alpha) = \mathcal{L}_1 (\alpha)$ and by injectivity of $\mathcal{L}_1$ we conclude that parallel 1-forms are $T^2$-invariant.
	The same argument applied to $\mathcal{L}_2$ proves that parallel 2-forms are $T^2$-invariant.
	Any $T^2$-invariant 1-form $\alpha$ and  $T^2$-invariant 2-form $\beta$ are of the form
	$$
		\alpha = a \eta + b \theta + \kappa \qquad \beta = c \eta \wedge \theta + \tau \wedge \eta + \chi \wedge \theta + \zeta
	$$
	where $a, b, c \in C^\infty(B)$, $\kappa, \tau, \chi \in \Omega^1(B)$ and $\zeta \in \Omega^2(B)$.
	
	From $d \alpha = 0$ we have that $d a = d b = 0$ and $-d \gamma = a d \eta + b d \theta$.
	Hence, the last equation in cohomology reads $a c_1(P_1) + bc_1(P_2) = 0$ where $a, b \in \R$ since we saw that they are constant.
	Thus, the hypothesis of linear independence implies that $a = b = 0$, and thus $\alpha \in \Omega^1(B)$, i.e.\ it is horizontal.
	Since $\eta_\infty, \theta_\infty$ and horizontal forms are linearly independent, we conclude that $a_\infty = b_\infty = 0$ and thus outside of $K$, $\kappa = \delta_\alpha$.
	Therefore, $\alpha = \kappa$ decays and it is parallel and therefore is 0.
	
	In the case of $\beta$, an argument based on closedness similar to before proves $dc = 0$ and $cd\eta - d\chi = 0$, which in cohomology reads $c \cdot c_1(P_1) = 0$, and by hypothesis of linear independence we get $c=0$.
	Thus, $d \tau = d \chi = 0$ and $d \zeta = \tau \wedge d\eta + \chi \wedge d \theta$.
	
	Since all other addends in the second equation \eqref{eq:beta} are $T^2$ invariant, $\gamma_\beta$ also is, and therefore we can write
	$$
	\gamma_\beta = c_\gamma \eta \wedge \theta + \eta \wedge \tau_\gamma + \theta \wedge \chi_\gamma + \zeta_\gamma
	$$
	where $c_\gamma \in C^\infty_\nu(B)$, $\tau_\gamma, \chi_\gamma \in \Omega^1_\nu(B)$ and $\zeta_\gamma \in \Omega^2_\nu(B)$, for $\nu < 0$.
	Moreover, $\eta_\infty = \eta + \xi_\eta$ and $\theta_\infty = \theta + \xi_\theta$, where $\xi_\eta, \xi_\theta \in \Omega^2_{-1}(M \setminus K)$.
	Therefore,
	$$
	\beta = (c_\infty + c_\gamma) \eta \wedge \theta + \eta \wedge ( c_\infty \xi_\eta + \tau_\gamma) + \theta \wedge (-c_\infty\xi_\eta + \chi_\gamma) + c_\infty \xi_\eta \wedge \xi_\theta + e_\infty \omega_\infty + \zeta_\gamma
	.
	$$	
	Now, comparing the two expressions of $\beta$ that we got, we get $c = 0 = c_\infty + c_\gamma$ which implies $c_\infty = c_\gamma = 0$ since one is constant and the other decays.
	Hence, $\beta = \tau \wedge \eta + \chi \wedge \theta + \zeta$ where $\tau$ and $\chi$ decay.
	
	Hence, $\beta$ is decaying in the vertical directions and this is not possible if we want it to be non-degenerate (hence symplectic) and compatible with the metric.
	One of the many possible ways to see it is the following.
	If $\beta$ were a non-degenerate form compatible with the metric, we would have $J\eta = -X \interior \beta = \tau$, where $J$ is the complex structure uniquely determined by $\beta$ and the metric.
	However, this is not possible as $\tau$ decays and $\eta$ does not, and $J$ is an isometry pointwise.
	This concludes the proof.
\end{proof}

\subsection{Additional topological constraints}

In this section we remark the additional topological constraints that we need to impose on the base $B$ in order to have the existence of a $\Spin(7)$-structure with full holonomy on $M$, and which need to be verified together with the topological constraints introduced in section \ref{sec:topological_constraints}.

\begin{lemma}\label{thm:additional_topological_constraint}
	Let $B$ and $M \iso P_1 \times P_2$ be in the statement of theorem \ref{thm:final_result}, and call $\Sigma$ the Sasaki-Einstein link of $B$, as usual.
	Then, the following topological conditions are necessary for the existence of a $\Spin(7)$-structure with full holonomy on $M$:
	\begin{enumerate}
		\item $\dim H^2(B) \geq 2$,
		\item $\dim H^2(\Sigma) \geq 1$.
	\end{enumerate}
\end{lemma}

	In particular, $\Sigma$ cannot be a finite quotient of the round 5-sphere $S^5$, which can be equivalently stated as $\Sigma$ cannot have constant sectional curvature.
\begin{proof}
	The first condition is obvious from theorem \ref{thm:full_holonomy}.

	As for the second condition, consider the long exact sequence of manifolds with boundary
	$$
		\ldots \to H^1(\Sigma) \to H^2_c(B) \to H^2(B) \to H^2(\Sigma) \to \ldots
	$$
	Recall that $H^2_c(B) \iso L^2\mathcal{H}^2(B)$.

	Now, by lemma \ref{thm:exact_star_H4}, $c_1(P_1)$ and $c_1(P_2)$, which are nonzero, cannot be in
	$$
		\Im \left( H^1(\Sigma) \to L^2\mathcal{H}^2(B)\right) = \ker(H^2(B) \to H^2(\Sigma))
	$$
	and thus $H^2(\Sigma)$ cannot be 0.
\end{proof}

A posteriori, we can thus see that the hypothesis of $\Sigma$ not having constant sectional curvature used in section \ref{sec:asymptotical} is not restrictive.

\section{Generalization to orbifolds}\label{sec:orbifolds}
In this section we explain how theorem \ref{thm:main_theorem} can be generalized to the case of orbifolds, without getting into the details of the proofs, but rather just mentioning the differences with respect to the smooth case.

Let us start with a definition.
In the context of orbifolds, by Calabi–Yau cone we mean a Riemannian cone whose link is an orbifold and whose holonomy is contained in $\SU(3)$.
The term might be overloaded because in the context of orbifolds, the word cone could also refer to a Riemannian cone including its singularity as an orbifold singularity, but in our case we are not interested the conic singularity, but we are rather concerned with singularities of the link.
Similarly to the smooth case, the link will be a Sasaki-Einstein orbifold.

Given a metric vector orbibundle $(E, g_E)$ with a metric connection, we can define H\"older norms $C^{k, \alpha}_\nu$ of $E$ as in the smooth case.
Then we define
\begin{equation}\label{eq:orbifold_Holder}
	C^{k, \alpha}_\nu(E) = \left\{u \in C^0_\text{orb}(E) \mid \norm{u}_{C^{k, \alpha}_\nu} < +\infty \right\}
	.
\end{equation}
Given an orbifold $\Sigma$, the definition of \textit{asymptotically conical orbifold} is analogous to definition \ref{def:AC}.
Note that singularities are allowed to exist both in the compact region and all the way to infinity.

We will say that a principal $T^k$ orbibundle on an orbifold $B$ is \textit{Seifert} if its total space is smooth.

Then, the following is true.
\begin{theorem}\label{thm:main_theorem_orbifold}
	Let $\left(B, g_B, \omega_0, \Omega_0\right)$ be an asymptotically conical Calabi Yau 3-orbifold asymptotic with rate $\nu \in \R^-$ to the Calabi-Yau cone $\left(\mathrm{C}(\Sigma), g_{\mathrm{C}}, \omega_{\mathrm{C}}, \Omega_{\mathrm{C}}\right)$ over a Sasaki-Einstein 5-orbifold $\Sigma$ with nonconstant sectional curvature.
	Let $M \rightarrow B$ be a principal Seifert $T^2$-bundle on $B$ such that $c_1^{\text{orb}}(M) \in H^2_{\text{orb}}\left(B; \Z^2\right)$ spans a 2-dimensional subspace in $H^2_{\text{orb}}\left(B; \R\right)$ and
	\begin{equation}\label{eq:topological_condition_orbifold}
		c_1^{\text{orb}}(M) \smile \left[\omega_0\right] = 0 \in H^4_{\text{orb}}\left(B; \R^2\right)
		.
	\end{equation}
	
	Then the 8-manifold $M$ carries an analytic curve of $T^2$-invariant torsion-free $\Spin(7)$-structures $\Phi : (0, \varepsilon_0) \to \Omega^4(M)$ such that, by denoting with $g_\varepsilon$ the Riemannian metric on $M$ induced by $\Phi_\varepsilon$, conditions 1, 2 and 3 from theorem \ref{thm:main_theorem} hold.
\end{theorem}

Now, by working with H\"older sections as defined in equation \eqref{eq:orbifold_Holder}, we claim that most of the analysis used in the proof of theorem \ref{thm:main_theorem} can be adapted to the orbifold case.
Indeed, for example, lemma \ref{thm:harmonic_1_6_forms} eventually relies on Stokes theorem, which is true for orbifolds and on Lichnerowicz–Obata Theorem.
The Lichnerowicz part of the theorem, i.e.\ the non-strict inequality, relies on Bochner's formula (and integration by parts), which is also true for orbifolds.
The Obata part on the other hand, does not hold for orbifolds, as there exist more orbifolds than the 5-sphere that have first eigenvalue of the Laplacian equal to 5.
However, by \cite{Shioya1999}, analogous to the smooth case, all orbifolds for which equality holds are quotients of the 5-sphere by finite subgroups of $\SO(6)$.
Thus, under the assumption that $\Sigma$ does not have constant sectional curvature, the strict bound holds for the orbifold case as well.
There is just one small technicality, because in \cite{Shioya1999}, the authors considers the Laplacian on $C^\infty_0(\Sigma \setminus \Sigma_{\text{sing}})$ functions whereas we are interested in $C^\infty(\Sigma)$ functions.
However, if the first eigenvalue of the Laplacian on $C^\infty_0(\Sigma \setminus \Sigma_{\text{sing}})$ is not 5, then the first eigenvalue of the Laplacian on $C^\infty(\Sigma)$ is also not 5.
We can see this by proving the contrapositive, i.e.\ if the first eigenvalue of the Laplacian on $C^\infty(\Sigma)$ is 5, then on $C^\infty_0(\Sigma \setminus \Sigma_{\text{sing}})$ is also 5.
Indeed, given that our orbifold $\Sigma$ is Sasaki-Einstein, it is in particular oriented, hence its singular locus has codimension at least 2.
Thus, $C^\infty_0(\Sigma \setminus \Sigma_{\text{sing}})$ is dense in $C^\infty(\Sigma)$ in norm $W^{1, 2}$ and thus given that the first eigenvalue of the Laplacian on $C^\infty(\Sigma)$ is 5, there is a sequence of functions in $f_k \in C^\infty_0\left(\Sigma \setminus \Sigma_{\text{sing}}\right)$ such that the Rayleigh quotient
$$
\frac{\displaystyle\int_\Sigma \norm{\nabla f_k}_{W^{1, 2}}^2}{\displaystyle\int_\Sigma \norm{f_k}_{W^{1, 2}}^2} \searrow 5 \text{ as } k \to \infty
.
$$
Since the spectrum of the Laplacian on $C^\infty_0\left(\Sigma \setminus \Sigma_{\text{sing}}\right)$ is discrete, by proposition 2.1 of \cite{Shioya1999}, 5 must also be an eigenvalue of the Laplacian on $C^\infty_0\left(\Sigma \setminus \Sigma_{\text{sing}}\right)$.

Any finite quotient of the round 5-shpere has constant sectional curvature, and thus, under our assumption of nonconstant sectional curvature for $\Sigma$, the strict inequality holds.

Given that appendix A and B from \cite{Foscolo2021} transfer to the case of orbifolds, the rest of the analysis, such as theorem \ref{thm:Dirac_iso}, holds true for orbifolds.
Appendix A holds in the orbifold case because it all about separations of variables arguments to exploit Hodge theory on compact manifolds to get results on Riemannian cones.
For Hodge theory on compact orbifolds see for example theorem 4.4 of \cite{Farsi2021}.

As for appendix B however, it is slightly less evident why it should transfer orbifolds, as there are underlying analysis arguments that need to be verified for orbifolds too.
The analytic argument underlying the results contained in appendix B is about transferring standard estimates on $\R^n$ on asymptotically conical manifolds, which is done by using a scaling argument for annuli $\left\{2^k R \leq r \leq 2^{k+1} R\right\}$, with $R \in \R^+$, and by using the fact that weighted Sobolev and H\"older norms on any annulus are equivalent to the ones of the $k=0$ annulus up to a factor of $\left(2^k R\right)^{-\nu}$.
Given that in our case singularities are allowed to extend all the way to infinity, we need to check that the estimates we are using transfer to annuli with orbifold singularities, but after having done so, the scaling argument works for orbifolds too.
The reason why the estimates of interest transfer to annuli over orbifolds is that the annuli are relatively compact, so we can cover them with a finite number of charts, and in each chart we can use the estimates for $\R^n$ restricted to invariant sections.

The central point in the above proof is that, even if singularities are allowed to extend to infinity, they can do so only as the product of the line and singularities on the link, so that the isotropy group of the singularities does not change going to infinity.
This is what allows us to claim that the geometries on different annuli are equivalent.
In other words, the orbifold case introduces the extra geometric information of the singularities, which to our purposes behaves similarly to the rest of the geometric information, such as the metric, that was there in the smooth case too.
The point is that all this geometric information, including the singularities in the orbifold case, has a scaling symmetry, which is a very strong constraint and which allows us to transfer estimates from relatively compact regions in $\R^n$ to the whole cone.

Thus, the proof of theorem \ref{thm:main_theorem_orbifold} is complete.

\section{New examples of \texorpdfstring{$\Spin(7)$}{Spin(7)} manifolds}\label{sec:examples}

In this section we use theorem \ref{thm:main_theorem} to build new examples of complete non-compact $\Spin(7)$ manifolds.
As anticipated in the introduction, this is possible thanks to the many new recent examples of AC Calabi-Yau manifolds built by multiple authors.
As it was done in \cite{Foscolo2021}, a particular class of AC Calabi-Yau manifolds that we are able to exploit to get some interesting examples of new $\Spin(7)$ manifolds is that of crepant resolutions of Calabi-Yau cones (as in \cite{Foscolo2021}, we do not consider affine smoothing because they have vanishing second cohomology).
Let $(C, \omega_C, \Omega_C)$ be a Calabi-Yau cone and denote with $o$ the singularity of the cone.
In this context, by (CY) \textit{crepant resolution} we mean a Calabi-Yau manifold $(B, \omega_B, \Omega_B)$ with a map $\pi : B \to C$, such that $\pi$ is a diffeomorphism outside the exceptional locus $\pi^{-1}(o)$ such that $\pi^* \Omega_C = \Omega_B$.
In general, the word crepant means that the canonical bundle of $B$ is required to be trivial but in our specific context this is already implied by it being Calabi-Yau.

We will make use of the following result on the existence and uniqueness of AC Calabi–Yau structures on crepant resolutions of Calabi–Yau cones stated in the special case of complex dimension 3.

\begin{theorem}\label{thm:crepant_resolution}
	Let $\left(C, \omega_{\mathrm{C}}, \Omega_{\mathrm{C}}\right)$ be a 3-dimensional Calabi-Yau cone.
	Let $\pi: B \rightarrow \mathrm{C}$ be a crepant resolution with complex volume form $\Omega_B$ extending $\pi^* \Omega_{\mathrm{C}}$.
	Then in every Kähler class there exists a unique AC Kähler Ricci-flat form $\omega_B$ on $B$ with $\frac{1}{6} \omega_B^3=\frac{1}{4} \operatorname{Re} \Omega_B \wedge \operatorname{Im} \Omega_B$.
	Moreover, $\left(B, \omega_B, \Omega_B\right)$ is asymptotic to the Calabi-Yau cone $\mathrm{C}$ with rate $-6$ if the Kähler class $\left[\omega_B\right]$ is compactly supported and with rate $-2$ otherwise.
\end{theorem}

The existence part of the statement was first proved in full generality in \cite[Theorem 3.1]{Goto2009} and the uniqueness in \cite[Theorem 5.1]{Conlon2012}.

\subsection{New infinite diffeomorphism types}

As we see in the subsection \ref{sec:toric}, the condition $c_1(M) \smile [\omega_0] = 0$ is not easy to check.
A simple strategy to make such condition straightforward is to work with manifolds $B$ such that $H^4(B) = 0$.
In the case of resolutions of cones, by theorem 5.2 of \cite{Caibar2005} this is the case if and only if the resolution $\pi : B \to C$ is small.
By \textit{small} resolution we mean one whose exceptional locus $\pi^{-1}(o)$ has codimension at least 2.
Small resolutions are automatically crepant.

The small resolution of the compound $A_p$ Du Val singularity were already considered in \cite{Foscolo2021} to build infinite diffeomorphism types of complete non-compact $G_2$ manifolds.
Apart for the case of $A_1$, we are able to make use of these manifolds also in the $\Spin(7)$ construction.

\begin{theorem}
	Let $B$ be a small resolution of the compound $A_p$ Du Val singularity $X_p \subseteq \C^4$
	$$
		x^2+y^2+z^{p+1}-w^{p+1}=0
	$$
	for $p \in \N^+$.
	
	Then, for any $p \geq 2$, $B$ admits a $T^2$-bundle whose total space carries an asymptotically $T^2$-fibred conical complete $\Spin(7)$-metric.
	Moreover, for $p, p^{\prime} \geq 2$ with $p \neq p^{\prime}$ the $\Spin(7)$-manifolds $M$ and $M^{\prime}$ constructed in this way are not diffeomorphic.
	In particular, there exists infinitely many diffeomorphism types of simply connected complete non-compact $\Spin(7)$-manifolds.
\end{theorem}
\begin{proof}
	Since $b_4(B) = 0$, the cohomological condition \eqref{eq:cohomological_condition} is automatically satisfied.
	Moreover, since $b_2(B) = p \geq 2$ we are able to find two integral linearly independent Chern classes in $H^2(B)$.

	In order to prove that for $p, p^{\prime} \geq 2$ with $p \neq p^{\prime}$ the $\Spin(7)$-manifolds $M$ and $M^{\prime}$ are not diffeomorphic, it is sufficient to see that $M$ and $M^\prime$ have different Betti numbers.
	In order to calculate the Betti numbers of $M$, we consider $M \to B$ as a circle bundle over a circle bundle $M \to P_1 \to B$ as in diagram \eqref{eq:T2_commutative_diagram}.
	Then by using the Betti numbers of $B$ calculated in the proof of theorem 9.3 of \cite{Foscolo2021} and by using the Gysin sequence twice we get that $b_0 = 1$,
	$$
	b_2(M)=p-2, \quad b_3(M)=2p -1
	$$
	and that the other Betti numbers vanish.
\end{proof}

\subsection{Toric $\Spin(7)$ manifolds}\label{sec:toric}

In this and in the next section, we build the first examples of toric $\Spin(7)$ manifolds.
Here we make use of the smooth version of the theorem \ref{thm:main_theorem} to build examples that are $T^2$-bundles over smooth Calabi-Yau 3-folds, whereas in the next section, we will extend the construction outlined in this section to build an infinite class of toric $\Spin(7)$ manifolds as $T^2$-bundles over Calabi-Yau 3-orbifolds.

As anticipated in section \ref{sec:intro}, the idea is to look for Calabi-Yau 3-folds with a $T^2$ action and lift it to a $T^4$ action on the $\Spin(7)$-manifolds produced with theorem \ref{thm:main_theorem} by adding the extra $T^2$ action coming from the principal action on the fibers.
There are quite a few Calabi-Yau 3-folds with $T^2$ actions but the topological condition \eqref{eq:topological_condition} restricts greatly the set of candidates.
For example, the crepant resolutions of the $\Z_p$ quotients of the conifold were used in \cite{Acharya2021} to construct new examples of $G_2$ complete non-compact manifolds with a $T^3$ symmetry, but they cannot be used in the $\Spin(7)$ case, because the topological condition in this latter case is harder to realize than in the $G_2$ case.

For manifolds $B$ with associated Sasaki-Einstein $\Sigma$ such that $H^3_c(B) = 0$ and $\pi_1(\Sigma) = 1$, like in the case of crepant resolutions of cones, then we have to have that $\dim H^2(\Sigma) \geq 2$, which rules out most known example of AC CY toric manifolds, such as for example resolutions of most hyperconifolds, i.e.\ quotients of the conifold by finite groups, including the ones mentioned above.
To see this recall (see e.g.\ section 5 of \cite{Foscolo2021}) that for an AC manifold, by regarding it as a manifold with boundary, we have the long exact sequence
$$
\cdots \rightarrow H^{k-1}(\Sigma) \rightarrow H_c^k(B) \rightarrow H^k(B) \rightarrow H^k(\Sigma) \rightarrow \cdots
$$
which, if $H^3_c(B) = 0$ and if $\Sigma$ is simply connected, produces the short exact sequence
$$
0 \rightarrow H^2_c(B) \rightarrow H^2(B) \rightarrow H^2(\Sigma) \rightarrow 0
.
$$
Now, for an asymptotically conical CY 3-fold, it holds that $H^2_c(B) \iso L^2 \mathcal{H}^2(B)$, where $L^2 \mathcal{H}^2(B)$ is the subset of $\mathcal{H}^2(B)$ given by $L^2$-integrable forms.
By using the Hodge theory isomorphism $H^2(\Sigma) \iso \mathcal{H}^2(\Sigma)$, the above short exact sequence gives a splitting $H^2(B) \iso L^2 \mathcal{H}^2(B) \oplus \mathcal{H}^2(\Sigma)$ which is $L^2$-orthogonal.
By representing $c_1(M)$ as a couple $(\kappa_\eta, \kappa_\theta)$ of closed and coclosed 2-forms decaying with rate $-2$ we see that the topological condition \eqref{eq:topological_condition} can be rewritten as $[{\star_B} \kappa_\eta]_{H^4} = [{\star_B} \kappa_\theta]_{H^4} = 0$, which in turn is saying that $\kappa_\eta$ and $\kappa_\theta$ are $L^2$-orthogonal to $L^2 \mathcal{H}^2(B)$, and thus need to be in $\mathcal{H}^2(\Sigma)$.
If we want them to be linearly independent, we have to require $\dim H^2(\Sigma) \geq 2$.

We were able to think to two examples among canonical bundles of del Pezzo surfaces, although we expect more examples to exist.
Only two out of ten del Pezzo surfaces are suitable because only 4 of them are toric, but $\mathbb{CP}^2$ and its blowup in a point do not have a big enough second cohomology to satisfy the topological condition.

A handy fact about tori that we will use is the following:
\begin{lemma}\label{thm:central_extension_torus}
	A central extension of a torus by a torus is a torus.
\end{lemma}
\begin{proof}
	Let $1 \to T^k \to G \to T^l \to 1$ be a central extension of Lie groups.
	It is clear from the definition of nilpotent group that $G$ is nilpotent.
	Moreover, $G$ being an extension of compact connected groups is itself compact and connected.
	Every compact connected nilpotent group is abelian.
	Indeed, consider the adjoint representation $\mathrm{Ad}_G$ of $G$ on $\mathfrak{g}$.
	Since $G$ is nilpotent, it can be made upper unipotent on a suitable basis of $\mathfrak{g}$, by Lie–Kolchin triangularization theorem.
	Thus, $\mathrm{Ad}_G(G)$ is a compact subgroup of the group $U_1$ of upper unipotent matrices.
	But in $U_1$, every non-identity element generates a closed non-compact subgroup, so the only compact subgroup of $U_1$ is $\{1\}$.
	So $G$ is contained in the kernel of the adjoint representation, so it is abelian.
	
	The thesis follows from the fact that every abelian compact connected Lie group is a torus.
\end{proof}

Whereas the Calabi ansatz gives us a Calabi-Yau metric on the blowup of $\mathbb{CP}^2$ in 3 points, it cannot be applied to the blowup in one or two points.
However, theorem \ref{thm:crepant_resolution} can be applied to the blowup of up to three points and it gives us the freedom to choose the Kähler class.
Thus, this is the approach that we take in the construction of toric $\Spin(7)$-manifolds.

\begin{theorem}
	Let $B$ be the total space of the canonical bundle on the blowup of $\mathbb{CP}^2$ in two or three generic points.
	
	Then $B$ admits a principal $T^2$ bundle such that its total space admits a one parameter family of toric $\Spin(7)$-structures.
	The resulting manifold is $AT^2C$, complete, non-compact and has full $\Spin(7)$ holonomy.
\end{theorem}
\begin{proof}
Let us consider a del Pezzo surface $X$ $\mathrm{Bl}_{S}(\mathbb{CP}^2)$ where $S \subset \mathbb{CP}^2$ is a subset of two or three generic points.
The blow up of $\mathbb{CP}^2$ in up to 3 points is toric.
This action lifts to the total space $B$ of the canonical bundle thus yielding a $T^2 \subset (\mathbb{C}^*)^2$-action on the 3-fold.
Being the canonical bundle on $B$ trivial, we can fix a complex volume form $\Omega$, and by averaging we can make it $T^2$-invariant.

We can make $B$ into a toric AC Calabi-Yau 3-fold as follows.
By blowing down the 0-section of the bundle, we can realize $B$ as a crepant resolution of a Calabi-Yau cone.
Thus, by theorem \ref{thm:crepant_resolution} there is a unique $\omega$ for each Kähler class such that $(\omega, \Omega)$ is an AC Calabi-Yau structure.
Since $T^2$ acts trivially on cohomology, by the uniqueness statement of the theorem, we get that $\omega$ is $T^2$-invariant.
There is an additional $S^1 \subset \mathbb{C}^*$ acting on the fibers of the canonical bundle thus making the 3-fold toric as a symplectic manifold.
However, this last action does not leave $\Omega$ invariant and thus we ignore it.

Let us now suppose $S$ contains 3 points, as if it contained 2, the proof would be analogous.
In order to apply theorem \ref{thm:main_theorem} to $B$, we need to find two linearly independent $a, b \in H^2(B)$ that satisfy equation \eqref{eq:cohomological_condition}.
Since $B$ retracts on $\mathrm{Bl}_{S}(\mathbb{CP}^2)$ the two cohomology rings are isomorphic.
Topologically the blow up is just a connected sum, thus the cohomology of $\mathrm{Bl}_{S}(\mathbb{CP}^2)$ is that of $\mathbb{CP}^2 \# \overline{\mathbb{CP}}^2\# \overline{\mathbb{CP}}^2\# \overline{\mathbb{CP}}^2$.
This has $b_0 = b_4 = 1$ and $b_2 = 4$ and the intersection form for $H^2(B)$ is $\mathrm{diag}(1, -1, -1, -1)$.

Let $E, D_1, D_2, D_3$ be the generators of $H^2(B, \mathbb{Z})$.
We saw that every Kähler class contains a unique AC CY metric, which is preserved by the $T^2$ action.
Consider an integral Kähler class $[\omega]$ on $B$.
Then
$$
[\omega] = e^\omega E + d^\omega_1 D_1 + d^\omega_2 D_2 + d^\omega_3 D_3
$$
where $e^\omega, d^\omega_i \in \mathbb{Z}$.
Then finding $a$ and $b$ is equivalent to finding two linearly independent vectors in $\mathbb{Z}^4$ that are orthogonal to $(e^\omega, d^\omega_1, d^\omega_2, d^\omega_3)$ with respect to the inner product given by $\mathrm{diag}(1, -1, -1, -1)$.
By basic theory of Diophantine equations, we know that there are three linearly independent integral solutions.


It remains to prove that the action lifts to a $T^4$ action on the $T^2$-bundle $M \to B$.
By running the proof of theorem \ref{thm:final_result} by using the implicit function theorem on Banach spaces of data that is invariant under the $T^2$-action on $B$, we get that the tuple $(\omega, \Omega, p, q, r, d\eta, d\theta)$ is $T^2$-invariant on $B$ (it is also crucial that $d\eta_1$ and $d\theta_1$ can be chosen to be $T^2$ invariant for the $T^2$ action on the base, which was proved in theorem \ref{thm:HYM_connection}).

Let $\tilde{H}$ be the set of all the $\tilde{g} \in \Aut(M, (\eta, \theta))$ (where the last notation means that it preserves the connection $(\eta, \theta)$) such that $\tilde{g}$ covers $g$ for some $g \in T^2$.
Define $q : \tilde{H} \to T^2$ to be the map $q(\tilde{g}) = g$ where $\tilde{g}$ covers $g$.
This map is surjective.
Indeed, let $g \in T^2$ and consider the pullback bundle with pullback connection $\left(g^*M, g^*(\eta, \theta)\right)$, and call $\overline{g} : g^*M \to M$ the canonical pullback map.
Since, as explained above, $g$ preserves $(d\eta, d\theta)$ and $B$ is simply connected, the $T^2$ bundle $g^*M$ is isomorphic as a principal bundle to $M$ by Chern-Weil theory.
Now, connections on a torus bundle modulo automorphisms are a homogeneous space over the moduli space of flat connections.
However, it is well known that the moduli space of flat connections on a circle bundle over a simply connected manifold is trivial, and the result clearly extends to torus bundles.
Thus, up to bundle automorphism, there is only one connection on $B$ with prescribed curvature, which implies that there is an isomorphism of principal bundles with connections between $\left(g^*M, g^*(\eta, \theta)\right)$.
By composing $\overline{g} : g^*M \to M$, the isomorphism $M \to g^*M$ and the bundle automorphism $M \to M$ just found we get a bundle automorphism $\tilde{g} : M \to M$ covering $g$ that preserves $(\eta, \theta)$ as desired.

Let us now study the kernel of $q$, which by definition is given by automorphisms of $M$ covering the identity (i.e.\ gauge transformation) that preserve $(\eta, \theta)$.
Since gauge transformation on a $T^2$-principal bundle act on connections as $g^*(\eta, \theta) = (\eta, \theta) + \sigma_g^*\mu_{T^2}$ where $\sigma : \Aut(M) \to C^\infty(B; T^2)$ and $\mu_{T^2}$ is the Mauer-Cartan form on $T^2$, those that preserve $(\eta, \theta)$ are such that $\sigma_g^*\mu_{T^2} = 0$.
By making $\mu_{T^2}$ explicit, we see that the kernel of $q$ is given by the pullback of constant functions in $C^\infty(B; T^2)$.

Thus, we have a central extension of Lie groups
$$
	1 \to T^2 \to \tilde{H} \xrightarrow{q} T^2 \to 1
$$
and by lemma \ref{thm:central_extension_torus} we get $\tilde{H} = T^4$.

Finally, elements of $\tilde{H}$ preserve the pullback of any $T^2$-invariant form on $B$ (and thus $\omega, \Omega, p, q$ and $r$) because, e.g., $\tilde{g}^*p_{P_1}^* \omega = (p_{P_1} \comp \tilde{g})^* \omega = (g \comp p_{P_1})^* \omega$ $= p_{P_1}^* g^* \omega $ $=p_{P_1}^* \omega$.
Moreover, by definition of $\tilde{H}$ all of its elements preserve $\eta$ and $\theta$.
Because $\Phi$ is expressed in terms of this data, this action preserves $\Phi$.

It remains to prove that the action is multi-Hamiltonian, but this follows directly from theorem 3.10 of \cite{Madsen2011}.
\end{proof}

\subsection{Infinite class of toric $\Spin(7)$ manifolds using the orbifold construction}\label{sec:toric_orbifolds}

As anticipated, in this section, we explain how one could use theorem \ref{thm:main_theorem_orbifold} to extend the construction of the previous section \ref{sec:toric} to build an infinite class of toric $\Spin(7)$-manifolds, whose $T^2$ quotients exhibit different singularity behaviours.
This only holds under the nontrivial hypothesis that the results of \cite{Coevering2009} and \cite{Futaki2009}, together with theorem \ref{thm:crepant_resolution} hold in the orbifold case.
As anticipated in the introduction, we do believe this to be true, but we do not have a complete proof of this.
We explain below why we believe these examples to exist.

We restrict our attention to the canonical bundle on the blowup $X_k$ of the weighted projective space $\mathbb{CP}^2_{[1, 1, k]}$ for $k \in \N_{\geq 2}$ in the two fixed smooth points, because we are interested in describing an infinite class of toric examples with minimal effort, but we expect that many more examples can be built with more knowledge of toric geometry.
This space is an orbifold with a conic point singularity in $[0\colon\!0\colon\!1]$.

As in the case of the blowup on $\mathbb{CP}^2$, given that the canonical bundle retracts on $X_k$, we can do all the calculations on $X_k$ which is more convenient.
From toric geometry, we know that $b_2(X_k) = 3$ and that $H^2(X_k)$ is generated by the pullbacks of the generators of $H^2(\mathbb{CP}^2)_{[1, 1, k]}$ and the exceptional divisors $D_1$ and $D_2$.
The generator of $H^2(\mathbb{CP}^2)_{[1, 1, k]}$ is the line $y=z=0$ (which is equivalent to the line $x=z=0$ but not equivalent to $x=y=0$) and we denote it by $E$.
Given that we are dealing with a singular variety, the intersection form is rational, and, in the ordered basis $(E, D_1, D_2)$, it is given by the matrix $\mathrm{diag}\left(\tfrac{1}{k}, -1, -1\right)$.

Now, on $\mathbb{CP}^2_{[1, 1, k]}$, $H^2$ is generated by $E$, with intersection $E\cdot E$ equal to $1/k$, and the smooth circle bundles are those with Chern class $c_1 = eE$ for $e \in \mathbb{Z}$, with $e$ coprime with $k$.
To see this, we claim that the total space of the line bundle corresponding to $mE$ is given by the quotient
$$
	\faktor{\left(\left(\C^3 \setminus \{0\}\right) \times \C\right)}{\C^*} \qquad \lambda \cdot (x, y, z, w) = \left(\lambda x, \lambda y, \lambda^k z, \lambda^m w\right) \quad \lambda \in \C^*
	.
$$
The circle bundle associated to such linear bundle is smooth if and only if the line bundle is smooth away from the zero section.
The complement of the zero section in the total space of the line bundle is a dense open set $U$ of $\mathbb{CP}^3_{[1, 1, k, m]}$.
If $m$ is coprime with $k$, then the only singular points of $\mathbb{CP}^3_{[1, 1, k, m]}$ are $[0\colon\!0\colon\!1\colon\!0]$ and $[0\colon\!0\colon\!0\colon\!1]$, which are both not included in $U$.
On the other hand, if $m$ is not coprime with $k$, then the singular locus is one-dimensional and it intersects with $U$.
In order to learn what are the circle bundles on $X_k$, we note that the circle bundle with Chern class $eE + d_1 D_1 + d_2 D_2$ is obtained by pasting the circle bundle on $\mathbb{CP}^2_{[1, 1, k]}$ with Chern class $eE$ with the circle bundle on $\mathbb{CP}^2$ with Chern class $d_1$ and the circle bundle on $\mathbb{CP}^2$ with Chern class $d_2$, where the two copies of $\mathbb{CP}^2$ mentioned are the ones that we are pasting to realize the blowup.
Given that the circle bundles are trivial on small enough neighborhoods near the points on which we plan to connect them, we can paste them smoothly.
Hence, the smooth circle bundles on $X_k$ are those with Chern class $eE + d_1 D_1 + d_2 D_2$ for $e, d_1, d_2 \in \mathbb{Z}$ with $e$ coprime with $k$, given that the circle bundles on $\mathbb{CP}^2$ are always smooth.
Given a Kähler class $[\omega] = e^\omega E + d_1^\omega D_1 + d_2^\omega D_2$, we need to find two linearly independent Chern classes $a$ and $b$ that satisfy the topological condition \eqref{eq:topological_condition_orbifold}, i.e.\ such that the integers $e^a, d_1^a, d_2^a, e^b, d_1^b, d_2^b$ determining them satisfy
$$
	e^\omega e^a - k d_1^\omega d_1^a - kd_2^\omega d_2^a = 0, \quad e^\omega e^b - k d_1^\omega d_1^b - k d_2^\omega d_2^b = 0
	.
$$
The presence of $k$ does not alter the solvability of the system, and we can find two linearly independent solutions.

However, we need to make sure that the $T^2$-bundle be Seifert, i.e.\ that $e^a$ and $e^b$ are coprime with $k$.
To do this, we need to have freedom of choice on the Kähler class, because from the above equations, we see that any factor of $k$ that does not divide $e^\omega$ needs to divide $e^a$ and $e^b$.
We have this freedom if theorem \ref{thm:crepant_resolution} holds for orbifolds and the canonical bundle over $\mathbb{CP}^2_{[1,1,k]}$ satisfies the hypotheses of that theorem, i.e.\ it is a crepant resolution of a Calabi–Yau cone.
In \cite{Foscolo2019} it was already argued how the theorem holds if singularities are contained in a compact set, and we expect that the proof can be extended to the case where singularities extend all the way up to infinity, by using the H\"older spaces of invariant sections defined above.
There are two proofs of the theorem, one by Goto in \cite{Goto2009} and one by van Coevering in \cite{Coevering2009a}.
The proof by Goto eventually boils down to proving two things: one is an application of the implicit function theorem and the other is a $C^{k+2, \alpha}_\nu$ estimate for the Monge-Ampere equation.
The application of the implicit function theorem requires proving that the linearization of a certain differential operator is invertible, and this is proved by studying the Laplacian on decaying section on conical manifolds.
Both results should extend to orbifolds, for essentially the same reasons why our theorem \ref{thm:main_theorem_orbifold} extends to orbifolds outlined in section \ref{sec:orbifolds}.

On the other hand, similarly to the smooth case, $\mathbb{CP}^2_{[1,1,k]}$ is the crepant resolution of its blow-down in the zero section, which we need to prove to admit a Calabi–Yau cone structure.
This is equivalent to prove that its link have a Sasaki-Einstein 5-orbifold structure.
Clearly, the link of such cone is the circle bundle corresponding to the canonical class on $\mathbb{CP}^2_{[1,1,k]}$, which we believe to possess a Sasaki-Einstein structure because of theorem 1.2 of \cite{Futaki2009} in the version of theorem 6.6 of \cite{Coevering2009}, in a hypothetical orbifold version.
Theorem 1.2 of \cite{Futaki2009} should extend to orbifolds, because it is the Sasakian (or, in other words, odd-dimensional) version of a result of Wang and Zhu (in \cite{Wang2004}) in the Kählerian case.
This result was extended to the orbifold case by Shi and Zhu in \cite{Shi2011} and, while we did not go through their proof, we expect the odd-dimensional case to be analogous.
Theorem 6.6 of \cite{Coevering2009} is just a rewriting of theorem 1.2 of \cite{Futaki2009} in the combinatorics picture of toric geometry, and we expect it to hold in the orbifold case as well, given that the Delzant construction can be extended to orbifolds, which was done in \cite{Lerman1997}.
In the orbifold case we have labeled polytopes instead of the classic polytopes of the smooth case.
By labeled polytopes we mean polytopes with an integer attached to each facet and these integers are responsible for orbifold singularities.
We believe that a label for the $k$th facet being equal to $m$ should correspond to taking $m$ copies of $u_k$ in the combinatorial picture of proposition 6.5 of \cite{Coevering2009}, which would imply that $\mathbb{CP}^2_{[1,1,k]}$ has a Sasaki-Einstein structure on its canonical circle bundle.
However, many details need to be checked and we leave this as an open problem.

Going back to our construction, we can thus substitute $e^\omega$ with $k \tilde{e}^\omega$ in the above equations and simplify the $k$ away.
Moreover, if we found any solution of the above equations where any factor of $k$ divides $e^a$ or $e^b$, we could reabsorb it in $e^\omega$.

The rest of the construction is analogous to the one in the previous section \ref{sec:toric}.

Note that the canonical divisor is $K_{X_k} = -(k+2) E + D_1 + D_2$ which, by what we saw before, is smooth only if $k$ is odd.
Thus, we see that there exist both examples of toric $\Spin(7)$ manifolds whose $T^2$ quotient is an AC CY 3-orbifold with singularities contained in a compact set and examples where the same happens but the singularities extend all the way up to infinity.

\section{The \texorpdfstring{$G_2$}{G2} case}\label{sec:G2_case}
As mentioned in the introduction, the result proved in this paper is strongly influenced by an analogous result proved by Foscolo, Haskins and Nordström in \cite{Foscolo2021} which swaps the group $\Spin(7)$ with $G_2$ and $T^2$-bundles with $S^1$-bundles.
In particular, the proof outlined in section \ref{sec:analytic_curve} is in great part analogous to the one in \cite{Foscolo2021}.
However, the proof presented in this paper makes use of the implicit function theorem, whereas the proof in the $G_2$ case was done by writing down a candidate series expansion of the solution, solving the equations term by term and then proving the convergence of the series.
The advantage of using the implicit function theorem, besides being more natural and aesthetically pleasing, is that it allows to skip the last step of proving the convergence manually.
As explained in section \ref{sec:analytic_curve}, the main ingredient that allows us to use the implicit function theorem is to solve a slightly modified version of the equation $d\Phi = 0$ that allows us to introduce an extra free parameter, as explained in remark \ref{rmk:free_parameter}.
There is nothing specific about the $\Spin(7)$ case and an analogous remark holds in the $G_2$ case.
In this section we briefly revise the proof in the $G_2$ case by making use of this extra trick, so to provide a more concise proof.

In the $G_2$ case, given a circle bundle $M \to B$, an $S^1$-invariant $G_2$ structure $\varphi$ can be expressed as
$$
\varphi = \theta \wedge \omega + p^3 \Re \Omega
$$
for some principal connection $\theta$ on $\pi$, an $\SU(3)$-structure $(\omega, \Re \Omega)$ on $B$ (that we identify with its pullback through $\pi$) and some positive smooth function $p$ on $B$.
In this case, the equations $d \varphi = d {\star} \varphi = 0$ are equivalent to the system
\begin{subnumcases}{\label{eq:G2_torsion}}
	d \omega = 0 \\
	d\left(p^{3} \operatorname{Re} \Omega\right)=-d \theta \wedge \omega \\
	d\left(p \operatorname{Im} \Omega\right) =0 \\
	2 p^3 d p \wedge \omega^2 = d \theta \wedge p\operatorname{Im} \Omega
	.
\end{subnumcases}

\subsection{Free parameters and gauge fixing conditions}

In this subsection and onward we want to solve system \eqref{eq:G2_torsion} and thus we suppose that $B$ be asymptotically conical.

As we discussed above, in order to apply the implicit function theorem, we solve the modified equations $d \varphi = 0$ and $d {\star} \varphi = p^*_M \star_\omega ds$ for some positive smooth decaying function $s$, analogously to what we did in remark \ref{rmk:free_parameter}.
We add this and other free parameters in the following lemma.
\begin{proposition}\label{thm:G2_free_parameters}
	Let $\left(\omega_0, \Omega_0\right)$ be an AC Calabi-Yau structure on a 6 -manifold $B$ and denote by $g_0$ and $\nabla_0$ the induced metric and Levi-Civita connection.
	Fix $k \geq \N^+, \alpha \in(0,1)$ and $\nu\in(-\infty,-1)$.
	Then there exists a constant $\varepsilon_0>0$ such that the following holds.
	Let $(\omega, \Omega)$ be a second $\mathrm{SU}(3)$-structure on $B$ whose distance from $\left(\omega_0, \Omega_0\right)$ in norm $C^1_{-1}(B)$ is less than $\varepsilon_0$, i.e.\ such that $\norm{\omega-\omega_0}_{C^1_{-1}} + \norm{\Omega-\Omega_0}_{C^1_{-1}} < \varepsilon_0$.
	Suppose that there exists a scalar function $p$ and one integral exact 2 -forms $d \theta$ on $B$ such that
	\begin{equation}
		d \omega = 0 \qquad d \theta \wedge \omega^2=0\label{eq:G2_parameters_linear_conditions}
		.
	\end{equation}
	Moreover, assume the existence of functions $s, u, v$ and a vector field $X$ in $C_{\nu+1}^{k+1, \alpha}$ such that
	\begin{subnumcases}{\label{eq:G2_parameters}}
		d\left(p^{3} \operatorname{Re} \Omega\right)+d \theta \wedge \omega = d {\star} d(X\lrcorner \operatorname{Re} \Omega+v \omega)\label{eq:G2_parameters_dRe_Omega} \\
		d\left(p \operatorname{Im} \Omega\right) = d {\star} d(u \omega)\label{eq:G2_parameters_dIm_Omega} \\
		\left(2 p^3 d p + Jds\right) \wedge \omega^2 - d \theta \wedge p\operatorname{Im} \Omega = 0\label{eq:G2_parameters_dtheta}
	\end{subnumcases}
	where the Hodge $\star$ is computed with respect to the metric induced by $(\omega, \Omega)$.
	Then $u = v = 0, s = 0$ and $X = 0$, i.e.\ $(\omega, \Re \Omega, p, \theta)$ is a solution system \eqref{eq:G2_torsion}.
\end{proposition}
\begin{proof}
	Analogous to \ref{thm:free_parameters}.
\end{proof}
Moreover, we impose the same gauge fixing conditions introduced in section \ref{sec:gauge}, with the obvious adaptation that now we only have one connection 1-form instead of two.

\subsection{The 0\textsuperscript{th} and 1\textsuperscript{st} order case}

As in the $\Spin(7)$ case, we imagine expanding the data appearing in system \eqref{eq:G2_torsion} as a series in $\varepsilon$ and we study its possible solutions.
As above, we manually impose $u_0 = v_0 = s_0 = 0$ and $X_0 = 0$ and $\theta_0 = 0$ because, in the limit where the fibers shrink to points, the connection remains undetermined and the free parameters loose meaning.
Thus,
\begin{subnumcases}{}
	d \omega_0 = 0 \\
	d\left(p_0^{3} \operatorname{Re} \Omega_0\right)= 0 \\
	d\left(p_0 \operatorname{Im} \Omega_0\right) = 0 \\
	2 p_0^3 d p_0 \wedge \omega^2_0 = 0
	.
\end{subnumcases}
From the last equation we get that $p_0$ is constant and thus, from the ones involving $\Re \Omega_0$ and $\Im \Omega_0$ we get that $(\omega_0, \Omega_0)$ has to be Calabi Yau.

We thus study the linearization of the equations at a point satisfying the conditions we just found, which is given by the following system:
\begin{subnumcases}{\label{eq:G2_linearization}}
	3p_0^2dP \wedge \operatorname{Re} \Omega_0 + p_0^3 d \rho +d \theta_1 \wedge \omega_0 = d {\star} d(X_1\lrcorner \operatorname{Re} \Omega_0+v_1 \omega_0) \\
	dP \wedge \Im \Omega_0 + p_0 d \hat{\rho} = d {\star} d(u_1 \omega_0) \\
	\left(2p_0^3 dP  + J_0ds_1\right) \wedge \omega^2_0 - d \theta_1 \wedge p_0\operatorname{Im} \Omega_0 = 0
\end{subnumcases}
where we called $P = p_1$ and $\rho = \Re \Omega_1$.
Recall that we are not expanding $\omega$ because we fixed it to $\omega_0$ through gauge transformations.
We have the following
\begin{lemma}
	Let $\theta_1 \in \mathcal{A}(M)$ and let $(d\theta_1, \rho, P, s_1, u_1, v_1, X_1) \in C^{k, \alpha}_\nu$ be a solution of system \eqref{eq:G2_linearization}, for $k \in \N$, $\alpha \in [0,1]$, $\nu \in (-\infty, 0)$, $p_0$ is constant, and $(\omega_0, \Omega_0)$ Calabi-Yau.
	
	Then $P, u_1, v_1, X_1 = 0$, $ \theta_1$ is a Hermitian Yang-Mills connection and $\rho \in \Omega^3_{12}(B)$.
	Moreover, this data satisfies
	\begin{subnumcases}{\label{eq:G2_rho_linearization}}
		d\rho = -p_0^{-3}d \theta_1 \wedge \omega_0 \\
		d \star\rho = 0
		.
	\end{subnumcases}
\end{lemma}
\begin{proof}
	Analogous to \ref{thm:first_order_necessary}
\end{proof}

For the same reasons explained in section \ref{sec:first_order} for the $\Spin(7)$ case, we solve the first step manually in the following lemma:
\begin{lemma}\label{thm:G2_HYM_connection}
	Let $M \to B$ be a principal $S^1$ bundle on an AC Calabi Yau manifold $B$.
	
	Then, for any principal connection $\bar{\theta}_\infty$ on $M$ such that $d\bar{\theta}_\infty \in C^\infty_{-2}(\Lambda^2(B))$, there exists a unique HYM connection $\theta_1$ with $\theta_1 - \bar{\theta}_\infty \in C^\infty_{-1}\left(\Lambda^1(B)^2\right)$ such that $d^* \left(\theta_1-\bar{\theta}_\infty \right) = 0$.
	
	Moreover, given $\theta_1$ as above, there exists a unique solution $\rho \in C_{-1}^{\infty}\left(\Lambda_{12}^3(B)\right) \cap \mathcal{W}_\nu^3(B)$ of \eqref{eq:G2_rho_linearization}.
\end{lemma}
\begin{proof}
	Analogous to \ref{thm:HYM_connection} and \ref{thm:rho_first_order}.
\end{proof}

\subsection{The implicit function theorem}

Now, fix some connection $\theta_\infty$ on $M$, $\nu \in (-2, -1)$, $\alpha \in (0,1)$ and $k \in \N$.
Let $Y = (0, \Re \Omega_0, p_0$, $ 0, 0, 0, 0) + C^{l, \alpha}_\nu\left(\Lambda^1(B) \times \Lambda^3(B) \times \Lambda^0(B)^4\right) \times C^{l + 1, \alpha}_{\nu + 1} \left(TB\right)$ be cut out by the following linear constraints:
$$
	d \theta \wedge \omega_0^2 = 0 \quad
	d^* (\theta - \theta_\infty) = 0 \quad
	\Re \Omega \wedge \omega_0 = 0 \quad
	\Re \Omega \wedge \Re \Omega_0 = 0 \quad
	\pi_{12}(\Re \Omega - \Re \Omega_0) \in \mathcal{W}^3_\nu(B)
	.
$$
We aim to apply the implicit function theorem to the function $\Psi: \R \times Y \to Z = \{z \in C^{k, \alpha}_\nu(\Lambda^6(B)) \times C^{k-1, \alpha}_{\nu - 1}(\Lambda^4(B)^2 \times \Lambda^5(B)) \mid z_1, z_2 \text{ exact}\}$ given by
$$
\Psi: \begin{bmatrix}
	\varepsilon \\
	\delta_\theta \\
	\Re \tilde{\Omega} \\
	p \\
	s \\
	u \\
	v \\
	X
\end{bmatrix}
\mapsto
\begin{bmatrix}
	\frac{1}{4}\Re \tilde{\Omega} \wedge \Im \tilde{\Omega} - \frac{1}{6} \omega^3_0 \\
	d\left(p^{3} \operatorname{Re} \tilde{\Omega} + \varepsilon(p^3 - p_0^3) \Re \Omega_1\right) + d \delta_\theta \wedge \omega_0 - d {\star} d(X\lrcorner \operatorname{Re} \Omega +v \omega_0) \\
	d\left(p \operatorname{Im} \tilde{\Omega} + \varepsilon (p - p_0) \Im \Omega_1 \right) - d {\star} d(u \omega_0) \\
	(2 p^3 d p + J ds) \wedge \omega_0^2 - d \left(\varepsilon\theta_1 + \delta_\theta \right) \wedge p \operatorname{Im} \Omega
\end{bmatrix}
$$
where, like in the $\Spin(7)$ case, $\Omega = \tilde{\Omega} + \varepsilon \Omega_1$.
$\Psi_1$ and $\Psi_2$ are manifestly exact and are obtained by subtracting the solutions to system \eqref{eq:G2_rho_linearization} from system \eqref{eq:G2_parameters}.
They take value in the spaces claimed because $\Re \Omega_1, \Im \Omega_1 \in C^{\infty}_{-1}$ are multiplied by decaying sections.
To apply the implicit function theorem, we calculate the linearization for $\varepsilon = 0$ and $y_0 =\left(0, \Re \Omega_0, p_0, 0, 0, 0, 0\right)$, which is
\begin{equation*}
	D \Psi_{(0, y_0)} : \begin{bmatrix}
		\gamma_\theta \\ \rho \\ P \\ S \\ U \\ V \\ X_1
	\end{bmatrix} \mapsto
	\begin{bmatrix}
		\operatorname{Re} \Omega_0 \wedge \hat{\rho}+\rho \wedge \operatorname{Im} \Omega_0 \\[4pt]
		p_0^{3}d\rho + dP \wedge \left(3p_0^{2} \operatorname{Re} \Omega_0\right) +d\gamma_\theta \wedge \omega_0 - d {\star_0} d(X_1\lrcorner \operatorname{Re} \Omega+V \omega_0) \\[4pt]
		p_0d\hat{\rho} + dP \wedge\operatorname{Im} \Omega_0 - d {\star_0} d(U \omega_0) \\[4pt]
		\left(2 p_0^3 d P + J dS\right) \wedge \omega_0^2 - d \gamma_\theta \wedge \operatorname{Im} p_0 \Omega_0
	\end{bmatrix}
	.
\end{equation*}
The following theorem allows to apply the implicit function theorem.
\begin{theorem}\label{thm:G2_linearization}
	Let $\left(B, \omega_0, \Omega_0\right)$ be an AC Calabi-Yau 3-fold.
	Fix $k \in \N^+, \alpha \in(0,1), \delta \in (0, + \infty)$ and $\nu \in(-3-\delta,-1)$ away from a discrete set of indicial roots.
	
	Then the map $y \mapsto D\Psi_{x_0, y_0}(0, y)$ is an isomorphism of Banach spaces.
\end{theorem}
\begin{proof}
	Analogous to \ref{thm:linearization}.
\end{proof}
\noindent Finally, we get the following result, which is theorem 1 in \cite{Foscolo2021}.
\begin{theorem}
	Let $\left(B, g_0, \omega_0, \Omega_0\right)$ be an AC Calabi-Yau 3-fold asymptotic with rate $\mu<0$ to the Calabi-Yau cone $\left(\mathrm{C}(\Sigma), g_{\mathrm{C}}, \omega_{\mathrm{C}}, \Omega_{\mathrm{C}}\right)$ over a smooth Sasaki-Einstein 5-manifold $\Sigma$.
	Let $M \rightarrow B$ be a principal $S^1$-bundle such that
	$$
		c_1(M) \neq 0 \qquad c_1(M) \smile \left[\omega_0\right]=0 \in H^4(B)
		.
	$$
	
	Then there exists $\varepsilon_0 \in \R^+$ and an analytic curve of $S^1$-invariant torsion-free $G_2$ structures $(0, \varepsilon_0) \to \Omega^3(M)$ such that, by denoting with $g_\varepsilon$ the Riemannian metric on $M$ induced by $\varphi_\varepsilon$,
	\begin{enumerate}
		\item $\operatorname{Hol}\left(g_\varepsilon\right)=G_2$
		\item $\left(M, g_\varepsilon\right)$ is an $AT^1C$ (usually called ALC) with rate $\max(\mu, -1)$.
		\item $\left(M, g_\varepsilon\right)$ collapses with bounded curvature to $\left(B, g_0\right)$ as $\varepsilon \rightarrow 0$.		
	\end{enumerate}
\end{theorem}
\begin{proof}
	Analogous to \ref{thm:main_theorem}.
\end{proof}

\bibliographystyle{alpha}
\bibliography{Biblio}

\end{document}